\documentclass[11pt]{amsart}
\usepackage[backref,colorlinks]{hyperref}
\usepackage{amssymb,bm}
\usepackage[foot]{amsaddr}
\usepackage[abbrev,msc-links]{amsrefs}
\usepackage[utf8]{inputenc}
\usepackage[shortlabels]{enumitem}
\usepackage{graphicx}
\usepackage{mathtools}
\usepackage{lineno}
\usepackage[framemethod=TikZ]{mdframed}
\usepackage{cases}

\theoremstyle{plain}
\newtheorem{thm}{Theorem}[section]
\newtheorem{fact}[thm]{Fact}
\newtheorem{prop}[thm]{Proposition}
\newtheorem{clm}[thm]{Claim}
\newtheorem{cor}[thm]{Corollary}
\newtheorem{lem}[thm]{Lemma}

\newtheorem*{clm*}{Claim}

\theoremstyle{definition}
\newtheorem{rem}[thm]{Remark}
\newtheorem{dfn}[thm]{Definition}
\newtheorem{exmp}[thm]{Example}

\newtheorem{conj}[thm]{Conjecture}
\newtheorem{obs}[thm]{Observation}

\newmdtheoremenv[skipabove=2mm,skipbelow=4mm]{const}[thm]{Construction}

\numberwithin{equation}{section}

\normalsize
\DeclareMathOperator{\ex}{ex}

\DeclareMathOperator{\dist}{dist}

\DeclarePairedDelimiter\ceil{\lceil}{\rceil}
\DeclarePairedDelimiter\floor{\lfloor}{\rfloor}

\DeclarePairedDelimiter{\final}{\langle}{\rangle}

\newcommand\NN{\mathbb N}
\newcommand\ZZ{\mathbb Z}

\newcommand\RR{\mathbb R}
\newcommand\eps{\varepsilon}

\newcommand\cH{\mathcal{H}}

\newcommand\cF{\mathcal{F}}
\newcommand\cL{\mathcal{L}}
\newcommand\cE{\mathcal{E}}

\newcommand\cR{\mathcal{R}}
\newcommand\cC{\mathcal{C}}

\newcommand\EE{\mathbb E}
\newcommand\PP{\mathbb P}
\let\phi=\varphi
\let\emptyset=\varnothing

\newenvironment{claimproof}[1]{\par\noindent\underline{Proof of Claim:}\space#1}{\hfill $\blacksquare$}

\usepackage{euflag}

\title[Slow graph bootstrap percolation  III]{Slow graph bootstrap percolation  III:  \\Chain constructions}
\author{David Fabian $^{1,\ast}$}
\author{Patrick Morris $^{2,\dagger}$}
\author{Tibor Szab\'o $^{3,\ddagger}$}

\address{$^2$ Departament de Matem\`atiques, Universitat Polit\`ecnica de Catalunya (UPC), Barcelona, Spain.}
\address{$^3$ Institute of Mathematics, Freie Universit\"at Berlin, Germany}

\thanks{$^\ast$ Research supported by the Deutsche Forschungsgemeinschaft (DFG)
	Graduiertenkolleg “Facets of Complexity” (GRK 2434). }

\thanks{$^\dagger$ Research supported by   the DFG Walter Benjamin program - project number 504502205, and by   the European Union's Horizon Europe   Marie Sk{\l}odowska-Curie grant RAND-COMB-DESIGN - project number
101106032 {\euflag}.}

\thanks{$^\ddagger$ Research
	supported by the DFG under Germany’s Excellence Strategy - The Berlin Mathematics Research Center
	MATH+ (EXC-2046/1, project ID: 390685689).}

\email{david.fabian@uni-ulm.de, pmorrismaths@gmail.com,  szabo@math.fu-berlin.de}

\date{\today}

\begin{document}

	\begin{abstract}
For graphs $H$, we study the extremal function $M_H(n)$ which is the \textit{maximum running time} (until stabilisation)  of an $H$-bootstrap percolation process  on $n$ vertices. 
Building on previous work in the clique case $H=K_k$, we develop a general framework of \textit{chain constructions}. We demonstrate the flexibility of this framework by applying several variations of the method to 
give lower bounds on $M_H(n)$ for a wide variety of different graphs $H$ including dense graphs, random graphs and complete  bipartite graphs. In particular, we focus on the question of whether  $M_H(n)$ is (almost) quadratic or not and our lower bounds develop  connections with additive combinatorics, utilising constructions of sets free of solutions to certain linear equations. Finally, our lower bounds are complemented by upper bounds which connect $M_H(n)$ to other problems in extremal graph theory such as the Ruzsa-Szemer\'edi (6,3)-Theorem. 
  \end{abstract}
\maketitle
\section{Introduction} \label{sec:intro}
Given graphs $H$ and $G$, let $n_H(G)$ denote the number of copies of $H$ in $G$. The $H$\textit{-bootstrap percolation process}
($H$\emph{-process} for short)
on a graph $G$ is the sequence $(G_i)_{i\geq 0}$ of graphs defined by $G_0 := G$ and
  \begin{linenomath}
    \begin{equation*}
    V(G_i) := V(G), \qquad \quad
    E(G_i) := E(G_{i-1}) \cup \left\{e\in\binom {V(G)} 2 : n_H\left(G_{i-1}+e\right)>n_H(G_{i-1})\right\},
              \end{equation*} \end{linenomath}
for $i\geq 1$.
We call $G$ the \emph{starting graph} of the process and $\tau_H(G):=\min\{t\in \NN: G_t=G_{t+1}\}$ the \emph{running time} of the $H$-process on $G$, which is the point at which the process stabilises. Finally, we define $G_\tau$ with $\tau=\tau_H(G)$ to be the \emph{final graph} of the process, and denote it $\final{G}_H:=G_\tau$. 

The $H$-process is an example of a cellular automaton and has been studied in various contexts \cite{balogh2012linear,balogh2012graph,bartha2024weakly}. In particular, over 50 years ago Bollob\'as \cite{bollobas1968weakly} introduced the model through the notion   of \emph{weak saturation}, asking for the smallest number of edges an $n$-vertex graph $G$ can have such that $\final{G}_H=K_n$. Here our focus will be on a more recent  extremal question, also raised by Bollob\'as, which asks for the maximum running time of an $H$-process on $n$ vertices. We define 
\[M_H(n):=\max\{\tau_H(G):v(G)=n\}\]
and we will be primarily interested in the asymptotics of $M_H(n)$. Note that trivially, one always has $M_H(n)=O(n^2)$ as the $H$-process on $G$ has to add a new edge at each time step $1\leq t\leq \tau_H(G)$.

\subsection{Cliques} Initial exploration of $M_H(n)$ focused on cliques $H=K_k$. Bollob\'as, Przykucki, Riordan and Sahasrabudhe \cite{bollobas2017maximum} and independently Matzke \cite{matzke2015saturation}, showed that $M_{K_4}(n)=n-3$,  in contrast to $K_3$ for which $M_{K_3}(n)=\ceil{\log (n-1)}$.  Both sets of authors  also gave lower bounds for $M_{K_k}(n)$ for larger $k$, with it being shown in \cite{bollobas2017maximum}  that there are constants $\lambda_k>0$ which tend to $0$ as $k$ tends to infinity, such that $M_{K_k}(n)\geq n^{2-\lambda_k}$. However, the authors believed that the $K_k$-process would always fall short of lasting as long as the trivial upper bound, conjecturing that $M_{K_k}(n)=o(n^2)$ for all $k\in \NN$. Somewhat surprisingly, this was then refuted by Balogh, Kronenberg, Pokrovskiy and Szab\'o \cite{balogh2019maximum} who showed that $M_{K_k}(n)=\Omega(n^2)$ for all $k\geq 6$. This leaves open the question of whether the conjecture holds for $K_5$, which remains the most intriguing outstanding question in the field. Although they did not show that $M_{K_5}(n)$ is quadratic, Balogh, Kronenberg, Pokrovskiy and Szab\'o \cite{balogh2019maximum} did show that $M_{K_5}(n)\geq n^{2-O(1/\sqrt{\log n})}= n^{2-o(1)}$ by building a connection with subsets of the integers that are free of arithmetic progressions of size 3 and using a famous construction of Behrend \cite{behrend1946sets}. 

\subsection{Graphs with fast running time} In this series of papers on the running time of graph percolation processes, we study $M_H(n)$ for different graphs $H$. In the first paper \cite{FMSz1}, we addressed cycles showing the $M_{C_k}(n)$ is logarithmic for all $3\leq k\in \NN$. In fact, we determined $M_{C_k}(n)$ exactly for large $n$, with the expression depending on the parity of $k$. The second paper \cite{FMSz2} explored properties of $H$ which lead to a maximum running time that is (sub-)linear. We showed for example that all trees $T$ have constant running time $M_T(n)=O(1)$ and that $M_{K_{2,s}}(n)$ is linear in $n$ for all   $3\leq s\in \NN$. We also proved general results exploring the relationship between  $M_H(n)$ with small values of certain graph parameters of $H$ such as connectivity. We refer to the paper \cite{FMSz2} for the details and will return briefly to discuss some of the results in Section \ref{sec: intro insep}. 
Suffice it to say here that our findings in the second paper \cite{FMSz2} suggest  that having a (sub-)linear running time $M_H(n)$  requires $H$ to have atypical and sparse structural properties, as with the examples listed above. 

\subsection{Graphs with (almost) quadratic running time} In this current paper, we continue our exploration of $M_H(n)$ and its dependence on properties of $H$. One of our main findings is that cliques of size at least $5$ are far from  unique in having maximum running time that is (almost) quadratic. In fact, graphs $H$ with $M_H(n)=\Theta(n^2)$  are abundant. One indication of this is our first  result which looks at the binomial random graph $H=G(k,p)$ on $k$ vertices with each edge present independently with probability $p=p(k)$.

\begin{thm}\label{thm:random}
	Let $H=G(k,p)$. Then with probability tending to 1 as $k\to\infty$ we have that 
\begin{linenomath}
 \begin{numcases}
 {M_H(n)=}
	 O(1) & if  $p = o(\log k/k)$; \label{random 0}\\ 
	\Omega(n^2) & if  $p= \omega(\log k/k)$. \label{random 1}
	\end{numcases}   
    \end{linenomath} 
	\end{thm}

 We remark that here, and throughout, the asymptotics of the function $M_H(n)$ are in terms of the parameter $n$ \emph{only}. 
 The asymptotic dependence on $k$ in the statement of Theorem \ref{thm:random} should be treated first and throughout this paper, the reader should think of $k=v(H)$ as being fixed and $n$ tending to infinity. For example, phrasing part \eqref{random 1} of Theorem \ref{thm;dense quad} explicitly we have that for any $\eps>0$ and for any function $p(k)$ such that $p(k)k/\log k$ tends to infinity with $k$, there exists some $k_0\in \NN$ such that the following holds  for all $k\geq k_0$. There is some $c=c(k)>0$ such that with probability at least $1- \eps$ the graph $H=G(k,p)$ satisfies  $M_H(n)\geq cn^2$ for all $n\in \NN$.

Note that Theorem \ref{thm:random} shows a stark phase transition in the behaviour of $H=G(k,p)$ around the connectivity threshold $\log k/k$. Below this threshold, asymptotically almost surely  $H(k,p)$ will have an isolated edge which leads to a very fast percolation process no matter the starting graph. In contrast, as soon as $p(k)$ is significantly above the connectivity threshold, asymptotically almost surely there is a choice of starting graph $G$ such that the $H$-process with $H=G(k,p)$ takes quadratically many steps before stabilising. This shows, in a very strong sense, that almost all graphs have quadratic maximum running time even when the average degree drops  to (almost) logarithmic in $k$. 

In relation to \emph{minimum degree} conditions, we establish the following.

	\begin{thm} \label{thm;dense quad}
    Suppose $5\leq k\in \NN$ and  $H$ is a graph with $v(H)=k$. We have that 
    \begin{linenomath}
 \begin{numcases}
 {M_H(n)\geq }
	 \Omega(n^2) & if  $k\geq 6$ and $\delta(H)> 3k/4$; \label{thm:dense}\\ 
	n^{2-O(1/\sqrt{\log n})}=n^{2-o(1)} & if  $\delta(H)\geq k/2+1$. \label{thm: min deg almost quad}
	\end{numcases}   
    \end{linenomath}
	\end{thm}

Note that Theorem \ref{thm;dense quad} generalises the lower bounds for cliques given by Balogh, Kronenberg, Pokrovskiy and Szab\'o \cite{balogh2019maximum}. It turns out that Theorem \ref{thm;dense quad} is tight in that there are $k$-vertex graphs $H$ with $\delta(H)=k/2$ which have much faster maximum running time. Indeed, the following example was shown in \cite{FMSz2} to have linear running time. This shows that there is also a sharp phase transition when considering   the  behaviour of $M_H(n)$ under extremal minimum degree conditions on $H$. 

   \begin{figure}[h]
    \centering
    \includegraphics[width=0.4\linewidth]{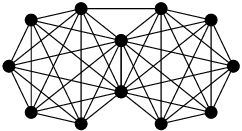}
    \caption{An illustration of the graph $H'_7$.  }
    \label{fig:connectivity}
    \end{figure}

    \begin{exmp} \label{ex:H'k}
    Let $k\geq 3$ and  $H'_k$ be the $(2k-2)$-vertex graph composed by `gluing together' two cliques of size $k$ along a singular edge $e$ and adding one more edge $e'$ between two non-adjacent vertices. (An illustration of $H'_7$ is given in Figure \ref{fig:connectivity}).    Then $M_{H'_k}(n)=\Theta(n)$ as shown in \cite{FMSz2}. 
    \end{exmp}

Returning to Theorem \ref{thm:random}, we remark that  $\log k/k$ is also the threshold for the random graph $H=G(k,p)$ to contain a \emph{Hamilton cycle}, that is, a cycle covering all the vertices of  $H$. It is not true that any graph containing a Hamilton cycle has a slow running time. Indeed, one can simply take the cycle itself which has logarithmic running time. However, it turns out that containing the \emph{square} of a Hamilton cycle is enough to guarantee a running time that is almost quadratic. Here, the \emph{square} of a graph $J$ is obtained by adding all edges between vertices of distance  2 in $J$.

    \begin{thm} \label{thm: square ham cycle}
          If $H$ with $k\geq 5$ vertices contains the square of a Hamilton cycle, then \[M_H(n)\geq n^{2-O(1/\sqrt{\log n})}=n^{2-o(1)}.\]
    \end{thm}

Note that Theorem \ref{thm: square ham cycle} also generalises the result on $K_5$ of Balogh, Kronenberg, Pokrovskiy and Szab\'o \cite{balogh2019maximum} as the square of a cycle of length 5 is $K_5$. Note also that by taking $H$ itself to be a square of a cycle, Theorem \ref{thm: square ham cycle} gives a sequence of extremely sparse $k$-vertex graphs for $k\geq 5$, that reach close to the maximum running time. Indeed, the square of a cycle is regular with vertex degrees just 4.

\subsection{The wheel graph} \label{intro sec: wheel}
One question raised by the results above is whether the lower bounds of the form $n^{2-o(1)}$ in Theorems \ref{thm;dense quad} and \ref{thm: square ham cycle} are tight or if they can be improved to prove quadratic running time. Of particular interest is the case $H=K_5$ which, as mentioned above, was conjectured to have $M_{K_5}(n)=o(n^2)$ by Bollob\'as, Przykucki, Riordan and Sahasrabudhe \cite{bollobas2017maximum}. Although Balogh, Kronenberg, Pokrovskiy and Szab\'o \cite{balogh2019maximum} disproved the conjecture for cliques of size at least 6, they stated that they tend to agree that the running time should be sub-quadratic for $K_5$.

Our next main result shows that it is possible to have a sub-quadratic running time that is  $n^{2-o(1)}$. 
The \emph{wheel graph} $W_k$ is defined to be the graph with $k+1$ vertices   obtained by taking a cycle of length $k$ and adding a new vertex that is adjacent to all vertices of  the cycle.

	\begin{thm}\label{thm:wheel}
	Let $k\geq 7$ be an odd integer.
	The wheel graph $W_k$ satisfies \[M_{W_k}(n) \geq n^{2-o(1)} \quad \mbox{ and } \quad  M_{W_k}(n) = o(n^2).\]
	\end{thm}

This suggests that such a running time could also be possible for $K_5$ and the other graphs captured by Theorem \ref{thm;dense quad} part \eqref{thm: min deg almost quad} and Theorem \ref{thm: square ham cycle}. We remark that the restriction to odd $k$ is for somewhat technical reasons and is necessary for the upper bound only. We expect that the result also holds for even $k$ but did not pursue this. 


\subsection{Bipartite graphs} \label{sec intro bip} The known results for maximum running times for graphs $H$ seem to indicate that apart from a few exceptions having maximum running times that are constant (trees), logarithmic (cycles) and linear  ($K_4$, $K_{2,s}$ for $s\geq 3$ and the graphs $H'_k$ in Example \ref{ex:H'k}), all other running times are either quadratic or of the form $n^{2-o(1)}$. Indeed the results mentioned above  cover a wide range of different graphs including graphs that are very sparse  such as the square of a cycle and the wheel. This raises questions as to whether other types of running times are possible and whether there is some large class of graphs which have running times polynomially separated from the trivial upper bound $n^2$.

Our next  result, whose proof is simple,  gives the first example of an effective upper bound on running times that covers a general class of graphs, namely bipartite graphs.  Recall that the  \textit{extremal number} (or\textit{ Tur\'an number}) of a graph, denoted $\ex(n,H)$, is the maximum number of edges of an $n$-vertex graph that does not contain $H$ as a subgraph.

	\begin{prop} \label{prop:bipartite_extremalnumber}
	Let $H$ be a graph with at least 2 edges. Then we have that
	\[
	M_H(n) \leq 2\mathrm{ex}(n,H).
	\]
	\end{prop}
    \begin{proof}
        Let $(G_i)_{i\geq 0}$ be an $H$-process with $\tau=\tau_H(G_0)=M_H(n)$. For $1\leq i\leq \tau$, let $e_i\in E(G_i)\setminus E(G_{i-1})$ be an edge added at time $i$. Now consider the graph $G'$ with $V(G')=V(G_0)$ and \[E(G')=\{e_i:i\in [\tau], i=1 \mbox{ mod }2\}.\]
        We claim that $G'$ has no copy of $H$. Indeed, if there is some copy $F$ of $H$ in $G'$ and $i_*\in [\tau]$ is the maximal index such that $e_{i_*}\in E(F)$, then $E(F)\setminus \{e_{i_*}\}\subseteq E(G_{i_*-2})$ and so $e_{i_*}$ would be added by time $i_*-1$, contradicting the definition of $e_{i_*}$. Thus, $\tau/2\leq e(G')\leq \ex(n,H)$ as required.  
    \end{proof}

As a consequence of Proposition \ref{prop:bipartite_extremalnumber} and the celebrated K\H{o}v\'ari-S\'os-Tur\'an theorem \cite{khovari_problem_1954}, we get the following upper bound for bipartite graphs.

	\begin{cor}\label{cor:extremal_bipartite}
	Let $H$ be a bipartite graph such that the two partite sets of $H$ have size $r$ and $s$, respectively, where $1\leq r\leq s$. Then
		\[
		M_H(n) = O(n^{2-1/r}).
		\]
	\end{cor}

Corollary \ref{cor:extremal_bipartite} shows that for bipartite graphs $H$, the running time $M_H(n)$ cannot reach quadratic length. However it could be that, as with our known examples of trees, cycles and $K_{2,s}$, the running time actually drops to linear. Our next  result shows that this is not the case if the bipartite graph is sufficiently dense.   

\begin{thm} \label{thm:bip dense}
Let $3\leq r\leq s$ and suppose $H$ is a bipartite graph with parts $X,Y$ with $|X|=r$, $|Y|=s$ and such that $d(x)\geq s/2+1$ for all $x\in X$ and $d(y)\geq r/2+1$ for all $y\in Y$. Then \[M_H(n)\geq n^{3/2-o(1)}.\]
\end{thm}

Theorem \ref{thm:bip dense} can be seen as a bipartite analogue to Theorem \ref{thm;dense quad} and in particular gives a running time of $M_{K_{r,s}}\geq n^{3/2-o(1)}$ for all complete bipartite graphs $K_{r,s}$ with $3\leq r\leq s$. This lower bound still falls short of the upper bound in Corollary \ref{cor:extremal_bipartite} in all cases. For larger complete bipartite graphs, we can improve the lower bound, giving the following.

\begin{thm} \label{thm:complete bip}
    For $3\leq r\leq s$, the maximum running time $M_{K_{r,s}}(n)$ is bounded from below by
	\begin{equation} \label{eq:comp bip lower}
	M_{K_{r,s}}(n) \geq n^{2-\tfrac{r+s-1}{r(s-1)-1}-o(1)}.
	\end{equation}
\end{thm}

The bound in Theorem \ref{thm:complete bip} becomes stronger than that of Theorem \ref{thm:bip dense} when $s>\tfrac{3r-1}{r-2}$, in particular for $K_{r,r}$ with $r\geq 5$. With $r$ fixed and $s$ growing, the bound of Theorem \ref{thm:complete bip} approaches the upper bound given by Corollary \ref{cor:extremal_bipartite}. Indeed $\tfrac{r+s-1}{r(s-1)-1}\rightarrow \tfrac{1}{r}$ as $s \rightarrow \infty$. 
 Moreover, Theorem \ref{thm:complete bip} implies infinitely many different possible running times. 

    \begin{cor} \label{cor:compl bip lower}
  For every $3\leq r\in \NN$ there exists $s_r\in \NN$ such that
  \[n^{2-1/({r-1})} \ll M_{K_{r,s_r}}(n)\leq O(n^{2-1/r}).\]
    \end{cor}


We remark that one can in fact prove a slightly better exponent in Theorem \ref{thm:complete bip}, giving that $M_{K_{r,s}}(n)\geq n^{2-\tfrac{1}{r} - \tfrac{1}{s-1}-o(1)}$.  
The proof is similar but requires some additional technicalities. As we do not believe that either lower bound is tight, and the main motivation for Theorem \ref{thm:complete bip} is Corollary \ref{cor:compl bip lower} and showing that the bounds approach the upper bound of Corollary \ref{cor:extremal_bipartite} as $s\rightarrow \infty$, we opt to prove the bound as in \eqref{eq:comp bip lower} and just sketch the improvement in the exponent after, see Remark \ref{rem:comp bip better}. 

In general, it is unclear which of our bounds are closer to the truth for bipartite graphs and the picture is complicated by the fact that the extremal number is not known for many bipartite graphs $H$ of interest, see for example \cite{furedi2013history}. We tend to think that the bound from Theorem \ref{thm:bip dense} is close to being tight for small graphs such as $K_{3,3}$ or the 3 dimensional cube graph $Q_3$. In fact, for the latter we can slightly improve the lower bound, removing the $o(1)$ from the exponent.

	\begin{thm}\label{thm:cube}
	The running time of the cube $Q_3$  has bounds
  \[\Omega(n^{3/2})\leq M_{Q_3}(n)\leq O(n^{8/5}).\] 
	\end{thm}

   We remark that the bounds on $M_{Q_3}(n)$ in Theorem \ref{thm:cube} match the best known bounds on $\ex(n,Q_3)$, see for example \cite{furedi2013history}. Indeed,  the upper bound of Theorem \ref{thm:cube} simply follows from Proposition \ref{prop:bipartite_extremalnumber} and the best known upper bound $\ex(n,Q_3)=O(n^{8/5})$ due to Erd\H{o}s and Simonovits \cite{erdos_extremal_1970}. The lower bound will be proven in Section \ref{sec: lad cube}.

\subsection{Inseparable graphs} \label{sec: intro insep}
The final topic covered in this paper returns to the main topic of our second paper \cite{FMSz2} in this series which concentrated on properties of graphs $H$ which have (sub-)linear running time. There we showed that all connected graphs $H$ that have sub-linear running time and are not cycles, must have some vertex of degree 1. When considering graphs with linear running time this condition is no longer necessary as we know from examples such as $K_4$ and $K_{2,s}$. To separate linear and super-linear running times, it turns out that the following notion of connectivity is useful, where we recall that a graph $F$ is $k$-connected if $v(F)>k$ and deleting any collection of at most $k-1$ vertices (and their incident edges) from $F$ leaves a connected graph.

	\begin{dfn}
	\label{def:inseparable}
 A graph $H$ is \emph{$(2,1)$-inseparable} if $H-e$ is 3-connected for any choice of $e\in E(H)$. 
	\end{dfn}

Note that a $(2,1)$-inseparable graph cannot be separated by deleting at most 2 vertices and one edge, hence the terminology. We call a graph \textit{$(2,1)$-separable} if it is not $(2,1)$-inseparable. 
Our last theorem of this paper shows that any graph $H$ with a running time that is at most linear, must be $(2,1)$-separable. 

	\begin{thm}\label{thm:connectivity}
	Any $(2,1)$-inseparable  graph $H$ with $v(H)=k$ satisfies
		\[
		M_H(n) = \Omega\left( n^{1+\tfrac{2}{3k-2}} \right) .
		\]
	\end{thm}

Note that $(2,1)$-inseparability is not necessary for super-linear running time.
For example, the wheel graph $W_k$ for odd $k\geq 7$, which has almost quadratic running time by Theorem \ref{thm:wheel}, can be disconnected by removing two vertices and an edge as it has vertices of degree three. However, in \cite{FMSz2} we gave a partial converse to Theorem \ref{thm:connectivity} showing that if a graph $H$ is $(2,1)$-separable \emph{and} it has the property that it \textit{self-percolates}, that is, $\final{H}_H=K_{v(H)}$, then $M_H(n)$ is at most linear in $n$. In particular, this was used in \cite{FMSz2} to show linear running time for the graphs $H'_k$ from Example \ref{ex:H'k} and for $K_5^-$ which is the graph obtained by removing an edge from $K_5$. 

\subsection{Methods and organisation}
 All of the lower bounds on running times stated above and proved in this paper follow from a general framework of \textit{chain constructions}. We defer the full definition and motivation of these constructions to Section \ref{sec:chains} and simply say here that it is useful to imagine a chain as a sequence of copies of $H$ with consecutive copies intersecting at a single edge. The hope is to create an $H$-process based on a chain that will run through the copies with the completion of one copy triggering the completion of the next copy in the chain at the next time step. Under some technical conditions, principally that the copies of $H$ in the chain do not induce some unwanted copy $F$ of $H-e$ for some $e\in E(H)$, one can show that the corresponding $H$-process behaves as desired and  lasts as long as the sequence of copies of $H$. 

The use of chains for constructions giving lower bounds for maximum running times  dates back to the work of Bollob\'as, Przykucki, Riordan and Sahasrabudhe \cite{bollobas2017maximum}
  on $M_{K_k}(n)$ for $k\geq 5$ who placed chains using a random process. The idea was then used again by Balogh, Kronenberg, Pokrovskiy and Szab\'o \cite{balogh2019maximum} who also introduced a method of linking chains for their lower bound on $M_{K_5}(n)$. Chain constructions have also featured in recent works looking at the analogous question for hypergraph cliques \cite{espuny_diaz_long_2022,hartarsky_maximal_2022,noel_running_2022}. One of the key contributions of our work here is to build a general theory of chain constructions which unifies  previous work and delves into the precise necessary conditions for chain constructions to work to give lower bounds on $M_H(n)$. Indeed, in Section \ref{sec:chains} we will motivate the use of chains and introduce the concept of \textit{proper chains} which are  those for which we can control the corresponding $H$-process. Section \ref{sec:linking} is then devoted to the process of linking chains, building on the initial idea present in \cite{balogh2019maximum}. In both these sections, we produce general lemmas that can be used for many different types of chain construction. We then demonstrate the applicability of this theory by giving three different types of chain construction, namely \textit{ladder chains} (Section \ref{sec:ladder}), \textit{dilation chains} (Section \ref{sec:dilation}) and \textit{line chains} (Section \ref{sec:line}). The first of these builds on the construction of Balogh, Kronenberg, Pokrovskiy and Szab\'o \cite{balogh2019maximum} giving a quadratic lower bound on $M_{K_k}(n)$ for $k\geq 6$. The second two constructions are distinct from previous constructions and introduce new ideas, in particular developing connections with other problems in additive number theory and extremal graph theory  and utilising additive constructions free from solutions to certain equations and constructions of graphs of large girth.  

 The constructions of Sections \ref{sec:ladder}, \ref{sec:dilation} and \ref{sec:line} cover all of the new lower bounds in our results listed above. In fact, in some cases the results are more general and apply to classes of graphs that in particular cover some of the cases listed in the introduction. Hence, after some simple preliminaries in Section \ref{sec:prelims} and before delving into our theory of chains, in Section \ref{sec:graph classes} we explore properties and connections between the different classes of graph $H$ for which our methods apply. After our sections covering chain constructions, in Section \ref{sec:upper} we give proofs of our upper bounds, principally the upper bound for the wheel in Theorem \ref{thm:wheel} which uses the Ruzsa-Szemer\'edi $(6,3)$-theorem \cite{ruzsa_solving_1993}. Finally in Section \ref{sec:conclude} we give some concluding remarks. 

\section{Notation and preliminaries} \label{sec:prelims}

\subsection{Notation} For a graph $H$ and an edge $e\in E(H)$, we use the notation  $H-e$ to denote the graph with vertex set $V(H-e)=V(H)$ and edge set $E(H-e)=E(H)\setminus \{e\}$. Similarly, for $e\in \binom{V(H)}{2}$ (which may or may not be an edge of $H$) the notation $H+e$ denotes the graph with the same vertex set $V(H+e)=V(H)$  and edge set $E(H+e)=E(H)\cup \{e\}$. If $U$ is some set of vertices (not necessarily a subset of $V(H)$), the graph   $H-U$ 
is the graph with vertex set $V(H-U)=V(H)\setminus (V(H)\cap U)$ and all edges induced by $H$ on this vertex set. When $U=\{u\}$ is a single vertex, we drop the set brackets and simply write $H-u$. We also drop set brackets for edges unless necessary, writing for example $e=uw$, instead of the more cumbersome $e=\{u,w\}$.

Given a graph $H$,  vertices $u,w\in V(H)$ and vertex subsets $U,W\subset V(H)$, we have that $N_H(u,U)$ denotes the set of neighbours of $u$ in $U$. If $U=V(H)$, we simply write $N_H(u)$. Similarly, the degree of $u$ in $U$ is  $d_H(u,U):=|N_H(u,U)|$ and $d_H(u)$ if $U=V(H)$. The set $E_H(U,W)$ denotes all edges with one endpoint in $U$ and the other in $W$ (with edges in $H[U\cap W]$ counted once) and $e_H(U,W):=|E_H(U,W)|$. Finally $\dist_H(u,w)$ denotes the distance between $u$ and $w$ in $H$, that is, the number of edges in a shortest path from $u$ to $w$ in $H$. In all of the above, we omit the subscript if the graph $H$ is clear from context. 

Given a set of integers $A$ in some abelian group and $t\in \ZZ$, the set $tA:=\{ta:a\in A\}$ denotes all the dilations of elements in $A$ by $t$. Finally, for $a,b,c,d\in \mathbb{R}$ we write $a=(b\pm c)d$ as shorthand for the  statement that $(b-c)d\leq a\leq (b+c)d$. 

\subsection{Graph bootstrap processes}
Recall that $\final{G}_H$ denotes the final graph of the $H$-process on $G$, that is, the graph  at which the $H$-process  stabilises. 
The following  observation can be proven easily by induction on $i\geq 0$, see \cite[Observation 2.1]{FMSz1}

    \begin{obs}\label{obs:hom}
   Let $\varphi: G \to G'$ be an injective graph homomorphism, and let $(G_i)_{i\geq0}$, $(G'_i)_{i\geq0}$ be the respective $H$-processes on $G$ and $G'$.
    Then $\varphi $ is also a homomorphism from $G_i$ to $G_i'$ for every $i\geq 0$.   In particular, $\varphi(\final{G}_H)\subseteq \final{G'}_H$. 
    \end{obs}

An immediate consequence of Observation \ref{obs:hom}  is that if $ G' \subseteq G$  and $(G'_i)_{i\geq0}$ and $(G_i)_{i\geq0}$ are the respective $H$-processes on $G'$ and $G$, 
    then $G'_i\subseteq G_i$ for every $i\geq 0$.
We say that a graph $G$ is $H$\emph{-stable} if $\final{G}_H=G$ and 
 we call $H$ \emph{self-stable} if $\final{H}_H=H$. Clearly cliques and complete bipartite-graphs are self-stable. Moreover the following are examples of self-stable graphs. 

\begin{obs} \label{obs:self-stable}
    Let $H$ be a wheel  with at least $5$ vertices or the cube graph $Q_3$. Then $\final{H}_H=H$. 
\end{obs}
\begin{proof}
   In both cases, fixing $e_1\in \binom{V(H)}{2}\setminus E(H)$ and any $e_{0}\in E(H)$, we have that $H-e_0+e_1$ has a vertex of degree 4 and so is not a copy of $H$. This shows that $e_1$ cannot complete a copy of $H$ with edges of $H$ and so $H$ is indeed self-stable. 
\end{proof}

By Observation \ref{obs:hom}, any $H$-stable graph containing $G$ must also contain every graph of the $H$-process on $G$. This implies the following observation. 

\begin{obs} \label{obs:final}
   If $G$ is an $n$-vertex graph, then 
   \[\langle G \rangle_H=\bigcap\big\{G':G' \mbox{ is a } H\mbox{-stable graph on } V(G) \mbox{ with }G\subseteq G' \big\},\]
 where here the intersection of a collection of $n$-vertex graphs refers to the graph obtained by intersecting the edge sets of the graphs in the collection.   
   \end{obs}

\subsection{Additive constructions} \label{sec:additive}
In some of our constructions, we  use sets of integers that are free to solutions of certain additive equations. Given an abelian group $\Gamma$ and an equation $E$ of the form 
\begin{equation} \label{eq:eq}
    \sum_{i=1}^h\alpha_ix_i=0, 
\end{equation}
with $\alpha_i\in \mathbb{Z}$ for $i\in [h]$, we say a solution $(x_1,\ldots,x_h)\in \Gamma^h$ to the equation is \emph{trivial} if there is some partition of $[h]$ as $[h]=T_1\cup \cdots \cup T_{\ell}$  such that for each $j\in [\ell]$ we have that $\sum_{i\in T_j}\alpha_i=0$  and $x_i=x_{i'}$ for all $i\neq i'\in T_j$. 
We denote by $r_E(\Gamma)$ the size of a largest subset of $\Gamma$ that is free from non-trivial solutions to the equation $E$ and similarly, $r_E(N)$ denotes the size of the largest subset of $[N]$  free from non-trivial solutions to $E$.

The problem of determining these extremal functions encapsulates many well studied problems in additive combinatorics. For example, when $E$ is the equation $x_1+x_2-x_3-x_4=0$, the only trivial solutions are when $\{x_1,x_2\}=\{x_3,x_4\}$ and a subset free of non-trivial solutions is called a \emph{Sidon set}. It is known that the largest Sidon set $S\subset [N]$ has size $r_E(N)=N^{1/2}(1+o(1))$ with the upper bound due to Erd\H{o}s and Tur\'an \cite{erdos_problem_1941} and the lower bound being a construction of Singer \cite{singer1938theorem}. When the equation $E$ is $x_1+x_2-2x_3=0$, non-trivial solutions are precisely 3-term arithmetic progressions with non-zero common difference. In this case, Behrend's famous construction \cite{behrend1946sets} gives that $r_E(N)\geq N^{1-O(1/\sqrt{\log N})}=N^{1-o(1)}$ whilst recent breakthrough results \cite{bloom2023improvement,kelley2023strong}  have shown that $r_E(N)\leq N^{1-\Omega(1/(\log N)^{8/9})}$. In general, Ruzsa \cite{ruzsa_solving_1993} initiated a systematic study of the behaviour of $r_E(N)$ and its dependence on the defining equation $E$. In Section \ref{sec:dil add} we will revisit the methods of Ruzsa to give sets simultaneously avoiding non-trivial solutions to several equations at once.

\subsection{Chernoff's inequality} \label{sec:chernoff}
We will use Chernoff's inequality  which we give in the following form (see~\cite[Theorem 2.1, Corollary 2.4 and Theorem 2.8]{Janson2011}).

\begin{thm}[Chernoff bounds] \label{thm:chernoff}
Let~$X$ be the sum of a set of mutually  independent Bernoulli random variables and let~$\lambda=\EE[X]$. Then for any~$0<\delta<\tfrac{3}{2}$, we have that 
\[\PP[X\geq (1+\delta)\lambda]\leq  e^{-\delta^2\lambda/3 } \hspace{2mm} \mbox{ and } \hspace{2mm} \PP[X\leq (1-\delta)\lambda] \leq  e^{-\delta^2\lambda/2 }.\]
Furthermore, if~$x\geq 7 \lambda$, then~$\PP[X\geq x]\leq e^{-x}$.
\end{thm}

\section{Graph classes} \label{sec:graph classes}

Our chain constructions give lower bounds on maximum running times $M_H(n)$ for a wide range of different graphs $H$. Indeed, some of our results discussed in the introduction are in fact more general than stated there and apply to larger classes of graphs $H$ with certain properties.  Although the definitions of these properties are tailored to our methods and what we are able to prove, the resulting graph classes are easy to define and have some interesting consequences and connections to other concepts. 
In this section, we explore simple properties of the classes of graphs $H$ we are interested in, first looking at \textit{inseparable} graphs and then introducing and exploring    \textit{Behrendian} and \textit{Sidonian} graphs. Finally, we derive properties of the random graph (which is simply a distribution on the class of all graphs) that we need to prove Theorem \ref{thm:random}.

\subsection{Inseparable graphs} \label{sec:inseparable}

Recall from Definition \ref{def:inseparable} that we say a graph $H$ is $(2,1)$-inseparable if  we cannot disconnect $H$ by removing an edge and at most two vertices. We remark that it  is implicit in the definition that a  $(2,1)$-inseparable graph is necessarily non-empty and has $v(H)\geq 4$. Moreover, as deleting an edge can only decrease the connectivity, we have that $(2,1)$-inseparable graphs are themselves 3-connected. As the only 3-connected graph $H$ with $v(H)=4$  is $K_4$ and deleting an edge from $K_4$ gives a graph with a vertex cut of size 2, we in fact have that any $(2,1)$-inseparable graph $H$ necessarily has $v(H)\geq 5$. 

We also have that 
any 4-connected graph is $(2,1)$-inseparable as if one could remove an edge and two vertices to get a separated graph, then removing those two vertices and an appropriate endpoint of the edge will also disconnect the graph.  Thus the class of $(2,1)$-inseparable graphs lies between 3-connected and 4-connected graphs. In particular, this has the following consequences.

    \begin{obs}\label{obs:mindegree_inseparable and square ham} 
    A graph $H$ with $v(H)=k\geq 5$ is $(2,1)$-inseparable if: 
      \begin{enumerate}[label=(\roman*)]
          \item \label{insep:min deg} $\delta(H)\geq k/2 + 1$; or 
          \item \label{insep:sq ham} $H$  contains the square of a Hamilton cycle.
      \end{enumerate} 
    \end{obs}
  \begin{proof}
    Let $U\subseteq V(H)$ be an arbitrary subset of $k-3$ vertices.
    We will show that $H[U]$ is connected and thus $H$ is $4$-connected and $(2,1)$-inseparable.
   In the case that $H$ satisfies \ref{insep:min deg}, if  $H[U]$ is disconnected then  there exists $W \subseteq U$ such that $|W|\leq |U|/2$ and $H[W]$ is a connected component of $H[U]$.
    However, then for any $u\in W$,
        \[
        d_H(u) \leq d_{H[W]}(u) + 3 \leq |W|-1 + 3 \leq \floor*{(k-3)/2} + 2 < k/2 + 1,
        \]
    which contradicts our assumption on $\delta(H)$. In the case that $H$ satisfies \ref{insep:sq ham}, for $u,w\in U$ arbitrary, we have two  paths between $u$ and $w$ given by traversing the Hamilton cycle in different directions. One of these paths, say $P$, has at most one vertex not in $U$ and so there is a path from $u$ to $w$ in $H[U]$ by traversing $P$ using an edge between two vertices of distance 2 in $P$ to pass over the vertex of $P$ missing in $U$ (if necessary). Thus $H[U]$ is indeed connected. 
    \end{proof}

\subsubsection{$(1,1,1)$-inseparable graphs} \label{sec:bip insep}
As discussed in the introduction, we are particularly interested in \textit{bipartite} graphs as for such graphs we have upper bounds coming from the extremal number (Proposition \ref{prop:bipartite_extremalnumber}). In the context of inseparable graphs, many of our proofs, in particular our super-linear lower bound for inseparable graphs (Theorem \ref{thm:connectivity}) extend to bipartite graphs (see Theorem \ref{thm:insep sup lin}) under slightly weaker conditions. We make the following definition which captures this.

\begin{dfn}[$(1,1,1)$-inseparable] \label{def:bip insep} A  bipartite graph $H$ is \emph{$(1,1,1)$-inseparable}  if
it  cannot be disconnected by removing an edge and at most one vertex from each partite set.
\end{dfn}

Again, implicit in the definition of $(1,1,1)$-inseparability is the fact that we only consider non-empty  graphs to be $(1,1,1)$-inseparable. Moreover, $(1,1,1)$-inseparable graphs must have at least 2 vertices in each part (otherwise we disconnect the graph by removing just one vertex) and in fact all $(1,1,1)$-inseparable graphs $H$ have $v(H)\geq 5$ as the only candidate for such a 4-vertex graph is $C_4$ which is not $(1,1,1)$-inseparable.

Note that the definition of $(1,1,1)$-inseparable  matches Definition \ref{def:inseparable} except that in the case that we can disconnect by removing one edge and two vertices, in the bipartite case we will only consider this to be valid if the two chosen vertices are in different parts of the bipartition. There are graphs for which this makes a difference. In particular, both the 3-dimensional cube $Q_3$ and the complete bipartite graph $K_{3,s}$ with $s\geq 3$ are  \textit{not} $(2,1)$-inseparable as they have vertices with degree 3. On the other hand, all these graphs \emph{are} $(1,1,1)$-inseparable.



\begin{lem} \label{lem:bip dense insep}
    Let $3\leq r\leq s$ and suppose $H$ is a bipartite graph with parts $X,Y$ with $|X|=r$, $|Y|=s$ and such that $d(x)\geq s/2+1$ for all $x\in X$ and $d(y)\geq r/2+1$ for all $y\in Y$. Then $H$ is $(1,1,1)$-inseparable.  \end{lem}
    \begin{proof}
        Suppose $H'$ is obtained from $H$ by removing at most one vertex from $X$ to get $X'$, at most one vertex from $Y$ to get $Y'$ and an edge in $H$ between $x'\in X'$ and $y'\in Y'$. Now if $x_1,x_2\in X'\setminus \{x'\}$, then as  $d_{H'}(x_1),d_{H'}(x_2)\geq s/2+1-|Y\setminus Y'|>|Y'|/2$, we have that $x_1$ and $x_2$ have a common neighbour in $H'$ and therefore lie in the same connected component. Similarly for any $y_1,y_2\in Y'\setminus \{y'\}$, we have that $d_{H'}(y_1),d_{H'}(y_2)>|X'|/2$ and all the vertices in $Y'\setminus \{y'\}$ lie in the same connected component in $H'$. Moreover, all the vertices in $Y'\setminus \{y'\}$ lie in the same connected component as the vertices in $X'\setminus \{x'\}$ because for $y\in Y'\setminus \{y'\}$, we have that $d_{H'}(y)\geq r/2\geq 3/2$ and so $y$ has at least one neighbour in $X'\setminus \{x'\}$. Finally it remains to establish that $x'$ and $y'$ are connected to the remainder of $H'$ but this follows because $d_{H'}(x')\geq s/2+1-2>0$ and similarly for $y'$. 
    \end{proof}

The next  lemma shows that if a starting graph $G_0$ is bipartite and $H$ is $(1,1,1)$-inseparable  then the graphs $G_i$ in the $H$-process will remain bipartite throughout. 
In fact $H$ being 2-edge-connected would suffice for this.

 \begin{lem}\label{lem:staying bipartite}
    Let $H$ be a $(1,1,1)$-inseparable graph. If $G$ is a bipartite graph with partite sets $X,Y\subset V(G)$, so is $\final{G}_{H}$.
    \end{lem}
    
    \begin{proof}
    Fix parts $X,Y$ of a bipartition of $G$ and let $(G_i)_{i\geq 0}$ be the $H$-process on $G$.  Suppose for a contradiction that the final graph $\final{G}_{H}$ is not bipartite and pick the smallest $i\in \NN$ for which $G_i$ contains an edge $e$ whose endpoints lie in the same part of $G$. Furthermore,  let $F$ be a copy of $H$ in $G_i$ completed by $e$.
    Then as $H$ is $(1,1,1)$-inseparable, we have that $F-e$ is connected  and so there exists a path  between the endpoints of $e$ in $F-e\subseteq G_{i-1}$. As this path connects two vertices from different parts of  the bipartite graph $F-e$,  the path must have odd length but this contradicts that $G_{i-1}$ is a bipartite graph respecting the bipartition $X,Y$. 
    \end{proof}

\subsubsection{Edges in inseparable graphs} Finally we   need the following easy consequences of inseparability. 

\begin{obs} \label{obs: disjoint edges}
    Let $H$ be a $(2,1)$-inseparable graph or a $(1,1,1)$-inseparable bipartite graph. Then for any $e\in E(H)$, there exists an $f\in E(H)$ that is vertex-disjoint from $e$. In particular, $H$ contains a pair of non-incident edges. 
\end{obs}
\begin{proof} 
    If no such $f\in E(H)$ existed, then we could disconnect $H$ by removing $e$ and its two vertices (using here that $v(H)\geq 4$). This would contradict the definition of inseparability. 
\end{proof}

\begin{lem} \label{lem:insep edges in cycle}
Let $H$ be a $(2,1)$-inseparable graph or a $(1,1,1)$-inseparable bipartite graph,  $e\in E(H)$ and  $F=H-e$. Then for any pair of edges $f,f'\in E(F)$, there is a cycle in $F$ containing both $f$ and $f'$. 
\end{lem}
\begin{proof}
   Due to $H$ being inseparable, the deletion of any one vertex from $F$ leaves a connected graph, that is, $F$ is 2-connected. Let $f=uv$ and $f'=xy$ and suppose first that $v=y$.  By Menger's theorem (see for example \cite[Section 4.2]{west2001introduction}), there are two internally-vertex-disjoint  paths with endpoints $u$ and $x$, one of which avoids $v=y$ and forms a cycle with $f$ and $f'$. On the other hand, if $f$ and $f'$ are vertex-disjoint then a well-known consequence of Menger's theorem (see for example \cite[Exercise 4.2.9]{west2001introduction}) gives that there are two vertex-disjoint paths with one endpoint in $\{u,v\}$ and the other in $\{x,y\}$. Together with $f$ and $f'$, these paths give the desired cycle. 
\end{proof}

\subsection{Higher inseparability} \label{sec:robust conn}

In Section \ref{sec:lad robust conn}, we will explore a class of graphs $H$ which have a higher level of inseparability.


\begin{dfn} \label{def:robust conn}
   We say that a graph $H$ is {\em $(\ell,1)$-inseparable} if it cannot be disconnected by the removal of an edge and at most $\ell$ vertices.
\end{dfn}

By definition, we have that $(\ell,1)$-inseparable graphs are $(\ell',1)$-inseparable for all $
1\leq \ell'\leq \ell$.

\begin{rem} \label{rem:rob conn equiv}
    If $H$ with $k=v(H)\geq 6$ is $(\ceil{k/2},1)$-inseparable then for any $e\in E(H)$ and $F=H-e$, we have that $F[U]$ is connected for all $U\subset V(F)$ with $|U|$ \textit{at least} $\floor{k/2}$.  
\end{rem}

In Theorem \ref{thm: robust conn} we will show that any $(\ceil{k/2},1)$-inseparable graph $H$ with $k\geq 6$ vertices has $M_H(n)=\Omega(n^2)$. Note that this recovers the result of Balogh, Kronenberg, Pokrovskiy and Szab\'o \cite{balogh2019maximum} that  $M_{K_k}(n)$ is quadratic in $n$ for all $k\geq 6$. Indeed, for any subset $U\subset V(K_k)$ with $|U|=\floor{k/2}$, we have that $K_k[U]$ is a clique of size at least 3 and so is 2-edge-connected. In fact, any graph $H$ with a large enough minimum degree is $(\ceil{k/2},1)$-inseparable. 

\begin{lem} \label{lem:min deg implies robust conn}
    If $H$ is a graph with $k:=v(H)\geq 6$ vertices and minimum degree $\delta(H)>3k/4$. Then $H$ is $(\ceil{k/2},1)$-inseparable. 
\end{lem}
    \begin{proof}
    It suffices to show that for any $U\subseteq V(H)$ with $|U|=\floor{k/2}$ we have that  $H[U]$ is  $2$-edge-connected. So suppose for a contradiction that there is some $U$ for which this is not the case. 
	Then we can find  $U'\subset U$ with $1\leq |U'|\leq \floor{|U|/2}$ such that there is at most one edge between $U'$ and $U\setminus U'$.
	If $|U'|\geq 2$ we can pick a vertex $u\in U'$ without neighbours in $U\setminus U'$ and arrive at
		 \begin{linenomath}  \begin{equation*}
		\delta(H) \leq d(u) \leq k - 1- |U\setminus U'|  \leq k  - 1 - \ceil*{\frac{\floor{k/2}}{2}} \leq \floor*{\frac{3k}{4}},
		\end{equation*} \end{linenomath}
	a contradiction.
	If $U' = \{u\}$ is a single vertex then $u$ has at most one neighbour in $U\setminus\{ u \}$ and so
	 \begin{linenomath}	\begin{equation*}\label{eq:finaldense}
		\delta(H) \leq d(u) \leq k - |U| + 1 = \ceil*{\frac{k}{2}} +1 \leq \floor*{\frac{3k}{4}},
		\end{equation*} \end{linenomath}
	which is again a contradiction, where we used  that $k\geq 6$ in the final inequality. 
    \end{proof}

Theorem \ref{thm;dense quad} part \eqref{thm:dense} is thus an immediate consequence of  Lemma \ref{lem:min deg implies robust conn} and Theorem \ref{thm: robust conn} mentioned above, which gives quadratic running times for all $(\ceil{k/2},1)$-inseparable $k$-vertex graphs.
The connection between minimum degree and $(\ceil{k/2},1)$-inseparable $k$-vertex graphs goes further. 

\begin{obs} \label{obs:robust conn min deg}
    If $H$  is a $k$-vertex $(\ceil{k/2},1)$-inseparable graph, then $\delta(H)\geq k/2+2$. 
\end{obs}
\begin{proof}
    Suppose that there is some vertex $v\in V(H)$ with $d(v)\leq \ceil{k/2}+1$. Then we can take a set $v\in U\subset V(H)$ such that $|U|=\floor{k/2}$ and $|U\cap N_H(v)|\leq 1$. Removing at most one edge in $U$ separates $v$ from the rest of $U$, implying that $H$ is not $(\ceil{k/2},1)$-inseparable. 
\end{proof}

Therefore the  class of $(\ceil{k/2},1)$-inseparable $k$-vertex  graphs is sandwiched between $k$-vertex graphs with minimum degree  greater than $3k/4$ and the class of graphs with minimum degree greater than  $k/2+1$. Subject to the  condition imposed by Observation \ref{obs:robust conn min deg}, the notion of $(\ceil{k/2},1)$-inseparability captures many graphs $H$ satisfying weak pseudorandom conditions. For example, we say a pair of disjoint vertex sets $A,B\subset V(H)$ with $|A|=|B|=\ell$ and $e_H(A,B)=0$ is a \emph{bipartite hole} of size $\ell$ (these were introduced by McDiarmid and Yolov \cite{mcdiarmid2017hamilton} in the context of Hamiltonicity). For  $2\leq \ell \leq k/2$, one can show that any $k$-vertex graph $H$ with $\delta(H)\geq k/2+\ell$  and no bipartite hole of size $\ell$, is $(\ceil{k/2},1)$-inseparable. Further natural examples of $k$-vertex $(\ceil{k/2},1)$-inseparable graphs are complete tripartite graphs with no part larger than $\floor{k/2}-2$ and $t^{th}$ powers of cycles for $t\geq k/4+1$, which are obtained by cyclically ordering the vertices and connecting each vertex to all vertices within distance $t$ in the ordering.  All these graphs   have quadratic maximum running time, as we will show in Theorem \ref{thm: robust conn}. 


\subsection{Behrendian and Sidonian graphs} \label{sec:coherent}

In Section \ref{sec:dilation} we will give a general  construction for lower bounds on running times which are defined via a set of \emph{dilations}.
When this set of dilations is generated using Behrend's construction of arithmetic progression-free sets, the  construction will allow us to prove lower bounds of the form $M_H(n)\geq n^{2-o(1)}$. It turns out that this works whenever we can guarantee that $H-e$ is \emph{Behrendian} for any $e\in E(H)$, with Behrendian graphs being defined as follows. 

\begin{dfn}[Behrendian graphs] \label{def:coherent}
    We say a connected graph $H$ is \emph{Behrendian} if any  colouring of $E(H)$ which is non-monochromatic results in a non-monochromatic cycle in $H$ which is the union of at most three monochromatic paths.
\end{dfn}

\begin{rem}
    We restrict our attention to Behrendian graphs which are connected;  any disconnected graph satisfying the  condition would have one Behrendian component and all other components being isolated vertices. 
\end{rem}

At first sight, the definition of Behrendian graphs may appear somewhat mysterious. We will show that the notion can be thought of as yet another type of connectivity and encompasses many graphs of interest. For example, a graph $H$ is \emph{triangularly connected} if  for any pair of distinct edges $e,f\in E(H)$, there is a sequence of triangles $T_1,\ldots,T_m$ in $H$ with $e\in E(T_1)$, $f\in E(T_m)$ and consecutive triangles in the ordering intersecting in an edge. 
It is not hard to see that triangularly connected graphs are Behrendian. Indeed, if $E(H)$ is coloured in a non-monochromatic fashion, there are edges $e,f\in E(H)$ with distinct colours but if $H$ is Behrendian, then there is no triangle that is \emph{not} monochromatic, forcing all the triangles in the sequence $T_1,\ldots,T_m$ between $e$ and $f$ to be the same colour as $e$, a contradiction. 
Triangularly connected graphs have been studied in the context of integer flows \cite{fan2008nowhere,hou2012z3,li2024integer} and provide a large family of Behrendian graphs which includes any triangulation of a surface and the square of any connected graph (see e.g.\ \cite[Proposition 6.1]{fan2008nowhere}).

For us, the significance of Behrendian graphs comes in the applications of our dilation chain constructions which will give almost quadratic running times for graphs with minimum degree at least $k/2+1$ (Theorem \ref{thm;dense quad} part \eqref{thm: min deg almost quad}), odd wheels (Theorem \ref{thm:wheel}) and graphs containing the square of a Hamilton cycle (Theorem \ref{thm: square ham cycle}). For these applications, we need that these graphs are Behrendian, even after deleting an edge. Before proving this in these various contexts, we explore the definition further, giving the following equivalent characterisation of Behrendian graphs.  

   \begin{lem} \label{lem:cohe equiv}
    Consider the class $\mathcal{C}$ of graphs obtained by the following rules:
        \begin{enumerate}
        \item \label{coh equiv 1} $H = K_2\in \cC$;
        \item \label{coh equiv 2} If $H$ has subgraphs $H_1, H_2 \subseteq H$ such that $H = H_1 \cup H_2$, $|V(H_1) \cap V(H_2)| \geq 2$ and $H_1, H_2 \in \mathcal{C}$, then $H \in \cC$;
        \item \label{coh equiv 3} If $H$ has subgraphs $H_1, H_2, H_3 \subseteq H$ and distinct vertices $u,v,w \in V(H)$ such that $H = H_1 \cup H_2 \cup H_3$, $V(H_1)\cap V(H_2) = \{u\}$, $V(H_1)\cap V(H_3) = \{v\}$, $V(H_2)\cap V(H_3) = \{w\}$, and $H_1, H_2, H_3 \in \mathcal{C}$, then $H\in \cC$.
        \end{enumerate}
   Then $\cC$ is precisely the class of Behrendian graphs.
    \end{lem}
\begin{proof}
We first show that any graph in $\cC$ is a Behrendian graph. Connectivity follows easily from the rules as $H=K_2$ is connected and inputs of connected $H_1,H_2,H_3\in \cC$ in \eqref{coh equiv 2} or \eqref{coh equiv 3} give a connected output $H\in \cC$. We also have that $H=K_2$ is  trivially Behrendian as any edge-colouring of $K_2$ is necessarily monochromatic. Now suppose that $H$ is as in \eqref{coh equiv 2} with $H_1, H_2$ both Behrendian (and thus connected) and let $\chi$ be a non-monochromatic colouring of $E(H)$.
    If one of the restricted colourings $\chi\vert_{E(H_1)}$, $\chi\vert_{E(H_2)}$ is non-monochromatic we can find the desired non-monochromatic cycle in $H$ by the assumption that $H_1$ and  $H_2$ are Behrendian.
   Therefore we can assume that $\chi\vert_{E(H_1)}$ and $\chi\vert_{E(H_2)}$ are both monochromatic with different colours. 
    Take distinct vertices $v_1, v_2 \in V(H_1) \cap V(H_2)$, a $v_1v_2$-path $P$ in $H_1$, and a $v_1v_2$-path $Q$ in $H_2$ (using here that both $H_1$ and $H_2$ are connected).
   If $P\cup Q$ is a cycle then it is the union of two monochromatic paths in different colours and we are done. If $P\cup Q$ is not a cycle, then there is a non-empty set of vertices $W$ such that $W\subset V(P)\cap V(Q)\setminus \{v_1,v_2\}$. Let $w\in W$ be the first vertex in $W$ which is hit when traversing from $v_1$ to $v_2$ in $P$. Taking the segment of $P$ between $v_1$ and $w$ and the segment of $Q$ from $v_1$ to $w$ now gives a cycle which is the union of two monochromatic paths and we are done. Finally assume that $H$ is a graph as in \eqref{coh equiv 3} such that  $H_1, H_2, H_3$ are all Behrendian (and connected) and let $\chi$ be a non-monochromatic colouring of $E(H)$.
    Again, if one of $\chi\vert_{E(H_1)}$, $\chi\vert_{E(H_2)}$, $\chi\vert_{E(H_3)}$ is non-monochromatic, we are done.
    If not, choose a $uv$-path in $H_1$, a $uw$-path in $H_2$, and a $vw$-path in $H_3$.
    These three paths are monochromatic with at most two of them having the same colour and their union is a cycle in $H$, therefore witnessing that $H$ is Behrendian.

    It remains to show that any  Behrendian graph lies in $\cC$. So fix $H$ to be  Behrendian  and 
    take the smallest $r\in\NN$ such that $H = H_1 \cup\ldots\cup H_r$ for some $H_1,\ldots,H_r\in \mathcal{C}$.
    This is well-defined as $H$ can be written as the union of its edges which are all in $\cC$ by \eqref{coh equiv 1}.
    We will show that $r=1$.
    Suppose for a contradiction that $r\geq 2$.
    Any two of the $H_i$ intersect in at most one vertex for if $H_i$ and $H_j$ have two common vertices for some $i\neq j$, we obtain $H_i \cup H_j \in \mathcal{C}$ and hence $r$ would not be minimal.
    In particular, the $H_i$ are pairwise edge-disjoint.
    Define an $r$-colouring $\chi: E(H) \to [r]$ by sending $e \in E(H)$ to the unique $i\in[r]$ with $e\in E(H_i)$.
    Since $\chi$ is non-monochromatic we can find a non-monochromatic cycle $C$ that is the union of two or three monochromatic paths.
    If there exist $i,j\in [r]$, $i\neq j$, and paths $P\subseteq H_i$, $Q\subseteq H_j$ such that $C = P\cup Q$, we have arrived at a contradiction because $P$ and $Q$ have the same endpoints while $|V(H_i)\cap V(H_j)|\leq 1$.
    Assume then that there are distinct $i,j,\ell \in [r]$ and paths $P\subseteq H_i$, $Q\subseteq H_j$, $R\subseteq H_\ell$ such that $C = P\cup Q\cup R$.
    The common endpoint of $P$ and $Q$ is the unique vertex of $H_i\cap H_j$.
    Similarly, $H_i\cap H_\ell = P\cap R$ and $H_j\cap H_\ell = R \cap Q$.
    Thus we see that $H_i \cup H_j \cup H_\ell \in \mathcal{C}$ by the third condition from the definition of the class $\mathcal{C}$, which contradicts the minimality of $r$.
    \end{proof}
    
This alternative characterisation leads to the following easy consequences. 

\begin{obs} \label{obs:coh easy}
   Suppose $H$ is Behrendian. If $H'$ is obtained from $H$ by either
   \begin{enumerate}[label=(\roman*)]
       \item \label{coh add edge} adding an edge with vertices in $V(H)$; or
       \item \label{coh add vx} adding a new vertex with at least two neighbours in $V(H)$;
   \end{enumerate}
   then $H'$ is also Behrendian.
\end{obs}
\begin{proof}
    In case \ref{coh add edge}, we apply Lemma \ref{lem:cohe equiv} \eqref{coh equiv 2} with $H_1=H$ and $H_2$ being the single new edge. In case \ref{coh add vx}, we apply Lemma \ref{lem:cohe equiv} \eqref{coh equiv 3} with $H_1=H$ and $H_2$ and $H_3$ being distinct single edges incident to the new vertex. If the new vertex has more than two neighbours in $V(H)$ then we obtain that $H'$ is Behrendian by adding the remaining edges one at a time and appealing to part \ref{coh add edge} each time.  
\end{proof}

In particular, Observation \ref{obs:coh easy} \ref{coh add vx} shows that the class of connected Behrendian graphs on $k$ vertices is closed under the addition of edges, a fact that is not immediate from the original definition. We now show that the graphs that we will be interested in are indeed Behrendian.

\begin{lem} \label{lem: H minus coherent}
    Any graph $H$ with $v(H)=k\geq 5$ has the property that $H-e$ is  Behrendian for any choice of $e\in E(H)$ if any of the following hold:
    \begin{enumerate}
        \item \label{coh: wheel} $H$ is a wheel; or
        \item \label{coh:sq} $H$ contains the square of a Hamilton cycle; or
        \item \label{coh: min deg} $\delta(H)\geq k/2+1$.  
    \end{enumerate}
\end{lem}
\begin{proof}
    Let $H$ be a graph satisfying one of the conditions \eqref{coh: wheel}-\eqref{coh: min deg} of the  lemma, let $e\in E(H)$ be arbitrary and let $F=H-e$. By repeated applications of Observation \ref{obs:coh easy}, it suffices to give some ordering $v_1,\ldots,v_k$ of $V(F)=V(H)$ such that $v_1v_2$ is an edge of $F$ and for $3\leq i\leq k$, we have that $v_i$ has at least two neighbours in $F$ amongst $\{v_1, \ldots,v_{i-1}\}$. If $H$ is a wheel, then take $v_1$ to be the centre of the wheel and order the remaining vertices as $v_2, \ldots, v_k$ respecting the cyclic order determined by the cycle of the wheel and such that the (one or two) vertices of degree two in $F$ come last in the order. One can check that this gives a valid ordering. For case \eqref{coh:sq}, by Observation \ref{obs:coh easy} \ref{coh add edge}, we can assume that $H$ itself is the square of a cycle. Similarly to the wheel,  we order the vertices according to the cyclic order given by the cycle, starting with one of the endpoints $v$ of the edge $e\in E(H)\setminus E(F)$ and ensuring that the other endpoint of $e$ is either $v_{k-1}$ or $v_k$, depending on whether $e$ is an edge between vertices of distance two or one on the cycle. 

    For case \eqref{coh: min deg}, we deduce the  existence of a valid ordering by showing that for any  $U\subset V(H)=V(F)$ with $2\leq |U| \leq k-1$, there is a vertex $v\in V(F)\setminus U$ with $d_F(v,U)\geq 2$. Indeed, then the ordering can be constructed greedily by starting with an arbitrary edge of $F$ defining $v_1$ and $v_2$.   Firstly, let $2\leq |U| \leq k/2 + 1$.
    We estimate the number of edges in $F$ between $U$ and $V(F)\setminus U$:
       \[
        e_{F}(U,V(F)\setminus U)
        \geq e_H(U,V(F)\setminus U) - 1 
        \geq |U|  (\delta(H)-|U|+1) - 1 
    \geq|U| (k/2 + 2 - |U|) - 1.\]
    The final quantity is strictly greater than $k-|U|$ for all $2\leq |U| \leq k/2 + 1$ and so there is indeed a vertex $v$ with $d_F(v,U)\geq 2$ by the pigeon hole principle.  
    If $k/2+1<|U|\leq k-1$, then for all vertices $v\in V(F)\setminus U$ we have that $d_H(v)+|U|>k+2$ and so $d_H(v,U)\geq 2$. Thus any vertex $v$ not incident to $e$ is a good choice. If there is no such $v$, then we have that $V(F)\setminus U$ consists only of vertices incident to $e$ and thus has size at most 2 and so for $v\in V(F)\setminus U$, we have that $d_F(v,U)\geq d_H(v,U)-1\geq d_H(v)-2\geq k/2-1$ and hence  $d_F(v,U)\geq 2$ as $k\geq 5$ and $d_F(v,U)\in \NN$.  This implies the existence of the required vertex in all cases and hence concludes the proof. 
    \end{proof}

\subsubsection{Sidonian graphs} \label{sec:bicoherent}
There are no bipartite graphs that satisfy Definition \ref{def:coherent} of being Behrendian as  if we colour a bipartite graph $H$ in a \emph{rainbow} fashion, that is giving each edge a unique colour, then any cycle in $H$ has length at least 4 and therefore induces at least 4 colours. Nonetheless, we will be able to work with a somewhat analogous variant which we call  \textit{Sidonian}. In the context of the dilation chains which we will introduce in Section \ref{sec:dilation}, we will be able to provide lower bounds on running times of the type $M_H(n)\geq n^{3/2-o(1)}$ by using a set of dilations based on a variation of a Sidon set.  We will show in Theorem \ref{thm:bip coherent run time} that this  construction works whenever $H-e$ satisfies the following property of being \emph{Sidonian}, for any choice of $e\in E(H)$.

\begin{dfn} \label{def:bicoh}
    We say a bipartite graph $H$ is \emph{Sidonian} if any non-monochromatic colouring of $E(H)$ results in a non-monochromatic copy of $C_4$.
\end{dfn}

The definition of Sidonian is perhaps not the most natural generalisation of Behrendian to a bipartite setting and is rather formulated to fit with our application to running times in Theorem \ref{thm:bip coherent run time}. 
In particular, we do not get  a   characterisation as in Lemma \ref{lem:cohe equiv} in this bipartite setting. Some consequences of the definition do carry over naturally though. Indeed, analogously to triangularly connected graphs being Behrendian, one can see that if there is a sequence of intersecting $C_4$s between any pair of edges in a bipartite graph $H$, then it is necessarily Sidonian. This includes all \textit{quadrangulations} of a surface, for example. For us, the notion will be used to prove Theorem \ref{thm:bip dense}  where we will need that all bipartite graphs with sufficiently large minimum degree conditions are Sidonian, even after deleting an edge. 

\begin{lem} \label{lem:dense bip coh}
      Let $3\leq r\leq s$ and suppose $H$ is a bipartite graph with parts $X,Y$ with $|X|=r$, $|Y|=s$ and such that $d(x)\geq s/2+1$ for all $x\in X$ and $d(y)\geq r/2+1$ for all $y\in Y$. Then $H-e$ is Sidonian for any choice of $e\in E(H)$.  
\end{lem}
\begin{proof}
    Let $e=xy\in E(H)$ be arbitrary with $x\in X$ and $y\in Y$. Consider a colouring $\chi$ of $E(H-e)$ with no non-monochromatic copy of $C_4$. We will show that $\chi$ is necessarily monochromatic and so $H-e$ is Sidonian. Firstly, let $F$ be the graph obtained from $H-e$ by further deleting the endpoints $x,y$ of $e$ and all their incident edges. We claim that $F$ is monochromatic under $\chi$. Indeed for any pair of  distinct edges $f,f'\in E(F)$ with $f\cap f'\neq \emptyset$, letting $v=f\cap f'$ and $\{u,w\}=f\cup f'\setminus \{v\}$, we have that \begin{equation} \label{eq:degs high bip coh} d_{H-e}(u)+d_{H-e}(w)=d_H(u)+d_H(w)\geq \ell+2,\end{equation}
    with $\ell=s$ if $u,w\in X$ and $\ell=r$ if $u,w\in Y$. 
    Hence there is some vertex $v'\in N_{H-e}(u)\cap N_{H-e}(w)\setminus \{v\}$ such that $v,v',u,w$ span a copy of $C_4$, which is monochromatic by assumption and so $\chi(f)=\chi(f')$. Note that we used  here that $\{u,w\}\cap \{x,y\}=\emptyset$ but it is possible that the vertex $v'\in \{x,y\}$. This shows that any pair of incident edges in $F$ are the same colour and using that $H$ is $(1,1,1)$-inseparable 
    (Lemma \ref{lem:bip dense insep}) which implies that $F$ is connected, we get that the whole of $F$ is monochromatic. It remains to show that edges incident to $x$ or $y$ are the same colour as the edges in $F$. Note that $d_{H-e}(x)\geq 2$ by the minimum degree condition on $H$, so for any edge $xz\in E(H-e)$ containing $x$, we can choose another  edge $xz'\in E(H-e)$. Moreover, using that $z,z'\neq y$, as in \eqref{eq:degs high bip coh}, we have that $d_{H-e}(z)+d_{H-e}(z')\geq r+2$ and the edges $xz$ and $xz'$ are contained in a $C_4$ with two edges from $F$. Therefore the edges $xz,xz'$ are the same colour as the edges of $F$. This argument can be repeated for any edge $xz\in E(H-e)$ and analogously for edges incident to $y$, rendering $H-e$ monochromatic and completing the proof. 
\end{proof}

\subsection{Random graphs} \label{sec:random}
We are interested in the  binomial random graph $H=G(k,p)$  with vertex set $[k]$  and each pair of vertices $\{u,v\}\in \binom{[k]}{2}$ appearing with probability $p$ independently. We say an event holds \textit{asymptotically almost surely} (a.a.s.\ for short) in $H$ if the probability it holds tends to 1 as $k$ tends to infinity. Let us first recall the following well-known fact, see e.g.\ \cite[Theorem 5.4]{bollobas_random_2001}. 

\begin{fact} \label{fact:G(k,p) isolated edge}
    If $p=o(\log k/k)$ then a.a.s\ $H=G(k,p)$ is either empty or contains an isolated edge. 
\end{fact}

Mostly, we are interested in the range $p=\omega(\log k/k)$ above the connectivity threshold. Here we will show that a.a.s.\ $H=G(k,p)$ has maximum running time $M_H(n)$ which is quadratic in $n$.  In order to prove this, we will need several properties of the random graph that hold a.a.s.\ in this range. Firstly, we need that $H$ is highly connected, see \cite[Theorem 7.5]{bollobas_random_2001}. 

\begin{fact} \label{fact: G(k,p) connected}
    If $p=\omega(\log k/k)$ then a.a.s.\ $H=G(k,p)$ is 4-connected and so is $(2,1)$-inseparable. 
\end{fact}

We will also use that  $H$ is self-stable. In order to deduce this, note that if $H$ is \emph{not} self-stable, then there is some edge $e\in E(H)$ and a non-trivial\footnote{That is, not the identity map.} bijection $\pi:[k]\rightarrow [k]$  such that  $\pi(H-e)\subset H$ and $\pi(e)\notin E(H)$, that is, every edge of $H-e$ is mapped via $\pi$ to an edge of $H$ on $[k]$ and $e$ is not. This in turn implies that the symmetric difference  $H \triangle \pi(H):=(E(H)\setminus E(\pi(H))\cup (E(
\pi(H))\setminus E(H))$ has size $|H \triangle \pi(H)|=2$. To rule this out, we can appeal to the well-studied topic of the asymmetry of $H=G(k,p)$. Indeed, if there was a non-trivial bijection $\pi:[k]\rightarrow [k]$ with $H\triangle \pi(H)=\emptyset$, this would give a non-trivial automorphism of $H$. In one of the earliest papers studying binomial random graphs, Erd\H{o}s and R\'enyi \cite{erdHos1963asymmetric} showed that a.a.s.\  $H=G(k,p)$ has a trivial automorphism group. Their result is stated for $p=1/2$ but works when
$\min\{p,1-p\}\geq \log k/k$ and both $H$ and its complement are connected. We will use a strengthening of this result due to Kim, Sudakov and Vu \cite[Theorem 3.1]{kim2002asymmetry} which implies that if $\min\{p,1-p\}=\omega (\log k/k)$, then a.a.s.\ for any non-trivial $\pi:[k]\rightarrow [k]$, one has that $|H\triangle\pi(H)|\geq (2-o(1))kp(1-p)>2$ (in fact, they show the stronger statement that there is some vertex $u\in [k]$ and this many edges in the symmetric difference $H\triangle \pi(H)$ that are all incident to $u$, a parameter that they call the \emph{defect} of $H$).
In particular, there is no non-trivial embedding $\pi:H-e\rightarrow H$ for some $e\in E(H)$ and we get the following. 

\begin{fact} \label{fact: G(k,p) self-stable}
     If $\min\{p, 1-p\}=\omega(\log k/k)$ then a.a.s.\ $H=G(k,p)$ is self-stable, that is $\final{H}_H=H$. 
\end{fact}

Finally we will need some basic properties of $H$ which we collect in the following lemma. 

\begin{lem} \label{lem:G(k,p) properties}
    Suppose $0<\eps<1/150$,  $p=\omega(\log k/k)$ and fix a partition $[k]=U\cup W$ with $|U|=\floor{k/2}$ and $W=\ceil{k/2}$. Then a.a.s.\ $H=G(k,p)$ on $[k]$ has the following properties:
    \begin{enumerate}
        \item \label{G(k,p): degrees U W} $d_H(v,U),d_H(v,W)=(1/2\pm\eps)kp$ for all $v\in [k]$;
        \item \label{G(k,p):hole} there are no disjoint sets $A,B\subset [k]$, with $|A|=|B|=m\geq \eps k$ and $e_H(A,B)\leq 1$; and
\item \label{G(k,p): dense small} for every $Y\subset [k]$, with $|Y|\leq 4\eps k$, we have that $e(H[Y])\leq  |Y|kp/10.$ 
    \end{enumerate}
\end{lem} 
\begin{proof}
    Property \eqref{G(k,p): degrees U W} follows from Chernoff bounds (Theorem \ref{thm:chernoff}) and a union bound over the $k$ choices of $v$ and a choice of $X=U$ or $X=W$, noting that $\mathbb{E}[d_H(v,X)]=(k/2\pm 1)p$ and $kp=\omega(\log k)$. Similarly, for property \eqref{G(k,p):hole}, for any choice of $A$ and $B$, we have that \[\mathbb{E}[e_H(A,B)]=|A||B|p=m^2p\geq \eps^2 k^2p=\omega(k\log k)\] and the probability that $e_H(A,B)\leq 1$ is less than $\exp(-\omega(k\log k))$ by Theorem \ref{thm:chernoff}. A union bound over choices of $A$ and $B$ (at most $2^{2k}$ choices) completes the proof of \eqref{G(k,p):hole}.  Finally, for \eqref{G(k,p): dense small}, note that for a fixed $1\leq m'\leq 4 \eps k$ and set $Y$ with $|Y|=m'$, we have that $\lambda=\mathbb{E}[e(H[Y])]=\binom{m'}{2}p \leq m'kp/70$ due to our upper bound on $m'$ and $\eps$. 
    Thus by the `furthermore' part of Theorem \ref{thm:chernoff}, we have that the probability that $e(H[Y])>m'kp/10$ is less than $\exp(-m'kp/10)$. 
    A union bound over the choices of $Y$ gives that the probability that \eqref{G(k,p): dense small} fails is at most 
    \[\sum_{m'=1}^{4\eps k}\binom{k}{m'}\exp(-m'kp/10)\leq \sum_{m'=1}^{4\eps k}\exp(m'(\log k-kp/10)) \leq 4\eps k \exp(\log k-kp/10) \rightarrow 0,\]
    as required. 
\end{proof}

\section{Chains} \label{sec:chains}

In this section, we introduce the concept of \emph{chains}, which will be our  tool for proving lower bounds for running times.  The definition of a chain (Definition \ref{def:H-chain}) is given in generality and we then define a chain to be \emph{proper} (Definition \ref{def:proper chain}) if it satisfies extra conditions that allow us to control the running time of a construction based on the chain. Indeed, the  key result of this section is Lemma \ref{lem:chain_runningtime}, giving a lower bound on the running time of an $H$-process on a starting graph formed by a proper $H$-chain after removing certain edges.  We remark that our methods here are closely related to those used by \cite{balogh2019maximum,bollobas2017maximum} in giving lower  bounds for $M_{K_r}(n)$ with $r\geq 5$ as well as the constructions in ~\cite{espuny_diaz_long_2022,hartarsky_maximal_2022,noel_running_2022} investigating $H$-processes for hypergraph cliques $H$. Indeed, all these previous constructions utilise starting graphs which are defined by certain chains and the proofs work by identifying  properties of the chain that allow control of the running time for these constructions. In \cite{bollobas2017maximum}, they use the terminology of \emph{good} chains whilst in \cite{balogh2019maximum}, they introduce an ordered hypergraph to characterise the chain and appeal to the key property of the hypergraph being \emph{induced $K_r^-$-free}. Our notion of a chain being \emph{proper} is analogous to these previous concepts  but, in order to deal with general graphs $H$, we need to refine our definition. The resulting conditions are geared towards being widely applicable, as we will demonstrate in the remainder of the paper. Before introducing formally our definitions, we provide some motivation for the concept, discussing how chains appear in all $H$-processes.

Consider an $H$-process on some starting graph $G_0$ and some integer $\tau\leq \tau_H(G_0)$. Then one can define a natural sequence $(H_i, e_i)_{i\in [\tau]}$ of copies $H_i$ of $H$ on $V(G_0)$, together with edges $e_i\in E(H_i)$ 
such that the following holds. Each 
 $H_i$ is completed at time $i$ with the addition of the edge $e_i$ and, for $i\in[\tau-1]$, adding $e_i$ then triggers the completion of $H_{i+1}$ in the next step. Indeed, such a sequence can be formed in reverse,   starting with an edge $e_\tau$ added at time $i=\tau$. Then  there is some $H_i$ completed at time $i$ by the edge $e_i$ and $H_i-e_i$ has some edge $e_{i-1}$ which is added at time $i-1$, otherwise $H_i$ already appears at step $i-1$. Repeating this process by using the edge $e_{i-1}$ to define $H_{i-1}$ which then defines the edge $e_{i-2}$ and so on, we create the desired sequence $(H_i, e_i)_{i\in [\tau]}$ which we can visualise as a \emph{chain} of copies of $H$, with $H_i$ and $H_{i+1}$ \emph{linked} by the edges $e_i\in E(H_i)\cap E(H_{i+1})$. The following definition  formalises the idea of a chain of copies of $H$.

 \begin{dfn}[Chains]\label{def:H-chain}
	Let $H$ be a  connected graph, and let $\tau\geq 1$.
	An $H$\emph{-chain}\footnote{When the graph $H$ is clear from context, we will simply refer to a \emph{chain}.} of \emph{length} $\tau$ is a sequence $(H_i, e_i)_{i\in [\tau]}$ of copies $H_i$ of $H$ together with edges $e_i\in E(H_i)$ such that
	 $e_i \in E(H_i)\cap E(H_{i+1})$ for $i\in[\tau-1]$. 
  	We call $G:=\bigcup_{i=1}^\tau H_i$ the \emph{underlying graph} of the $H$-chain.
  \end{dfn}

The discussion above shows that any $H$-process $(G_i)_{i\geq 0}$ defines an $H$-chain of length $\tau_H(G_0)$. In what follows, we will use chains to construct starting graphs $G_0$ for $H$-processes. Indeed by taking  some $H$-chain $(H_i, e_i)_{i\in [\tau]}$, we can hope to construct  $G_0$ with $\tau_H(G_0)\geq \tau$, by taking $G_0$ to be the union of the graphs $H_i$ with all of the $e_i$ removed. The aim will then be to show that for $i\in [\tau]$, the  edge $e_i$ is added at time $i$ and so  $M_H(n)\geq \tau$ where $n=v(G_0)$ is the number of vertices of the starting graph.

 Before discussing this approach in general, we now introduce  the notion of a \emph{simple} $H$-chain, which is a chain in which the edges $e_i$ are the only intersections of the copies $H_i$ of $H$.

\begin{dfn}[Simple chains] \label{def:simple-chain}
Let $4\leq k \in \NN$, $\tau \in \NN$   and suppose that $H$ is a graph with  $V(H)=[k]$  such that $e':=\{1,2\}$ and $f':=\{k-1,k\}$ are both edges of $H$. A \emph{simple} $H$-chain of length $\tau$ is a chain $(H_i,e_i)_{i\in [\tau]}$  on vertex set $V=\{v_1,\ldots,v_{\tau(k-2)+2}\}$ such that for $i\in [\tau]$, $H_i$ is the image of the graph isomorphism  $\varphi_i : H\rightarrow H_i$ defined via the map $\varphi_i(j)=v_{(i-1)(k-2)+j}$ and $e_i=v_{i(k-2)+1}v_{i(k-2)+2}\in E(H_i)$. For $i\in[\tau-1]$ we have that  $e_i=\varphi_{i}(f')=\varphi_{i+1} (e')=V(H_i)\cap V(H_{i+1})$ and  $e_\tau=\varphi_\tau(f')$. 
	\end{dfn}

An example of a simple chain of length 4 when $H$ is $K_4$ is given in Figure \ref{fig:k4chain} with the edges $e_i$ represented by dashed lines.

    \begin{figure}[h]
    \centering
  \includegraphics[scale=1]{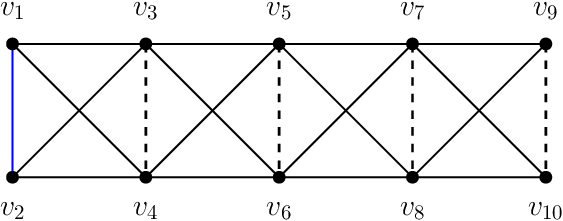}
    \caption{   \label{fig:k4chain} A simple $K_4$-chain.} 
  \end{figure}

	\begin{rem} 
	Note that the definition of a simple chain requires $H$ to have an ordering on its vertices and different orderings of $V(H)$ may produce different chains. The existence of a pair of vertex-disjoint edges  in $H$ (as given in Observation \ref{obs: disjoint edges} for inseparable graphs) is necessary to give a \emph{valid} ordering. 
      In most cases though,  the choice of valid ordering used to give the simple chains  will have no bearing on the following arguments and any choice will suffice.  
	\end{rem}
 For a simple $H$-chain  $(H_i, e_i)_{i\in[\tau]}$ of length $\tau$, one has  $n:=|V(\bigcup_{i=1}^\tau H_i)|= \tau (v(H)-2)+2$ vertices. For many graphs $H$, taking a simple chain and removing the edges $e_i$ can give a starting graph $G_0$ with linear running time. This is the case, for example, with cliques of size at least 4 (as in Figure \ref{fig:k4chain}), and in fact for all \emph{inseparable} graphs $H$, as we will see shortly (Lemmas \ref{lem:chain_runningtime} and \ref{lem:inseparableimpliesproper}). 

\subsection{Proper chains} \label{sec:proper chains} Our main aim with chain constructions is to provide \emph{super-linear} bounds on running times. Therefore, we cannot require such a strict intersection condition as in Definition \ref{def:simple-chain} and the copies $H_i$ of $H$ in our chain will have to overlap more significantly in order to give constructions achieving longer running times. 
When considering such chains,  we will need extra conditions that guarantee  no additional edges are added that interfere with the desired percolation along the chain.  
This leads us to the following definition of \emph{proper} chains, recalling that $\final{G}_H$ denotes the graph at which the $H$-process on $G$ stabilises.

\begin{dfn}[Proper chains] \label{def:proper chain}
	An $H$-chain $(H_i, e_i)_{i\in [\tau]}$ with an additional assigned edge $e_0\in E(H_1)\setminus \{e_1\}$ is called \emph{proper} if:
		\begin{enumerate}[label=(\arabic*)]

\item \label{chain:proper1} $e_i \notin E(\final{H_j}_H)$ for $1\leq j < i \leq \tau$;
		\item \label{chain:proper2} for any $e\in E(H)$ and copy $F$ of $H-e$ in $\bigcup_{i=1}^\tau \final{H_i}_{H}$,  there is a unique $j\in [\tau]$ such that $F \subseteq \final{H_j}_H$. Furthermore,  $F \not\subseteq \left(\bigcup_{i\in[\tau]\setminus j}\final{H_i}_H \bigcup  H_j \right)-\{ e_{j-1},e_j \}$.
		\end{enumerate}
        
\end{dfn}

We will refer to the additional assigned edge $e_0$ in a proper chain as the  \emph{origin} of the chain. This is an edge whose addition will kickstart the percolation along the chain.  For the sake of brevity, we will often suppress the mention of the origin, especially if it has no bearing on the statement or argument in question. When referring to a proper chain or proving that a chain is proper or not, it is taken to be implicit that it has an origin defined. In particular, when using simple chains as in Definition \ref{def:simple-chain},   we will always take a canonical choice of origin as $e_0:=v_1v_2 =\phi_1(e')$. In Figure \ref{fig:k4chain}, this choice is highlighted in blue. However, note that the chain $(H_i,e_i)_{i\in[4]}$ in Figure \ref{fig:k4chain} (and indeed any simple $K_4$-chain) is \emph{not} actually proper as, for example, the vertex set $\{v_1,v_3,v_4,v_5\}$ hosts a copy of $K_4^-$ that is  not contained in one of the $H_i=\final{H_i}_{K_4}$, violating condition \ref{chain:proper2}. This example shows that it is not necessary  for an $H$-chain $(H_i, e_i)_{i\in [\tau]}$ with underlying graph $G=\bigcup_{i=1}^\tau H_i$ to be proper in order to have $\tau_H(G-\{e_1,\ldots,e_\tau \}) \geq \tau$. Indeed one can check that a simple $K_4$-chain $(H_i,e_i)_{i\in[\tau]}$ as in Figure \ref{fig:k4chain} adds the edge $e_i$ at time step $i$ for $i\in [\tau]$.

However, the conditions of proper chains are natural and allow more refined control over the percolation process, which will be essential for our purposes.  Condition \ref{chain:proper1} ensures that an edge $e_j$ is not added prematurely before step $j$ in the process. Condition \ref{chain:proper2} of Definition \ref{def:proper chain} guarantees that the only edges added during the process are those that we expect from running the process individually on each $H_i$ (see Lemma \ref{lem:final graph proper} below) so we can be sure that there are no unwanted edges being added that will accelerate the process.  Finally the `furthermore' part of condition \ref{chain:proper2} implies that for each $j\geq 2$, the process on $H_j$ is not triggered until the addition of the edge $e_{j-1}$. 

From this point onwards we will deal exclusively with proper chains and we begin by proving properties of the percolation processes based on these chains. 

\begin{lem} \label{lem:final graph proper}
    Let $(H_i,e_i)_{i\in[\tau]}$ be a proper $H$-chain with underlying graph $G := \bigcup_{i=1}^\tau H_i$. Then we have that  $\final{G}_H=\bigcup_{i=1}^\tau \final{H_i}_H$.
\end{lem}
\begin{proof}
    Certainly we have that $\final{H_i}_H\subseteq \final{G}_H $ for all $i\in [\tau]$ due to Observation \ref{obs:hom}. For the other inclusion, note that $G\subseteq \bigcup_{i=1}^\tau \final{H_i}_H$ and so by Observation \ref{obs:final} it suffices to establish that $\bigcup_{i=1}^\tau \final{H_i}_H$ is $H$-stable, which follows from condition \ref{chain:proper2} of Definition \ref{def:proper chain}. 
\end{proof}

We now show that proper chains can be used to give lower bounds on running times. 

	\begin{lem}[Running times for proper chains]
	    \label{lem:chain_runningtime}
	Let $(H_i,e_i)_{i\in[\tau]}$ be a proper $H$-chain with underlying graph $G := \bigcup_{i=1}^\tau H_i$.
	Then the  graph $G_0:=G - \{ e_1,\ldots,e_\tau \}$ satisfies $\tau_H(G_0) \geq \tau$, and therefore
		\[
		M_H\left( v(G) \right)\geq \tau.
		\]
	\end{lem}
	\begin{proof}
	The lemma follows from the fact that in the $H$-process $(G_t)_{t\geq 0}$ on $G_0=G -\{ e_1,\ldots,e_\tau \}$ we have
	 \begin{linenomath} 	\begin{equation*}
		E(G_t)\cap \{ e_1,,\ldots,e_\tau \} = \{ e_1,\ldots,e_t\}
		\end{equation*} \end{linenomath}
	for every $0\leq t \leq \tau$.
	To prove this fact we induct on $t$ to show the following stronger statement
     \begin{linenomath}
		\begin{equation}\label{eq:ladder_time}
			E(G_t)\cap \{ e_1,,\ldots,e_\tau \} = \{ e_1,\ldots,e_t\} \quad\text{ and }\quad G_t\;\subseteq\; \bigcup_{i=1}^{t}\final{H_i}_H \:\cup\: \bigcup_{i=t+1}^\tau H_i.
		\end{equation}
        \end{linenomath}
	The case $t=0$ follows immediately from the definition of $G_0$.
	Let $t\geq 1$.
	The induction hypothesis yields
     \begin{linenomath}
		\begin{equation}\label{eq:ladder_induction1}
		E(G_{t-1})\cap \{ e_1,\ldots,e_\tau \} = \{ e_1,\ldots,e_{t-1}\}
		\end{equation}
        \end{linenomath}
	and
     \begin{linenomath}
		\begin{equation}\label{eq:ladder_induction2}
		G_{t-1}\;\subseteq\; \bigcup_{i=1}^{t-1}\final{H_i}_H \:\cup\: \bigcup_{i=t}^\tau H_i.
		\end{equation}
        \end{linenomath}
	Now property \ref{chain:proper1} of Definition \ref{def:H-chain} implies that $E(H_t)\cap \{e_t,\ldots,e_\tau\}=\{e_t\}$ and so  $H_t-e_t\subseteq G_0 + \{e_1,\ldots,e_{t-1}\} \subseteq  G_{t-1}$, where \eqref{eq:ladder_induction1} was used in the second subgraph inclusion. Therefore we have that  $e_t\in E(G_t)$, as required.
	Also property \ref{chain:proper2}  of Definition \ref{def:H-chain} and \eqref{eq:ladder_induction2} guarantee that every copy of $H$ completed at time $t$ lies in $\final{H_j}_H$ for some $j\in[\tau]$, while \eqref{eq:ladder_induction1}, \eqref{eq:ladder_induction2} and the `furthermore' part of  property \ref{chain:proper2} imply that $j\leq t$ for any such $j\in [\tau]$.
	Therefore, $G_t \subseteq \bigcup_{i=1}^{t}\final{H_i}_H \cup \bigcup_{i=t+1}^{\tau} H_i$.

	Finally, we have to make sure that $e_{i'}\notin E(G_t)$ for $t < i' \leq \tau$.
	Suppose for a contradiction there was some $i'\in[t+1,\tau]$ with $e_{i'}\in E(G_t)$.
	Then $e_{i'}$ lies in $E(G_t)\setminus E(G_{t-1})$ because of \eqref{eq:ladder_induction1}.
	Let $H'$ be a copy of $H$ completed by $e_{i'}$ at time $t$. In the previous paragraph, we showed that any copy of $H$ completed at time $t$, in particular $H'$, lies in $\final{H_j}_H$ for some $j\leq t$. 
	However, property \ref{chain:proper1} of being proper gives $e_{i'}\notin E(\final{H_j}_H)$ so we have arrived at the desired contradiction.
	Thus $\{ e_{t+1},\ldots,e_\tau \} \cap E(G_t) = \varnothing$, concluding the induction step.
	\end{proof}

Lemma \ref{lem:chain_runningtime} shows that proper chains can give lower bounds on running times. However  it is not clear that proper $H$-chains (of arbitrary length) exist for all $H$. Indeed, for example, there are no such chains if $\delta(H) = 1$ because many unwanted copies of $H-e$  can be found  when $e$ is taken to be the edge incident to the degree 1 vertex. We also saw that simple $K_4$-chains
 as in Figure \ref{fig:k4chain} are not proper. We close this section by showing that, for many graphs $H$, simple chains \emph{do} give examples of proper chains. First, we show this under the assumption that $H$ is inseparable. As simple chains exist for inseparable graphs (Observation \ref{obs: disjoint edges}), this gives linear sized proper chains.

	\begin{lem}\label{lem:inseparableimpliesproper}
	If $H$ is $(2,1)$-inseparable or $(1,1,1)$-inseparable then any simple $H$-chain is proper. 
	\end{lem}
	\begin{proof}
	Let $(H_i, e_i)_{i\in[\tau]}$ be a simple $H$-chain with origin $e_0$, and let $\hat G := \bigcup_{i=1}^\tau \final{H_i}_H$. Condition \ref{chain:proper1} of Definition \ref{def:proper chain} is clearly satisfied by the definition of the simple chain. 
	Now let $e\in E(H)$, and let $F$ be a copy of $H-e$ in $\hat G$.
	Suppose that $V(F)$ does not coincide with a single $V(H_j)$.
	Pick the smallest $j\in[\tau-1]$ such that $V(F)\cap V(H_j)\setminus (V(H_{j+1})\cup\ldots\cup V(H_\tau) )\neq \varnothing$.
	Since the intersection of $\bigcup_{i=1}^j \final{H_i}_H$ and $\bigcup_{i=j+1}^\tau \final{H_i}_H$ is precisely the edge $e_j$, we have that removing the two vertices of the edge $e_j$ disconnects $F$, contradicting that $H$ is inseparable. In the case that $H$ is $(1,1,1)$-inseparable, we used here that $\hat G$ is also bipartite by Lemma \ref{lem:staying bipartite} and the embedding of $F$ must respect the partition of $\hat G$, which in turn implies that the vertices of $e_j$ are indeed in different parts of the bipartition of $F$. So  
 we must have that $V(F)\subseteq V(H_j)$.
	All edges in $\hat G[V(H_j)]$ lie in $\final{H_j}_H$ because the chain is simple, so $F \subseteq \final{H_j}_H$, establishing the first part of condition \ref{chain:proper2} in Definition \ref{def:proper chain}.
	Finally, in the graph $G':=\left(\bigcup_{i\in [\tau]\setminus j}\final{H_i}_H \bigcup  H_j \right)-\{ e_{j-1},e_j \}$ there are only $e(H)-2$ edges with both endpoints in $V(H_j)$.
	Therefore we cannot have that $F\subset G'$ and  the `furthermore' part of condition \ref{chain:proper2} in  Definition \ref{def:H-chain} is also satisfied. This completes the proof. 
	\end{proof}

We remark that the definitions of $(2,1)$-inseparable graphs (Definition \ref{def:inseparable}) and $(1,1,1)$-inseparable  graphs (Definition \ref{def:bip insep}) stem precisely from Lemma \ref{lem:inseparableimpliesproper} as the conditions that arise when trying to prove properness of simple chains. 
However, even when a graph $H$ is not inseparable, it is still possible to have proper chains. This is the case for the wheel graph, as we now show. 

\begin{lem} \label{lem:wheel simple}
    For $7\leq k\in \mathbb{N}$ and let $W_k$ be the wheel graph with  $V(W_k)=[k+1]$ such that $v\in \{3,\ldots,k-1\}$ is a universal hub vertex and  $W_k-v$ is a cycle of length $k$ containing edges $\{1,2\}$ and $\{k,k+1\}$. Then for any $\tau\in \mathbb{N}$, the simple $H$-chain $(H_i,e_i)_{i\in [\tau]}$ of length $\tau$ is proper. 
\end{lem}
\begin{proof}
    Recall that $W_k$ is self-stable (Observation \ref{obs:self-stable}). 
    Now consider the simple chain $(H_i,e_i)_{i\in [\tau]}$ as given by Definition \ref{def:simple-chain} with $H=W_k$, the vertex ordering as in the statement of the lemma and the canonical choice of origin $e_0=\{v_1,v_2\}$. We certainly have that $e_i\notin E(\final{H_j}_H)=E(H_j)$ for $1\leq j<i\leq \tau$ by construction and so the chain satisfies condition \ref{chain:proper1} of being proper (Definition \ref{def:proper chain}). For condition \ref{chain:proper2}, fix some $e\in E(H)$ and a copy $F$ of $H-e$ in $G=\bigcup_{i=1}^\tau\final{H_i}_H=\bigcup_{i=1}^{\tau}H_i$. We need to show that there is a unique $j\in [\tau]$ such that $F\subseteq \final{H_j}_H=H_j$ and as $G[V(H_j)]=H_j$, it suffices to show that $V(F)\subseteq V(H_j)$. Let $v'$ be the vertex in $F$ corresponding to the universal vertex $v$. As $d_F(v')\geq d_H(v)-1\geq 6$, we have that there is some $j\in [\tau]$ such that $v'$ is the universal vertex of $H_j$. Indeed all other vertices in $G$ have $d_G(v)\leq 5$. If $v\notin e$ then all other vertices of $F$ are adjacent to $v'$ and so lie in $V(H_j)$. If $v\in e$, we have that there is some $w\in V(F)$ such that $F-v'-w$  is a path $P$ with $k-1$ vertices, all of which are adjacent to $v'$ in $G$. Thus $V(P)\subseteq V(H_j)$ and $P$ is a subpath of the cycle $H_j-v'$. This implies that $w\in V(H_j)$ as no vertex outside of $V(H_j)$ can complete a cycle with $P$.  This proves that $V(F)\subseteq V(H_j)$ and so $F\subseteq H_j$. The `furthermore' part of condition \ref{chain:proper2} for being proper follows because $(G-\{e_{j-1},e_j\})[V(H_j)]$ has $e(H)-2=2k-2$ edges and so cannot contain a copy of $F$. \end{proof}

    \section{Linking collections of chains} \label{sec:linking}

In order to get  lower bounds for running times of graph bootstrap processes, we can use Lemma \ref{lem:chain_runningtime} which reduces the question to exhibiting long proper chains. In this section we  show how to construct such chains by \emph{linking} a collection of smaller proper chains. The birth of this idea is due to Balogh, Kronenberg, Pokrovskiy and Szab\'o \cite{balogh2019maximum} who  linked chains to prove their lower bound on the running time of the $K_5$-process. As with the previous section, our aim here is to axiomatise this procedure in order to produce general results with wide applicability. We begin by extending the definition of a proper chain (Definition \ref{def:proper chain}) to a collection of chains.

\begin{dfn}[Proper collection of chains] \label{def:proper chaincoll}
	Let $A$ be an indexing set, $V$ a vertex set and $\tau\in \mathbb{N}$. Suppose that for each $a\in A$ we have a proper $H$-chain $(H^a_i, e^a_i)_{i\in [\tau]}$   with origin $e_0^a$ and   underlying graph $G^a=\bigcup_{i\in [\tau]}H^a_i$ on $V$.   We say that the collection  $\cH=\{(H^a_i, e^a_i)_{i\in [\tau]}:a\in A\}$ is \emph{proper} if:
		\begin{enumerate}[label=(\roman*)]
\item \label{chaincoll: edges not early} for any  $a\neq b\in A$ and $i\in \{0,\ldots,\tau\}$ we have that  $e_i^a\notin \final{G^b}_H$;
\item \label{chaincoll:H-e new 1} for any $e\in E(H)$ and copy $F$ of $H-e$ in $\bigcup_{a\in A} \final{G^a}_{H}$,  there exists a unique $b\in A$ such that $F\subseteq \final{G^b}_H$.
Furthermore, for the unique $j\in [\tau]$ such that $F \subseteq \final{H_j^b}_H$, we have that  $F\nsubseteq \left( \bigcup_{a\in A\setminus b} \final{G^a}_H\bigcup_{i\in [\tau]\setminus j} \final{H_i^b}_H \bigcup H^b_j\right)-\{e^b_{j-1},e^b_j\}$.
		\end{enumerate}
        \end{dfn}
        
Note that due to Lemma \ref{lem:final graph proper}, we have that for each $a\in A$, the final graph  of the chain $(H^a_i,e^a_i)_{i\in [\tau]}$  has the form $\final{G^a}_H=\bigcup_{i=1}^\tau \final{H_i^a}_H$.  Note also that in condition \ref{chaincoll:H-e new 1}, we implicitly use that as $F\subseteq \final{G^b}_H=\bigcup_{i=1}^\tau \final{H_i^b}_H$ and the chain $(H^b_i,e^b_i)_{i\in [\tau]}$ is proper, there is some unique $j\in [\tau]$ such that $F\subseteq \final{H_j^b}_H$. 
We will soon see how to construct a single long proper chain  from a proper collection of chains. First though, we derive some implications of the definition. We begin by showing that the second condition  simplifies when $H$ is self-stable.

\begin{rem}\label{rem:self stable coll}
	If  $\final{H}_H = H$, then condition \ref{chaincoll:H-e new 1}  can be replaced by the following  condition:
 \begin{enumerate}[label=(\roman**)]
 \addtocounter{enumi}{1}
 \item \label{chaincoll:proper 2 star} for any $e\in E(H)$ and copy $F$ of $H-e$ in $\bigcup_{a\in A} G^a$,  there is a unique $b\in A$  such that $F \subseteq G^b$.
 \end{enumerate}
\end{rem}
\begin{proof} 
    Using that $\final{H}_H=H$, we have that $\final{G^a}_H=\bigcup_{i=1}^\tau\final{H_i^a}_H=\bigcup_{i=1}^\tau H^a_i=G^a$ for each $a\in A$. 
 Thus the first part of condition \ref{chaincoll:H-e new 1} reduces to \ref{chaincoll:proper 2 star} when $H$ is self-stable. To see that the `furthermore' part of condition \ref{chaincoll:H-e new 1} is redundant when $\final{H}_H=H$, note that if our copy $F$ of $H-e$ is contained in $\final{H^b_j}_H=H_j^b$ and $F\subseteq \left( \bigcup_{a\in A\setminus b} \final{G^a}_H\bigcup_{i\in [\tau]\setminus j} \final{H_i^b}_H \bigcup H^b_j\right)-\{e^b_{j-1},e^b_j\}$, then we have that $F\subseteq H^b_j-\{e^b_{j-1},e^b_j\}$ which is impossible as $H^b_j-\{e^b_{j-1},e^b_j\}$ does not contain enough edges. 
\end{proof}

For some of our constructions, it will be convenient for us to work with stronger conditions. We make the following definition of being \emph{strongly proper.} 

\begin{dfn}[Strongly proper collection of chains] \label{def:stronger-proper}
Let $\cH=\{(H^a_i, e^a_i)_{i\in [\tau]}:a\in A\}$ be a collection of proper $H$-chains such that for each $a\in A$, the chain  $(H^a_i, e^a_i)_{i\in [\tau]}$  has  underlying graph $G^a=\bigcup_{i\in [\tau]}H^a_i$.   We say that the collection  $\cH$ is \emph{strongly proper} if the following  hold:
     \begin{enumerate}[label={(\Roman*)}]
     \item \label{chaincol:strong proper edge}
     the graphs $\{\final{G^a}_H:a\in A\}$ are pairwise edge-disjoint.
 \item \label{chaincoll:strong proper H-e} for any $e\in E(H)$ and copy $F$ of $H-e$ in $\bigcup_{a\in A} \final{G^a}_H$,  there is a  $b\in A$  such that $F \subseteq \final{G^b}_H$.
 \end{enumerate}
\end{dfn}

As hinted by our terminology, we now show that strongly proper collections are indeed proper.

\begin{lem}
\label{lem:strong implies proper}
Suppose $\cH=\{(H^a_i, e^a_i)_{i\in [\tau]}:a\in A\}$ is a strongly proper collection of  $H$-chains. Then $\cH$ is proper. 
\end{lem}
\begin{proof}
    We need to verify that the conditions \ref{chaincol:strong proper edge} and \ref{chaincoll:strong proper H-e} imply the conditions \ref{chaincoll: edges not early} and \ref{chaincoll:H-e new 1}. So let us fix some strongly proper collection $\cH=\{(H^a_i, e^a_i)_{i\in [\tau]}:a\in A\}$ (and an implicit choice of origin $e_0^a$ for each $a\in A$). Firstly, if $a\neq b\in A$ and $i\in \{0,\ldots,\tau\}$ then we certainly have that $e_i^a\in H_i^a\subset \final{G^a}_H$ (or $e_i^a\in H_{i+1}^a$ if $i=0$) and so by \ref{chaincol:strong proper edge} there is no other $b\in A\setminus a$ such that $e_i^a\in \final{G^b}_H$, verifying \ref{chaincoll: edges not early}. 

Next, suppose that $e\in E(H)$ and $F$ is a copy of $H-e$ in $\bigcup_{a\in A}\final{G^a}_H$. Then by \ref{chaincoll:strong proper H-e}, we have that there is a  $b\in A$ such that $F\subseteq \final{G^b}_H$ 
and this choice of $b\in A$ is unique by \ref{chaincol:strong proper edge} establishing the first part of \ref{chaincoll:H-e new 1}. 
For the `furthermore' part of condition \ref{chaincoll:H-e new 1}, taking the $j\in [\tau]$ such that $F\subset \final{H_j^b}$, we have that no edge of $F$ lies in $\bigcup_{a\in A\setminus b}\final{G^a}_H$ by condition \ref{chaincol:strong proper edge}. Therefore if the `furthermore' condition of \ref{chaincoll:H-e new 1} is violated, then $F\subset \left(\bigcup_{i\in [\tau]\setminus j}\final{H_i^b}_H\bigcup H_j^b\right)-\{e^b_{j-1},e^b_j\}$ which contradicts condition \ref{chain:proper2} of the   chain $(H^b_i, e^b_i)_{i\in [\tau]}$ being proper (Definition \ref{def:proper chain}). We therefore conclude that $\cH$ is indeed proper. 
\end{proof}

\subsection{Long chains from proper collections} \label{sec:long}
We are now ready to form long chains  from proper collections of chains via linking. The idea here is simple. Taking a proper collection of chains $\cH=\{(H_i^a,e_i^a)_{i\in [\tau]}:a\in A\}$   on a vertex set $V$ (with origins $e_0^a$), we introduce an order on $A$ and for each pair of consecutive $a$ and $b$ in $A$, we introduce a new chain $(J^{a}_\ell,e_\ell)_{\ell\in [2k]}$  which runs between $H^a_{\tau}$ and $H^b_1$ and whose vertices are new apart from at the edge that connects to $H^a_{\tau}$ (which will be $e_\tau^a$) and the edge that connects to $H^b_1$ (which will be $e_0^b$). Moreover, we guarantee that all the new vertices added for distinct $J_\ell^{a}$ are disjoint. This results in one long chain with length $|A|\tau+ 2k\cdot (|A|-1)\geq |A|\tau$. We have also added roughly  $2v(H)k|A|$ new vertices for these linking chains but this will be $O(|V|)$ in applications. Finally, we will use that $\cH$ was a proper collection to derive that this new long chain is also proper.

\begin{lem}[Long chains from proper collections] \label{lem:linking chains}
     Let $H$ be an $(2,1)$-inseparable graph, $ \tau\leq n \in \mathbb{N}$  and  $V$  an $n$-vertex set. Suppose that $\cH=\{(H^a_i,e^a_i)_{i\in [\tau]}:a\in A\}$ is a proper collection of  $H$-chains  on the vertex set $V$. 
	Then there exists a proper $H$-chain of length at least $\tau |A|$ on at most $2v(H)^2 n$ vertices.
Moreover, the same holds if $H$ is $(1,1,1)$-inseparable and $\bigcup_{a\in A,i\in [\tau]} H^a_i$ is bipartite on $V$. 
\end{lem}
\begin{proof}
For $a\in A$, let $e_0^a$ be the origin of the chain $(H_i^a,e^a_i)_{i\in[\tau]}$.   We begin by (arbitrarily) introducing an order on $A$ and so we can assume that $A=[s]$ with $s=|A|\in \NN$.  We also fix $k:=v(H)$, let $e',f'\in E(H)$ be a pair of disjoint edges in $H$ (Observation \ref{obs: disjoint edges}) and introduce an order on  $V(H)=[k]$ such that 
    $e'=\{1,2\}$ and $f'=\{k-1,k\}$. If $H$ is bipartite, we ensure that vertices $1$ and $k-1$ are in the same part of $H$ with vertices $2$ and $k$ in the other part.
    
    For $a\in [s-1]$,  let $U^a$ be  a new vertex set disjoint from $V$ with $|U^a|=2k(k-2)-2$ so that $U^a\cap U^b=\emptyset$ for $a\neq b\in [s-1]$. We label $U^a$ as $U^a=\{u^a_3,\ldots,u^a_{2k(k-2)}\}$ and also label the vertices of $e^a_\tau$ as $\{u^a_1,u^a_2\}$ and the vertices of $e^{a+1}_0$ as $\{u^a_{2k(k-2)+1},u^{a}_{2k(k-2)+2}\}$. If $G'=\bigcup_{a\in A,i\in [\tau]}H_i^a$ is bipartite (and thus $H$ itself is bipartite), we ensure that the labelling of $e^a_\tau$ and $e_0^{a+1}$ is such that $u^a_1$ and $u^a_{2k(k-2)+1}$ are in the same part of the bipartition of $G'$ (and thus $u_2^a$ and $u^{a}_{2k(k-2)+2}$ are in the other part). Finally for $a\in [s-1]$ and $\ell\in [2k]$ we define $J^a_\ell$ to be the image of the graph isomorphism  $\varphi^a_\ell : H\rightarrow J^a_\ell$ defined via  $\varphi^a_\ell(j)=u^a_{(\ell-1)(k-2)+j}$ for $j\in [k]$ and fix  $f^a_\ell:=\phi^a_\ell(f')=u^a_{\ell(k-2)+1}u^a_{\ell(k-2)+2}\in E(J^a_\ell)$. 

    Note   that for $a\in [s-1]$, the collection $(J_\ell^a,f_\ell^a)_{\ell \in [2k]}$ defines a  simple $H$-chain (with respect to the ordering on $V(H)= [k]$ given above)   if $e^{a}_\tau\cap e^{a+1}_0=\emptyset$.  
This is not the case if $e^{a}_\tau$ and  $e^{a+1}_0$ intersect but we certainly have  that both $(J_\ell^a, f_\ell^a)_{\ell=1}^{2k-1}$ and $(J_\ell^a,f_\ell^a)_{\ell=2}^{2k}$ are both  simple $H$-chains. 
 
 We are now in a position to define our  long chain $(H^*_i,e^*_i)_{i\in [\tau^*]}$ with $\tau^*:= \tau s+2k(s-1)\geq \tau |A|$. This chain will run through each  of the chains $(H^a_i,e^a_i)_{i\in [\tau]}$ in order of the $a\in A=[s]$ and in between $(H^a_i,e^a_i)_{i\in [\tau]}$ and $(H^{a+1}_i,e^{a+1}_i)_{i\in [\tau]}$ the chain $(H^*_i,e^*_i)_{i\in [\tau^*]}$ will run through $(J^a_\ell,f^a_\ell)_{\ell\in [2k]}$. Formally we define 
 \begin{linenomath}
\begin{equation}
(H^*_i,e^*_i)=\begin{cases}
			(H^a_{i'},e^a_{i'}), & \text{ if } i=(a-1)(\tau+2k)+i' \text{ with } a\in [s], i'\in [\tau]; \text{ or } \\
            (J^a_{\ell},f^a_{\ell}), & \text{ if } i=(a-1)(\tau+2k) + \tau + \ell \text{ with } a\in [s-1], \ell\in [2k].
		 \end{cases}
\end{equation}
\end{linenomath}
We also define $e_0^*$ to be the origin $e_0^1$ of the first chain  $(H^1_i,e^1_i)_{i\in [\tau]}$ in the collection.   With this, it follows from  the fact that each of the $(H^a_i,e^a_i)_{i\in [\tau]}$, $(J_\ell^a, f_\ell^a)_{\ell=1}^{2k-1}$ and $(J_\ell^a,f_\ell^a)_{\ell=2}^{2k}$ are themselves chains, that $e^*_{i}\in E(H_i^*)\cap E(H^*_{i+1})$ for each $i\in [\tau^*-1]$. Thus $(H^*_i,e^*_i)_{i\in[\tau^*]}$  is indeed a chain of length $\tau^*\geq \tau s$ with origin $e_0^*$. The underlying graph $G^*:=\bigcup_{i\in [\tau^*]}H_i^*$ has vertex set $V\cup(\bigcup_{a=1}^{s-1}U^a)$ of size $n+(\tau-1)(2k(k-2)-2)\leq n+2(\tau-1)k^2\leq 2k^2 n$ using here that $\tau\leq n$. Moreover, note that if  $H$ is bipartite and so is $\bigcup_{a\in A,i\in [\tau]} H^a_i$, then $G^*$ is also bipartite due to how we labelled the vertices of $H$ and the edges $e_\tau^a$ and $e_0^{a+1}$ for $a\in [s-1]$.  

 It remains to show that this new chain is proper. For condition \ref{chain:proper1}  of Definition \ref{def:proper chain} we fix some $1\leq j\neq i\leq \tau^*$ and suppose that 
 $e_{i}^*\in E(\final{H^*_{j}}_H). $
 Then we need to show that $j>i$. 
 Suppose first that $e_{i}^*=f_{\ell}^a$ for some $a\in [s-1]$ and $\ell\in [2k-1]$. As $e_{i}^*\in E(\final{H^*_{j}}_H)$ we must have that $H^*_j$ is equal to  $J^a_{\ell+1}$  as this is the only copy of $H$ amongst the $H_i^*$ (other than $J^a_{\ell}$) whose vertices contain the edge $f_\ell^a$. Hence $j=i+1$ and we are done. 
 
 If $e_i^*\neq f_{\ell}^a$ for some $a\in [s-1]$ and $\ell\in [2k-1]$, then we necessarily have that $e_i^*=e_{i'}^a$ for some $i'\in\{0,\ldots,\tau\}$ and $a\in [s]$. Therefore by property \ref{chaincoll: edges not early} of the collection being proper $e_i^*\notin \final{H^b_{i''}}$ for any $b\neq a$ and $i''\in [\tau]$. In 
 particular, we have that $e_{i'}^a\notin\{e^b_0,e^b_{\tau}:b\neq a\}$.
 Therefore as all of the  $J^{a'}_\ell$ are either vertex disjoint from $V$ or intersect $V$ in some  $e^{b}_0$ or $e^{b}_\tau$, we have that if $H^*_j=J^{a'}_\ell$ for some $a'\in [s-1]$, $\ell\in [2k]$ then either $e_{i'}^a=e_0^a$ and $H^*_j=J^{a-1}_{2k}$ or $e_{i'}^a=e_\tau^a$ and $H^*_j=J^a_{1}$. The former case is ruled out as this would imply that $j=i$  and in the latter case $j=i+1$.   
 Finally, we have to  consider the case that $H_j^*$ is   $H^a_{i''}$ for some $i''\in [\tau]$. Then we have that $i''>i'$ on account of property \ref{chain:proper1} of the chain $(H^a_i,e^a_i)_{i\in [\tau]}$ being proper (Definition \ref{def:proper chain}). Thus $j>i$ and we are done.

   To finish the proof, we need to establish condition \ref{chain:proper2} for the chain $(H^*_i,e^*_i)_{i\in [\tau^*]}$ so fix some $e\in E(H)$ and some copy $F$ of $H-e$ in the graph $\hat G:=\bigcup_{i\in [\tau^*]}\final{H^*_i}_H$.  We will begin by showing that either $V(F)\subset V$ or $V(F)\subset U^a\cup W^a$ for some $a\in [s-1]$ and $W^a\in \{e^{a}_\tau,e^{a+1}_0\}$. Indeed, suppose for a contradiction that this does not hold. Then as $V(F) \nsubseteq V$, there exists some $a\in [s-1]$ and vertex $u\in V(F)\cap U^a$. Suppose first  that $u\in U^a_0:=\{u_3,\ldots,u_{k(k-2)+2}\}$. As $V(F)\nsubseteq U^a \cup \{e^a_\tau\}$, there is some vertex $y\in V(F) \cap (V(\hat G)\setminus (U^a\cup \{e^a_\tau\}))$. We claim that $F-\{u_1,u_2\}$ is disconnected, with $u$ and $y$ in different components, recalling here that $e^a_\tau=\{u_1,u_2\}$. Indeed, we have that \[\dist_{\hat G-\{u_1,u_2\}}(u,y)\geq \min\left\{\dist_{\hat G-\{u_1,u_2\}}(u,w):{w\in e_0^{a+1}}\right\},\] as any path from $u$ to $y$ in $\hat G-\{u_1,u_2\}$ has to go through $e_0^{a+1}$. The quantity on the right hand side of the inequality is at least $k$ as any path from $y$ to $w\in e_0^{a+1}$ must contain at least one vertex from each set $V(J^a_\ell)\setminus \{f^a_{\ell-1}\}$ for $\ell=k,\ldots,2k$. 
   Therefore as $F-\{u_1,u_2\}\subseteq \hat G-\{u_1,u_2\}$, we indeed have that $F-\{u_1,u_2\}$ is disconnected, contradicting the $(2,1)$-inseparability of $H$. If $H$ is $(1,1,1)$-inseparable, we use here that $\hat G$ is also bipartite (from how we defined the underlying graph $G^*$ and  Lemma \ref{lem:staying bipartite})  implying that if $\{u_1,u_2\}\subseteq V(F)$, then the vertices $u_1,u_2$ lie in different parts of the bipartition of $V(F)$. The case that  $u\in U^a\setminus U^a_0=\{u_{k(k-2)+3},\ldots,u_{2k(k-2)}\}$ is analogous and one can see from the same argument that in such a case, removing the vertices of $e_0^{a+1}$ from $F$ disconnects $F$ and contradicts the inseparability of $H$.

We now treat these two cases separately, assuming first that $V(F)\subseteq V$. As $\hat G[V]=\bigcup_{a\in A}\final{G^a}_H$,  by condition \ref{chaincoll:H-e new 1} of the collection being proper, we have that there is a unique $b\in A$  such that $F\subseteq \final{G^b}_H$ and so also a unique $j'\in [\tau]$ such that $F\subseteq \final{H^b_{j'}}_H$. Letting $j=(b-1)(\tau+2k)+j'$, we then have that $j$ is the unique index in $[\tau^*]$ such that $F\subseteq \final{H^*_{j}}_H$. Moreover, note that for any $i\in [\tau^*]\setminus j$, we have that there is some $a\in A$ and $i'\in [\tau]$ such that $E(\final{H^*_{i}})\cap \binom{V}{2}\subseteq E(\final{H^a_{i'}})$. Therefore, we have that $V[\bigcup_{i\in [\tau^*]\setminus j}\final{H^*_i}_H\bigcup H^*_j]=\bigcup_{a\in A\setminus b} \final{G^a}_H\bigcup_{i\in [\tau]\setminus j} \final{H_i^b}_H \bigcup H^b_{j'}$ and the `furthermore' part of condition \ref{chain:proper2} of being proper is implied by the `furthermore' condition \ref{chaincoll:H-e new 1} of the collection being proper.

Finally, suppose that there is some $a\in [s-1]$ and $W^a\in \{e_\tau^a,e_0^{a+1}\}$ such that  $V(F)\subseteq U^a\cup W^a$. In this case, 
we claim that either $V(F)\cap \{u_3,\ldots,u_{k-2}\}=\emptyset$ or $V(F)\cap \{u_{(2k-1)(k-2)+3}, \ldots,u_{2k(k-2)}\}=\emptyset$. Indeed, if this is not the case then $F-W^a$ is disconnected and this contradicts the inseparability of $H$, using here that the shortest path in $\hat G[U^a]$ between a vertex in
$\{u_3,\ldots,u_{k-2}\}$ and a vertex in $\{u_{(2k-1)(k-2)+3}, \ldots,u_{2k(k-2)}\}$ is of length at least $2k> v(H)$. Hence, we have that either  $V(F)\subseteq Y_1:= \bigcup_{\ell=1}^{2k-1} V(J_\ell^a)$ or $V(F)\subseteq Y_2:=\bigcup_{\ell=2}^{2k}V(J_\ell^a)$. Due to the fact that  $\hat G[Y_1]= \bigcup_{\ell=1}^{2k-1} \final{J_\ell^a}_H$ and $\hat G[ Y_2]= \bigcup_{\ell=2}^{2k} \final{J_\ell^a}_H$, 
we have that condition \ref{chain:proper2} for $F$ in $(G_i^*,e_i^*)_{i\in [\tau^*]}$ follows from condition \ref{chain:proper2} holding for the chains $(J_\ell^a, f_\ell^a)_{\ell=1}^{2k-1}$ and $(J_\ell^a,f_\ell^a)_{\ell=2}^{2k}$ which are both proper as they are simple $H$-chains and $H$ is inseparable (Lemma \ref{lem:inseparableimpliesproper}).
\end{proof}

\begin{rem} \label{rem: chains diff size}
    We remark that the condition that each chain in the proper collection has the same size $\tau$ is unnecessary. Indeed the proof of Lemma \ref{lem:linking chains} can be easily adapted to show that if $\tau_a\in \NN$ with $\tau_a\leq n$ for $a \in A$  and $\cH=\{(H_i^a,e_i^a)_{i\in [\tau_a]}:a\in A\}$ is a collection of chains which is proper (with the conditions of Definition \ref{def:proper chaincoll} adapted to the different length chains in the obvious fashion), then there is a proper chain on at most $2v(H)^2n$ vertices of length at least $\sum_{a\in A}\tau_a$. 
\end{rem}

Lemma \ref{lem:linking chains} will be our main tool in proving lower bounds from collections of chains.  However, in one of our key  applications (namely Theorem \ref{thm:wheel}), our graph $H$ is $(2,1)$-separable. In such a situation, we cannot guarantee that the new chains $(J^{a}_\ell,e^a_\ell)_{\ell\in [2k]}$ introduced in the proof of Lemma \ref{lem:linking chains} to link, do not create unwanted copies of $H-e$ for some $e\in E(H)$. We therefore require some extra conditions on our collection, namely that the ends of each chain are disjoint from all of the other chains. This is the content of the following lemma. 

\begin{lem} \label{lem:linking chains plus}
    Let  $4\leq k\in \NN$ and let $H$ be a graph with $V(H)=[k]$ such that $e':=\{1,2\}$ and $f':=\{k-1,k\}$ are both edges of $H$. Moreover,  fix  $5k< \tau\leq n \in \mathbb{N}$  and let $V$ be   an $n$-vertex set. Suppose  that $\cH=\{(H^a_i,e^a_i)_{i\in [\tau]}:a\in A\}$ is a proper collection of simple  $H$-chains, each with underlying graph $G^a=\bigcup_{i\in [\tau]}H^a_i$ on $V$. Moreover, suppose that 
	\begin{equation}\label{eq:additional}
		\bigcup_{i=1}^{k}V(H^a_i) \cap V(G^b) = \varnothing \text{ and } \bigcup_{i=\tau-k}^\tau V(H^a_i) \cap V(G^b) = \varnothing \text{ for } a\neq b\in A.
		\end{equation}
	Then there exists a proper $H$-chain of length at least $\tau |A|$ on at most $2v(H)^2 n$ vertices.
\end{lem}
\begin{proof}
The proof here follows exactly the same scheme as that of Lemma \ref{lem:linking chains} and the chain $(H^*_i)_{i\in [\tau^*]}$ is defined identically to how we defined it in that proof. In particular, we define the new vertex sets $U^a$ and chains $(J^a_\ell,f_\ell^a)_{\ell\in [2k]}$ for $a\in [s-1]$, using the ordering on $V(H)=[k]$ (and edges $e'$ and $f'$) in the statement of this lemma. For $a\in [s-1]$, let us also define $(J^a_\ell,f_\ell^a)=(H^{a}_{\tau+\ell},e^a_{\tau+\ell})$ for $\ell=-k,\ldots,0$ and $(J^a_\ell,f_\ell^a)=(H^{a+1}_{\ell-2k},e^{a+1}_{\ell-2k})$ for $\ell=2k+1,\ldots,3k$. With this labelling, by choosing appropriately the labelling of $e_\tau^a$ as $\{u_0,u_1\}$ and $e_0^{a+1}$ as $\{u_{2k(k-2)+1},u_{2k(k-2)+2}\}$ when defining the $J^a_\ell$, we can assume that $(J^a_\ell,f^a_\ell)_{\ell=-k}^{3k}$ is a simple $H$-chain of length $4k+1$. Hence  we have that the chain $(J^a_\ell,f^a_\ell)_{\ell=-k}^{3k}$ is also proper, as the longer chains $(H^a_i,e^a_i)_{i\in[\tau]}$ are simple  and are proper by assumption. 

Now condition \ref{chain:proper1} of $(G^*_i,e^*_i)_{i\in [\tau^*]}$ being proper follows directly from the proof of  Lemma \ref{lem:linking chains} (where at no point did we use the $(2,1)$-inseparability of $H$). For condition \ref{chain:proper2}, take some $e\in E(H)$ and a copy $F$ of $H-e$ in $\hat G:=\bigcup_{i\in[\tau^*]}\final{G^*_i}_H$. If $V(F)\subset V$, we again refer to the proof of Lemma \ref{lem:linking chains} which dealt with this case without appealing to the $(2,1)$-inseparability of $H$. So we are left with the case that there is some $a\in [s-1]$ and vertex $y\in V(F)\cap U^a$. We then have that $V(F)\subseteq W:=U^a\bigcup_{i=\tau-k}^\tau V(H_i^a)\bigcup_{i=1}^k V(H^{a+1}_i)$ as the $F$ is connected (as otherwise simple proper $H$-chains would not exist) and all vertices that are within distance $k=v(H)$ from $y$ in $\hat G$ are contained in $W$. Now property \ref{chain:proper2} of Definition \ref{def:proper chain} for $F$ in the chain $(H^*_i,e^*_i)_{i\in [\tau^*]}$ follows from the fact that $\hat G[W]=\bigcup_{\ell=-k}^{3k}\final{J^a_\ell}_H$ and from property \ref{chain:proper2} for $F$ in the proper chain $(J^a_\ell,f^a_\ell)_{\ell=-k}^{3k}$. This completes the proof. 
\end{proof}

\section{Ladder chains} \label{sec:ladder}

Our first general construction of chains are \emph{ladder chains}. This section builds on ideas of Balogh, Kronenberg, Pokrovskiy and Szab\'o \cite{balogh2019maximum} in their work showing a quadratic maximum running time for cliques with at least 6 vertices. They gave one long continuous chain that wrapped back onto itself whereas here we exhibit a collection of chains and appeal to Lemma \ref{lem:linking chains}. This allows for more flexibility in the construction and also simplifies the exposition and analysis. The key idea  is to use a bipartition to restrict the intersection of different chains (or in their case different segments of the long chain). We will partition the vertex set of $H$ into two as $V(H)=U\cup W$ and the vertex set $V$ of our collection of chains will be composed of a sequence of copies $L_1,\ldots,L_M$ of $U$ with consecutive copies joined at a single vertex as well as a sequence of copies $R_1,\ldots,R_M$ of $W$ with consecutive copies linked at a single vertex. Now each  chain $(H^a_i,e_i^a)_{i\in [\tau]}$ in our collection will come with some \emph{slope} $\lambda_a\in \NN$ and the chain will be simple with $H^a_i$ lying on the vertex set $L_i\cup R_{i+\lambda_a}$ for $i\in [\tau]$. The reason that we refer to this as a collection of \emph{ladder} chains is that we think of $\bigcup_{j\in [M]}L_j$ as the left side, $\bigcup_{j\in [M]}R_j$ as the right side and each of the chains $(H^a_i,e_i^a)_{i\in [\tau]}$ for $a\in A$ defining some \emph{ladder} which uses these fixed sides and whose rungs are determined by the slope $\lambda_a$. 

This construction allows us to place many chains on the same vertex set whilst maintaining a strong control on the intersections of the different chains. Indeed, crucial to our analysis is the fact that when considering the graph $\hat G=\bigcup_{a\in A,i\in [\tau]}\final{H_i^a}_H$, we have that for each $j\in [M-1]$ there is a single vertex at the intersection of $L_j$ and $L_{j+1}$, whose removal separates $\hat G[\bigcup_{j\in M}L_j]$ and similarly on the right side.  Using this we can force that any copy of $H-e$ in $\hat G$ has a vertex set of the form $L_j\cup R_{j'}$ for some $j,j'\in [M]$. 
This works well in the case that $H$ is highly connected and in this case we can  derive quadratic bounds on maximum running times. This is demonstrated  for $(\ceil{k/2},1)$-inseparable $k$-vertex graphs  in Section \ref{sec:lad robust conn} and for random graphs in Section \ref{sec: lad random}.  Finally in Section \ref{sec: lad cube} we use ladder chains with slopes defined by a Sidon set to prove the a running time of $\Omega(n^{3/2})$ when $H$ is a three-dimensional cube.

We continue with  the formal definition of ladder chains.

\vspace{2mm}

\begin{const}[Collection of ladder chains] \label{const:ladder}
    Let $H$ be a graph with $k:=v(H)\geq 4$, $U\subset V(H)$, $W:=V(H)\setminus U$ and $e',f'\in E_H(U,W)$ vertex-disjoint edges. Furthermore, let
$A\subseteq \mathbb{N}$ some index set and fix $m:=\max A$. Then the collection of $H$-chains \[\cH_{\mathrm{Lad}}(H,U,e',f';A)=\{(H^a_i,e^a_i)_{i\in[m]}:a\in A\}\] is defined as follows.

Fix $k_U:=|U|-1$, $k_W:=|W|-1=k-k_U-2$ and label $U$ as $U=\{u_0,\ldots,u_{k_U}\}$ and $W$ as $W=\{w_0,\ldots,w_{k_W}\}$ such that $e'=u_0w_0$ and $f'=u_{k_U}w_{k_W}$. The vertex set $V$ of our collection is partitioned as $V=L\cup R$ with $L=\{\ell_0,\ldots,\ell_{5mk_U}\}$, $R=\{r_0,\ldots,r_{5mk_W}\}$ and for $j\in \{0,\ldots,5m-1\}$, we define 
\[
		L_j := \{ \ell_{jk_U},\ldots, \ell_{(j+1)k_U} \} \quad\text{ and }\quad R_j := \{  r_{j k_W},\ldots, r_{(j+1)k_W} \}.
		\]
For $i\in [m]$ and $a\in A$,  $H_i^a$ is the image of the graph isomorphism defined by the bijection $\phi_i^a:U\cup W\rightarrow L_{i-1}\cup R_{i+4a-1}$ such that \[\phi_i^a(u_s)=\ell_{s+(i-1)k_U} \mbox{ for } s\in \{0,\ldots,k_U\} \mbox{ and } \phi^a_i(w_{s'})=r_{s'+(i+4a-1)k_W} \mbox{ for } s'\in\{0,\ldots,k_W\}.\]  We define $e^a_i:=\phi^a_i(f')=\ell_{ik_U}r_{(i+4a)k_W}$ and we take the origin of the chain $(H^a_i,e^a_i)_{i\in [m]}$ to be $e_0^a:=\phi^a_1(e')=\ell_{0}r_{(4a-1)k_W}$. 
\end{const}

\vspace{2mm}

Note that each $H$-chain $(H_i^a,e_i^a)_{i\in [m]}$ in the collection $\cH$ defined through Construction \ref{const:ladder} is simple (Definition \ref{def:simple-chain}) with respect to an ordering on $V(H)=[k]$ which places the vertices of $e'$ first and the vertices of $f'$ last. 
Thus an application of Lemma \ref{lem:inseparableimpliesproper} gives the following. 

\begin{obs} \label{obs:lad prop}
    If $H$ is $(2,1)$-inseparable or $(1,1,1)$-inseparable and $\cH=\cH_{\mathrm{Lad}}(H,U,e',f';A)$ is some collection of ladder chains as in Construction \ref{const:ladder}, then each chain in $\cH$ is proper. 
\end{obs}

Note also  that when $m$ is sufficiently large, the vertex set $V$ of a collection of  chains produced by 
Construction \ref{const:ladder} has size $|V|=5m(k_U+k_W)+2\leq 5m k$ and each chain $(H^a_i,e_i^a)_{i\in [m]}$ in the collection $\cH$ has length $m$. In our various applications of this construction, we will show that the collection is proper and appeal to Lemma \ref{lem:linking chains} to give a proper $H$-chain of length at least $m|A|$ on at most $n:=10k^3m$ vertices, thus implying that $M_H(n)=\Omega(|A|n)$. 

Before using these constructions to give lower bounds on running times, we provide  some general lemmas that will be useful in our proofs. 

\begin{lem} \label{lem:ladders crossing edges}
    Let $\cH=\cH_{\mathrm{Lad}}(H,U,e',f';A)$ be some collection of ladder chains  and let $\ell\in L,r\in R$.  
Then there is at most one $a\in A$ such that $\{\ell,r\}\subset V(H_i^a)$ for some $i\in [m]$. In particular, if $\{\ell,r\}\in E( \bigcup_{a\in A,i\in [m]}\final{H_i^a}_H)$, then $\{\ell,r\}\in E( \bigcup_{i\in [m]}\final{H_i^a}_H)$ for a unique $a\in A$. 
    \end{lem}
    \begin{proof}
         Choose $j,j'\in \{0,\ldots,5m-1\}$ such that $\ell\in L_j$ and $r\in R_{j'}$.  By the definition of the construction we have that  if $\{\ell,r\}\subseteq V({H_i^a})$ for some $i\in [m]$ it implies that  $j'\in J(j,a):=\{j+4a-2,j+4a-1,j+4a,j+4a+1, j+4a+2\}$. Indeed we must have that $V(H_i^a)=L_h\cup R_{h'}$ for some $h\in \{j-1,j,j+1\}$ and $h'\in \{j'-1,j',j'+1\}$ such that  $h'=h+4a$. 
    The lemma then follows because the sets $J(j,a)$ are disjoint for different choices of $a\in A$ (and $j$ fixed).
    \end{proof}

This has the following  immediate consequence.

\begin{cor} \label{cor: ladder I II}
    If $\cH=\cH_{\mathrm{Lad}}(H,U,e',f';A)$ is a collection of ladder chains then it satisfies condition  \ref{chaincoll: edges not early} of being proper (Definition \ref{def:proper chaincoll}).
\end{cor}

Finally, we give the following general lemma about ladder chains. 
    \begin{lem} \label{lem: ladder induced}
         Let $\cH=\cH_{\mathrm{Lad}}(H,U,e',f';A)$ be a collection of ladder chains as in Construction \ref{const:ladder} and let $\hat G=\bigcup_{a\in A,i\in [m]}\final{H_i^a}_H$. For any $j,j'\in \{0,\ldots,5m-1\}$ the following holds:
    \begin{enumerate}
        \item \label{item: ladd induce full} If there  is some $a\in A$ with $j'=j+4a$ then $\hat G[L_j \cup R_{j'}]=\final{H^a_{j+1}}_H$; 
        \item \label{item: lad induce empty} Otherwise, letting $G_1'=\hat G-\{\ell_{jk_U},r_{(j'+1)k_W}\}$ and $G_2'=\hat G-\{\ell_{(j+1)k_U},r_{j'k_W}\}$, we have that at least one of $E_{G_1'}(L_j,R_{j'})$ and $E_{G'_2}(L_j,R_{j'})$ is empty. In particular,
        $e_{\hat G}(L_j,R_{j'})<k$. 
    \end{enumerate}
    \end{lem}
    \begin{proof}
        For part \eqref{item: ladd induce full}, note that if $j'=j+4a$, we have that $V(H^a_{j+1})=L_j\cup R_{j'}$ and so certainly $\hat G[L_j\cup R_{j'}]$ contains $\final{H^a_{j+1}}_H$. For the other direction, first note that any edge in $E_{\hat G}(L_j,R_{j'})$ comes from the simple chain $(H_i^a,e_i^a)_{i\in [m]}$ by Lemma \ref{lem:ladders crossing edges} and so is contained in $\final{H_{j+1}^a}_H$. If $g\in E(\hat G[L_j])$ with $g\in E(\final{H_{j+1}^b}_H)$ for  some $b\in A$ then the map $\phi^b_{j+1}:U\cup W \rightarrow L_j \cup R_{j+4b}$ induces a graph isomorphism from $H$ to  $H^b_{j+1}$. Using Observation \ref{obs:hom} we see that $\phi^b_{j+1}$ also induces a graph isomorphism from $\final{H}_H$ to $\final{H^b_{j+1}}_H$  that maps an edge of $\final{H}_H$ on the vertex set $U$ to $g$. But then $\phi^a_{j+1}\circ (\phi^b_{j+1})^{-1}$ is again a graph isomorphism from $\final{H^b_{j+1}}_H$ to $\final{H^a_{j+1}}_H$ and thus $g\in \final{H^a_{j+1}}_H$  also. Similarly if $g\in E(\hat G[R_{j'}])$ with $g\in E(\final{H_{j'-4b+1}^b})$ for some $b\in A$, then we see that $g\in E(\final{H_{j+1}^a}_H)$  by considering the graph isomorphism $\phi_{j+1}^a\circ (\phi_{j'-4b+1}^b)^{-1}: \final{H_{j'-4b+1}^b} \rightarrow \final{H_{j+1}^a}_H$. 

        For part \eqref{item: lad induce empty}, suppose that $j'\neq j+4a$ for any $a\in A$ and consider   $g=\{\ell,r\}\in E(\hat G)$ with $\ell\in L_j\setminus \{\ell_{jk_U}\}$ and $r\in R_{j'}\setminus \{r_{(j'+1)k_W}\}$. We will show that in such a case we either have that $\ell=\ell_{(j+1)k_U}$ or $r=r_{j'k_W}$ (or both) which will establish the first part of \eqref{item: lad induce empty}. As $g\in E(\hat G)$, we have that $g\in E(\final{H_i^a}_H)$ for some $a\in A$ and $i\in [m]$. Fixing such an $a$ and $i$, we therefore have that $\ell\in  L_{i-1}\cap L_j$ and $r\in R_{i+4a-1}\cap R_{j'}$. This in turn implies that $j\in \{i-2,i-1\}$, ruling out $j=i$ here as otherwise $\ell=\ell_{jk_U}$. Similarly, $j'\in \{i+4a-1,i+4a\}$. If $j=i-2$ then $\ell=\ell_{(j+1)k_U}$ and we are done. Similarly, we are done if $j'=i+4a$ in which case $r=r_{j'k_W}$. The only remaining case is that $(j,j')=(i-1,i+4a-1)$ which is ruled out as we are not in case \eqref{item: ladd induce full}. The `in particular' statement of \eqref{item: lad induce empty} follows because the number of edges in $\hat G$ between $L_j$ and $R_{j'}$ is at most \begin{equation} \label{eq:crossing count}
        d_{\hat G}(x,R_{j'}\setminus \{y\})+d_{\hat G}(y,L_j\setminus \{x\})+1\leq |L_j|+|R_{j'}|-1=k-1   
        \end{equation} for some choice of $(x,y)\in \{(\ell_{jk_U},r_{(j'+1)k_W}),(\ell_{(j+1)k_U},r_{j'k_W})\}$ (the $+1$ in \eqref{eq:crossing count}  coming from the possible edge $xy$). 
    \end{proof}

\subsection{Ladder chains for $(\ceil{k/2},1)$-inseparable $k$-vertex  graphs} \label{sec:lad robust conn}

Recall the definition of \emph{$(\ceil{k/2},1)$-inseparable} graphs (Definition \ref{def:robust conn}). In this section, we use Construction \ref{const:ladder} to show that any $(\ceil{k/2},1)$-inseparable $k$-vertex graph has a quadratic maximum running time. As shown in Lemma \ref{lem:min deg implies robust conn}, this implies a quadratic running time for $k$-vertex graphs $H$ with minimum degree $\delta(H)>3k/4$ (Theorem \ref{thm;dense quad} part \eqref{thm:dense}). 

\begin{thm} \label{thm: robust conn}
    Let $H$ be a $(\ceil{k/2},1)$-inseparable $k$-vertex graph with $k\geq 6$. Then $M_H(n)=\Omega(n^2)$. 
\end{thm}

\begin{proof}
Let $k=v(H)\geq 6$ and let $U\subset V(H)$ be a vertex subset such that $|U|=\floor{k/2}$ and $e_H(U,W)$ is minimised, where $W:=V(H)\setminus U$. Moreover, let
 $e'$ and $f'$ be some pair of non-incident edges in $E_H(U,W)$. To see these exist, note that one can take some  arbitrary $u\in U$ and $ w',w\in W$ and use that $H[U\setminus \{u\}\cup \{w\}]$ is 2-edge-connected to find some edge $e':=u'w$ and then again use that $H[U\setminus \{u'\}\cup \{w'\}]$ is 2-edge-connected to find an edge $f'=u''w'$ disjoint from $e$.

 Next, for $n$ sufficiently large, we let $m:=\floor{n/10k^3}$ and take \[\cH=\cH_{\mathrm{Lad}}(H,U,e',f';[m])=\{(H_i^a,e^a_i)_{i\in [m]}:a\in [m]\}\]
to be the collection of ladder chains on a vertex set $V$ as defined in Construction \ref{const:ladder} with $A:=[m]$ (with other notation as defined there). 
We further define \[L^{\mathrm{int}}:=L\setminus \{\ell_{jk_U}:j=0,\ldots,5m\} \quad \mbox{and} \quad R^{\mathrm{int}}:=R\setminus \{r_{jk_W}:j=0,\ldots,5m\}. 
\]
As $H$ is $(2,1)$-inseparable, each chain in $\cH$ is proper (Observation \ref{obs:lad prop}) and we will show $\cH$ is a proper collection of $H$-chains and so by Lemma \ref{lem:linking chains}, we will get a proper $H$-chain of length at least $m^2$, on at most $2k^2|V|\leq n$ vertices. Hence by Lemma \ref{lem:chain_runningtime} we have that $M_H(n)\geq m^2=\Omega(n^2)$, as required.

Thus it remains to prove that $\cH$ is indeed a proper collection of $H$-chains and by Corollary \ref{cor: ladder I II}, we only need to establish condition \ref{chaincoll:H-e new 1} of Definition \ref{def:proper chaincoll}. To this end, we let $G^a:=\bigcup_{i\in [m]}H_i^a$ for $a\in [m]$ and we fix some $e\in E(H)$ and a copy $F$ of $H-e$ in \[\hat G:= \bigcup_{a\in [m]}\final{G^a}_H=\bigcup_{a\in [m],i\in[m]}\final{H^a_i}_H.\]
 We also define $U':=V(F)\cap L$ and $W':= V(F)\cap R$. 
 We have to show that there is a unique $b\in A=[m]$ such that $F\subseteq \final{G^b}_H$.  
 We   proceed  in stages \ref{S1}-\ref{S5}, with the outcomes of each stage highlighted before being followed by their respective proofs.

\begin{enumerate}[\textbf{(S1)}]
    \item \label{S1} We have that $|U'| = \floor*{{k/2}}$ and $|W'| = \ceil*{k/2}$. \end{enumerate}
Suppose for a contradiction that $|U'| > \floor{k/2}$.
Let $j\in\{0,\ldots,5m-1\}$ be the minimum  value such that $U' \cap  (L_{j}\setminus \{\ell_{(j+1)k_U}\}) \neq \varnothing$. As $|U'|> |U|=|L_{j}|$, we have that $U'\setminus L_{j}\neq \varnothing$. But then $F[U'\setminus \{\ell_{(j+1)k_U}\}]$ is disconnected, which contradicts that $H$ is $(\ceil{k/2},1)$-inseparable   as  $|U'\setminus \{\ell_{(j+1)k_U}\}|\geq \floor{k/2}$ (see Remark \ref{rem:rob conn equiv}). Thus $|U'|\leq \floor{k/2}$ and an analogous argument gives $|W'|\leq \ceil{k/2}$. As $|U'|+|W'|=k$, we get \ref{S1}. In particular,  $F[U']$ and $F[W']$ are connected by Remark \ref{rem:rob conn equiv}.

\begin{enumerate}[\textbf{(S2)}]
    \item \label{S2} We have that either $U'\cap L^{\mathrm{int}}\neq \emptyset$ or $W'\cap R^{\mathrm{int}}\neq \emptyset$.\end{enumerate}

Suppose for a contradiction we have that 
$U'\subseteq \{\ell_{jk_U}:j=0,\ldots,5m\}$ and $W'\subseteq \{r_{jk_W}:j=0,\ldots,5m\}$. Then as $F[U']$ and $F[W']$ are both connected, we must have that both $F[U']$ and $F[W']$ are paths. However, we have that there is a single edge that can be added to $F$ to obtain some copy $H'$ of $H$ and  so we have that either $F[U']=H'[U']$ or $F[W']=H'[W']$ (or both). Therefore at least one of $F[U']$ and $F[W']$ is 2-edge-connected, contradicting that they are both paths.

\begin{enumerate}[\textbf{(S3)}]
    \item \label{S3} 
If  $U'\cap L^{\mathrm{int}}\neq \emptyset$ then there exists $j_R\in \{0,\ldots,5m-1\}$ such that  $W'=R_{j_R}$. 
\end{enumerate}

Assuming that $U'\cap L^{\mathrm{int}}\neq \emptyset$, there is some $j_*\in \{0,\ldots,5m-1\}$ such $U'\cap L^{\mathrm{int}}_{j_*}\neq \emptyset $ where $ L^{\mathrm{int}}_{j_*}:=L_{j_*}\setminus \{\ell_{j_*k_U},\ell_{(j_*+1)k_U}\}$. Now for any pair of vertices $r,r'\in W'$, we have that $X=U'-\{\ell_{j_*k_U},\ell_{(j_*+1)k_U}\}+\{r,r'\}$ has size at least $
\floor{k/2}$ and hence $F[X]$ is connected, and in particular at least one of $r$ and $r'$ is adjacent in $F\subset \hat G$ to a vertex in $L^{\mathrm{int}}_{j_*}$. 
Therefore at least one of $r$ and $r'$ lies in $\mathcal{R}:=\{R_{j'}:j'=j_*+4a \mbox{ for some } a\in A=[m]\}$ by Lemma \ref{lem:ladders crossing edges} as the sets in $\cR$ are the only vertices that have neighbours in $L^{\mathrm{int}}_{j_*}$ in $\hat G$. Repeating this for different choices of $\{r,r'\}\subseteq R$, we get that there is some $W''\subseteq W'$ with $|W''|= |W'|-1$ such that $W''\subset \bigcup \cR$. In fact, we cannot have vertices $u,v\in W''$ in distinct sets in $\cR$ as $F[W']\subset \hat G$ is connected and $W'$  is obtained from $W''$ be adding a single vertex of $R$ (which cannot have neighbours in two distinct sets in $\cR$). Therefore, we have that $W''\subset R_{j_R}$ for some $j_R\in \{0,\ldots, 5m-1\}$. Now suppose for a contradiction that the vertex $r_{s}\in W'\setminus W''$ is such that $r_s\notin R_{j_R}$. If $s<j_Rk_W$ then, fixing $y$ to be some vertex in $L^{\mathrm{int}}_{j_*}$ we have that $Y=W'-\{r_{j_Rk_W}\}+\{y\}$ has size  at least $\floor{k/2}$ and thus $F[Y]$ is connected and $r_{s}y\in E(\hat G)$ and again by Lemma \ref{lem:ladders crossing edges}, we get that $r_s\in \cR$. As $r_s\notin R_{j_R}$, we get a contradiction as $F[W']$ is connected. The case that $s>(j_R+1)k_W$ is analogous.

\begin{enumerate}[\textbf{(S3')}]
    \item \label{S3'} 
If  $W'\cap R^{\mathrm{int}}\neq \emptyset$ then there exists $j_L\in \{0,\ldots,5m-1\}$ such that  $U'=L_{j_L}$. 
\end{enumerate}

The proof of this is identical to the proof of \ref{S3} with the sides reversed.

\begin{enumerate}[\textbf{(S4)}]
    \item \label{S4} We have that $V(F)=V(H^b_{i'})$ for some unique $i'\in [m]$ and $b\in {[m]}$.  \end{enumerate}

Firstly, we have  $j_L,j_R\in \{0,\ldots,5m-1\}$ such that  $U' = L_{j_L}$ and $W' = R_{j_R}$. This follows from appealing to \ref{S2} and then either \ref{S3} followed by \ref{S3'} if we are in the first case or vice versa if we are in the second. 
We will appeal to Lemma \ref{lem:ladders crossing edges}  and show that $e_{\hat G}(U',W')\geq k$ and so $j_R=j_L+4b$ for some $b\in [m]$ and $U'\cup W'=L_{j_L}\cup R_{j_R}=V(H_{i'}^b)$ with $i'=j_L+1$.  To count edges in $E_{\hat G}(U', W')$, note that $d_{F}(w,U')\geq 1$ for all
vertices $w\in W'$ as $F[U'\cup \{w\}]$ is connected. In fact, for all
but at most one $w\in W'$, we have $d_{F}(w,U')\geq 3$. Indeed, let $g$ be the edge corresponding to $e$ in the copy $F$ of $H-e$  and suppose 
$w$ is not an endpoint of  $g$ with the other  vertex of $g$ in $U'$  (which  happens  for all but at most one $w$). Then letting $H':=F+g$, we have that  $H'[U'\cup \{w\}]$ is 2-edge-connected and so $w$ has at least two neighbours  in $U'$ (in both $H'$ and $F$). Moreover, taking $u$ to be one of these neighbours, we have that $H'[(U'\setminus \{u\})\cup \{w\}]$
 is again 2-edge-connected and so $w$ has at least two $F$-neighbours in $U'$ \emph{other than} $u$. Thus $e_{\hat G}(U',W')\geq e_{F}(U',W')\geq 3\left(\ceil*{\tfrac{k}{2}}-1\right)+1\geq k$ as required, using here that $k\geq 6$.

\begin{enumerate}[\textbf{(S5)}]
    \item \label{S5} We have that $F\subseteq \final{H^b_{i'}}_H\subseteq \final{G^b}_H$. \end{enumerate}
This final stage follows from Lemma \ref{lem:ladders crossing edges} part \eqref{item: ladd induce full}.

\vspace{2mm}
This completes the proof of the first part of condition \ref{chaincoll:H-e new 1} of Definition \ref{def:proper chaincoll}. To finish the proof, we need to show the `furthermore' part of condition \ref{chaincoll:H-e new 1}. As $\hat G[V(F)]=\final{H_{i'}^b}_H$ by Lemma \ref{lem:ladders crossing edges} \eqref{item: ladd induce full}, we have to show that $F\nsubseteq \final{H_{i'}^b}_H-\{e_{i'-1}^b,e_i^b\}$. If this was not the case, then we would have that $e_F(U',W')\leq e_H(U,W)-2$  and  we could add some edge $h$ to $F$ to get some copy $H'$ of $H$ such that $e_{H'}(U',W')<e_{H}(U,W)$ contradicting our choice of $U$ which minimised the number of crossing edges between $U$ and $W$. This completes the proof.
\end{proof}

\subsection{Ladder chains for random graphs} \label{sec: lad random}

In this section we prove the second part of Theorem \ref{thm:random}, showing that if $p$ is above the connectivity threshold $\log k/k$, then a.a.s.\ the maximum running time $M_H(n)$ for the binomial random graph $H=G(k,p)$ is quadratic in $n$. 

\begin{proof}[Proof of Theorem \ref{thm:random} part \eqref{random 1}]
    Fix  $0<\eps<1/150$, $k$ sufficiently large and some partition of $[k]$ as $[k]=U\cup W$ with $|U|=\floor{k/2}$ and $|W|=\ceil{k/2}$. Let $p=\omega(\log k/k)$ and let $H=G(k,p)$ on the vertex set $[k]$. We have that a.a.s.\ $H$ satisfies all the properties of Lemma \ref{lem:G(k,p) properties}. In particular, if $p\geq 4/5$, then by Lemma \ref{lem:G(k,p) properties} \eqref{G(k,p): degrees U W}, we have that a.a.s.\  $\delta(H)> 3k/4$ and $M_H(n)=\Omega(n^2)$ by Theorem \ref{thm;dense quad} part \eqref{thm:dense}. Thus, we can assume that $p<4/5$ and  so a.a.s.\ $H$ is self-stable due to Fact \ref{fact: G(k,p) self-stable}. 
So let us fix an instance of $H$ satisfying all the properties of  Lemma \ref{lem:G(k,p) properties}  as well as Fact \ref{fact: G(k,p) connected} and Fact \ref{fact: G(k,p) self-stable}  which  a.a.s.\ occur. With such an instance $H$, we will prove that $M_H(n)=\Omega(n^2)$, thus proving the theorem. 

With $U$ and $W$ as above, using the assumed outcome of Lemma \ref{lem:G(k,p) properties} \eqref{G(k,p): degrees U W} in $H$, we have that there is some pair of disjoint edges $e',f'\in E_H(U,W)$. Fixing such a pair of edges arbitrarily, for $n$ large, we let $m=\floor{n/10k^3}$ and 
    take \[\cH=\cH_{\mathrm{Lad}}(H,U,e',f';[m])=\{(H_i^a,e^a_i)_{i\in [m]}:a\in [m]\}\]
to be the collection of ladder chains on a vertex set $V$ as defined in Construction \ref{const:ladder} with $A:=[m]$. As $H$ is $(2,1)$-inseparable (Fact \ref{fact: G(k,p) connected}), by  Observation \ref{obs:lad prop}, Lemma \ref{lem:linking chains} and Lemma \ref{lem:chain_runningtime}, it suffices to show that $\cH$ is a proper collection of chains in order to show that $M_H(n)=\Omega(n^2)$. 

By Corollary \ref{cor: ladder I II}, we have that condition  \ref{chaincoll: edges not early} of the collection being proper (Definition \ref{def:proper chaincoll}) is satisfied. Letting $G^a=\bigcup_{i\in [m]}H^a_i$, we have that as $H$ is self-stable (Fact \ref{fact: G(k,p) self-stable}), in order to show that $\cH$ is proper, we can show condition \ref{chaincoll:proper 2 star} in Remark \ref{rem:self stable coll} which states that if $e\in E(H)$ and $F$ is a copy of $H-e$ in $\hat G:=\bigcup_{a\in [m]}G^a$, then there is a unique $b\in A=[m]$ such that $F\subseteq G^b$. To this end, let us fix some $e\in E(H)$ and $F$ some copy of $H-e$ in $\hat G$. Take $U':=V(F)\cap L$ and $W':=V(F)\cap R$. As in the proof of Theorem \ref{thm: robust conn}, we  will proceed  in stages  with the outcomes of each stage highlighted before being followed by their respective proofs. For this proof we have   four stages \ref{T1}-\ref{T4}.

\begin{enumerate}[\textbf{(T1)}]
    \item \label{T1} There are  $j_L, j_R\in \{0,\ldots, 5m-1\}$ such that $Z:=V(F)\setminus (L_{j_L}\cup R_{j_R})$ has size $|Z|\leq 4\eps k$. \end{enumerate}

Assume first that $|U'|\geq 2\eps k$. For $j\in \{0,\ldots, 5m-1\}$, let $L_j^-= L_{0}\cup \cdots \cup L_{j}$ and $L_j^+=L_j\cup \cdots \cup L_{5m-1}$ and let $L^-_{-1},L^+_{5m}:=\emptyset$.   Define  $j_L$ to be the minimum index $j\in\{0,\ldots,5m-1\}$ such that $|U'\cap L_j^-|\geq \eps k+1$. Then by definition we have that $|U'\cap L^-_{j_L-1}|\leq \eps k$. Moreover $|U'\cap (L^+_{j_L+1}\setminus\{\ell_{(j_L+1)k_U}\})|\leq \eps k$. Indeed, this is immediate if $j_L=5m-1$ and if not then this follows from (the assumed conclusion of) Lemma \ref{lem:G(k,p) properties} \eqref{G(k,p):hole} in $H$ as there is no edge in $F$ between $U'\cap (L^+_{j_L+1}\setminus \{\ell_{(j_L+1)k_U}\})$ and $U'\cap (L^-_{j_L}\setminus\{\ell_{(j_L+1)k_U}\})$ as there is no edge in $\hat G$ between these sets. As the latter set has size at least $\eps k$, we get that the former set must have size at most $\eps k$ and so we have that $U'\setminus L_{j_L}$ has size at most $2\eps k$. In particular, we have that $|U'|\leq |L_{j_{L}}|+2 \eps k\leq (1/2+3\eps) k$ and so $|W'|\geq 2 \eps k$ and running an analogous argument on the right side gives some $j_R\in \{0,\ldots, 5m-1\}$ with $|W'\setminus R_{j_R}|\leq 2\eps k$. This proves \ref{T1}. Note that we began the proof by assuming that $|U'|\geq 2\eps k$ but if this was not the case then we would certainly have $|W'|\geq 2\eps k$ and we could run the same proof by looking at the right side first.

\begin{enumerate}[\textbf{(T2)}]
    \item \label{T2} We have that  $U'=L_{j_L}$ and $W'=R_{j_R}$.\end{enumerate}

We will show that $Z=\emptyset$ from which the statement follows. For  $z\in V(F)$, we have that \begin{equation} \label{eq: degs lower random}
d_F(z)\geq d_H(z)-1\geq (1-3\eps)kp \end{equation} by the assumed conclusion of Lemma \ref{lem:G(k,p) properties} \eqref{G(k,p): degrees U W}. Moreover,  if $z\in Z$ then \begin{equation} \label{eq: degs upper  random}d_F(z,L_{j_L}\cup R_{j_R})\leq d_{\hat G}(z,L_{j_L}\cup R_{j_R})\leq 1 + (1/2+\eps)kp\leq (1/2+2\eps)kp.\end{equation} Indeed, for example if $z\in Z\cap U'$, then $z$ has at most one neighbour in $L_{j_L}$ and at most $(1/2+\eps)kp$ neighbours in $R_{j_R}$ because of Lemma \ref{lem:G(k,p) properties} \eqref{G(k,p): degrees U W} and the fact that edges from $z$ to $R_{j_R}$ in $\hat G$ correspond to edges from a vertex of $U$ to $W$ in some copy $H_{j'}^{a'}$ of $H$. Combining \eqref{eq: degs lower random} and \eqref{eq: degs upper  random}, for all $z\in Z$ we get that $d_F(z,Z)\geq (1/2-5\eps)kp\geq kp/4$ and so $e(F[Z])\geq |Z|kp/8$. As $|Z|\leq 4\eps k$, we must have that $Z=\emptyset$ as otherwise we violate the assumed conclusion of Lemma \ref{lem:G(k,p) properties} \eqref{G(k,p): dense small} in $H$.

\begin{enumerate}[\textbf{(T3)}]
    \item \label{T3} 
There exists a unique $b\in [m]$ such that $j_R=j_L+4b$.
\end{enumerate}
As in \eqref{eq: degs lower random}, we have that every vertex $v\in V(F)$ has $d_F(v)\geq (1-3\eps) kp$ and if $v\in L_{j_L}$ then $d_F(v,L_{j_L})\leq d_{\hat G}(v,L_{j_L})\leq (1/2+\eps) kp$ by Lemma \ref{lem:G(k,p) properties} \eqref{G(k,p): degrees U W} and the fact that $\hat G[L_{j_L}]$ is a copy of $H[U]$. Thus for each $v\in L_{j_L}$, we have that $d_{\hat G}(v,R_{j_R})\geq d_{F}(v,R_{j_R})\geq (1/2-4\eps )kp\geq kp/4$ and so $e_{\hat G}(L_{j_L},R_{j_R)})\geq |L_{j_L}|kp/4\geq k^2p/8>k$. By Lemma \ref{lem: ladder induced}, we must indeed have that $j_R=j_L+4b$ for some unique $b\in [m]$.

\begin{enumerate}[\textbf{(T4)}]
    \item \label{T4} We have that $F\subseteq H^b_{j}\subseteq G^b$ with $j:=j_L+1$.  \end{enumerate}

Step \ref{T3} shows that $V(F)\subseteq V(H^b_j)$ and as $\hat G[H_j^b]=H_j^b$  by Lemma \ref{lem: ladder induced} \eqref{item: ladd induce full}, we are done. This concludes the proof of property \ref{chaincoll:proper 2 star} in Remark \ref{rem:self stable coll} for $F$ and thus the proof that $\cH$ is proper. 
\end{proof}

\subsection{Ladder chains for the cube} \label{sec: lad cube}
Our final application of Construction \ref{const:ladder} is to show Theorem \ref{thm:cube} giving a lower bound on the running time for the 3-dimensional cube graph $Q_3$. Unlike the previous two chain constructions, the set $A$ that we use to define our collection of ladder chains is not dense. Indeed, we will use a \emph{Sidon set}, appealing to its defining properties to deduce that the collection of chains  is proper. Recall from Section \ref{sec:additive} that a subset $S\subset \mathbb{Z}$ is called a \emph{Sidon set} if the only solutions to the equation $s_1+s_2=s_3+s_4$ with $s_i\in S$ for $i\in [4]$, are the trivial solutions with $\{s_1,s_2\}=\{s_3,s_4\}$. Recall also that Sidon sets $S\subseteq [m]$ of size $m^{1/2}(1-o(1))$ exist \cite{singer1938theorem}. We now proceed with the proof of Theorem \ref{thm:cube}.

\begin{proof}[Proof of Theorem \ref{thm:cube}]
    We recall that the cube $H=Q_3$ is $(1,1,1)$-inseparable, using Lemma \ref{lem:bip dense insep} and the fact that that each vertex of $H$ has degree 3 while the parts in the bipartition of $H$ have size 4. We will also use that $H=Q_3$ is self-stable (Observation \ref{obs:self-stable}). We fix $U:=\{u_0,\ldots,u_3\}\subset V(H)$ and $W=V(H)\setminus U:=\{w_0,\ldots,w_3\}$ such that both $U$ and $W$ induce copies of $C_4$ in $H$, $u_iw_i\in E(H)$ for $i=0,\ldots, 3$ and $u_0u_3,w_0w_3\notin E(H)$. We also define $e':=u_0w_0$ and $f':=u_3w_3$ and  for $n$ sufficiently large, we let $m:=\floor{n/150}$ and take $A\subseteq [m]$ to be a maximum sized Sidon set in $[m]$. Take the collection of ladder chains 
       \[\cH=\cH_{\mathrm{Lad}}(H,U,e',f';A)=\{(H_i^a,e^a_i)_{i\in [m]}:a\in A\}\]
on a vertex set $V$ as defined in Construction \ref{const:ladder} (with other notation also adopted from there), noting that by Observation \ref{obs:lad prop} each chain in $\cH$ is proper. Note that $G:=\bigcup_{a\in A,i\in [m]}H_i^a$ is bipartite on $V$ with partite sets  
\[X_1:=\{\ell_s:s\equiv 0 \mbox{ mod } 3\}\cup \{r_t:t\not\equiv 0 \mbox{ mod } 3\} \mbox{ and }  X_2:=\{\ell_t:t\not\equiv 0 \mbox{ mod }3\}\cup \{r_s:s\equiv 0 \mbox{ mod } 3\}.\]
Indeed, one can check that with $U$ and $W$ defined (and labelled) as above, any edge in one of the $H_i^a$ given by Construction \ref{const:ladder} goes between a vertex in $X_1$ and one in $X_2$. In particular note that for any  $s,t\in \{0,\ldots, 15m\}$ we have that 
\begin{equation} \label{eq:cubenbrs}
    N_G(\ell_s)\cap R\subset \{r_{s+12a}:a\in A, s+12a\leq 15m\}  \mbox{ and }   N_G(r_t)\cap L\subset \{\ell_{t-12a}:a\in A, 12a\leq t+1\}.
\end{equation}

We now show the following claim. 

\begin{clm} \label{clm:C4s in cube ladders}
    Any copy $C$ of $C_4$ in $G=\bigcup_{a\in A,i\in [m]}H_i^a$ satisfies one of the following:
    \begin{enumerate}
        \item \label{cube:left} $V(C)=L_j$ for some $j\in \{0,\ldots,5m-1\}$;
        \item \label{cube:right} $V(C)=R_j$ for some $j\in \{0,\ldots,5m-1\}$;
        \item \label{cube:cross} $|V(C)\cap L|=|V(C)\cap R|=2$ and both $V(C)\cap L$ and $V(C)\cap R$ host edges of $C$. 
    \end{enumerate}
\end{clm}
\begin{claimproof}
    First we rule out that $|V(C)\cap L|=1$.  Indeed, the neighbourhoods as in \eqref{eq:cubenbrs} in particular imply that  for any $\ell\in L$ and any pair $r,r'\in N_G(\ell)\cap R$ have $\dist_G(r,r')>2$ and $\ell$ cannot form a $C_4$ with 3 vertices in $R$. Similarly, for any $r\in R$, its neighbours in $L$ are far apart and cannot be the endpoints of a path with three vertices. 

    This implies that $V(C)$ is either entirely contained in one side or has two vertices in each. In the first case, considering that  $V(G)\cap L$ is composed of a sequence of $C_4$s on the  sets $L_j$, each linked at a single vertex, and similarly for $V(G)\cap R$, we get that we must fall into  \eqref{cube:left} or \eqref{cube:right}. In the second case we only need to rule out that $V(C)\cap L=\{\ell_s,\ell_{s'}\}$ and $V(C)\cap R=\{r_t,r_{t'}\}$ and all edges of $C$ are between $L$ and $R$. However, if this is the case, then using the conditions \eqref{eq:cubenbrs} we have that up to relabelling, there are $a_1,a_2,a_3,a_4\in A$ such that $t-s=12a_1,t'-s'=12a_2, t-s'=12a_3$ and $t'-s=12a_4$. This in turn implies that  $a_1+a_2=a_3+a_4$ and as $A$ is a Sidon set  we must have that $\{a_1,a_2\}=\{a_3,a_4\}$. This is a contradiction to the fact that the vertices $\ell_s,\ell_{s'},r_t,r_{t'}$ are all distinct. Indeed, if $a_1=a_3$ then  $s=s'$ and similarly if $a_1=a_4$ then  $t=t'$. 
\end{claimproof}

\vspace{2mm}

We will show $\cH$ is a proper collection of $H$-chains and so by Lemma \ref{lem:linking chains}, using that $H$ is $(1,1,1)$-inseparable and the fact  that $G$ is bipartite on $V$, we will get a proper $H$-chain of length at least $m|A|\geq m^{3/2}(1+o(1))$, on at most $128|V|\leq n$ vertices. Hence by Lemma \ref{lem:chain_runningtime} we will have that $M_H(n)\geq m^{3/2}/2=\Omega(n^{3/2})$, as required. By Corollary \ref{cor: ladder I II} we already know that $\cH$ satisfies condition \ref{chaincoll: edges not early} of being proper (Definition \ref{def:proper chaincoll}) and so, using that $H$ is self-stable, it remains to prove condition \ref{chaincoll:proper 2 star} from Remark \ref{rem:self stable coll} which states that for any $e\in E(H)$ and copy $F$ of $H-e$ in $G$,  there is a unique $b\in A$  such that $F \subseteq G^b$.

    \begin{figure}[h]
    \centering
  \includegraphics[scale=1]{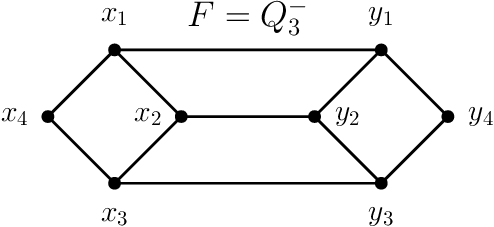}
    \caption{   \label{fig:Q3minus} The graph $F=Q_3^-$.} 
  \end{figure}

So fix a copy $F$ of $Q_3^-$ in $G$ and label the vertices of $F$ as in Figure \ref{fig:Q3minus}. We begin by showing that there are $j,j'\in \{0,\ldots,5m-1\}$ such that $V(F)=L_j\cup R_{j'}$. By Claim \ref{clm:C4s in cube ladders}, the cycle on vertices $X:=\{x_1,x_2,x_3,x_4\}$ must fall into one of the cases \eqref{cube:left}, \eqref{cube:right} or \eqref{cube:cross}.  Let us first consider case \eqref{cube:left} and so there is some $L_j$ such that $X=L_j$. Considering the copy of $C_4$ on $\{x_1,x_2,y_1,y_2\}$  and the structure of $G[L]$, appealing to Claim \ref{clm:C4s in cube ladders} again, we see that we must have $\{y_1,y_2\}\subset R$. Similarly $\{y_2,y_3\}\subset R$ and so by Claim \ref{clm:C4s in cube ladders} we must have that $Y:=\{y_1,y_2,y_3,y_4\}$ is equal to $R_{j'}$ for some $j'\in \{0,\ldots,5m-1\}$. 
The case \eqref{cube:right} in Claim \ref{clm:C4s in cube ladders} where $X= R_{j'}$ is analogous and we will get that $Y=L_j$ for some $j\in \{0,\ldots,5m-1\}$. Therefore it remains to consider the case that the cycle on $X$ falls into case \eqref{cube:cross} of Claim \ref{clm:C4s in cube ladders} and up to relabelling the vertices and switching the roles of the sides, we can assume that $\{x_1,x_2\}\subset L$ and $\{x_3,x_4\}\subset R$. 
By Claim \ref{clm:C4s in cube ladders}, as $x_2$ and $x_3$ lie on different sides, we must have that $y_2\in L$ and $y_3\in R$. Then, considering the cycle on $Y$ gives that $y_1\in L$ and $y_4\in R$. So we have that $\{x_1,x_2,y_1,y_2\}=L_j$ for some $j\in \{0,\ldots,5m-1\}$. Considering that each vertex in $Z:=\{x_3,x_4,y_3,y_4\}\subset R$ has a neighbour in $L_j$ and the vertices in $Z$ form a path in $G$, the neighbourhood condition \eqref{eq:cubenbrs} leads us to conclude that there must be some $j'\in \{0,\ldots,5m-1\}$ with $Z=R_{j'}$.

So we have that $V(F)=L_j\cup R_{j'}$ and  by Lemma \ref{lem: ladder induced}, we must in fact have that $j'=j+4b$ for some (unique)  $b\in A$ as in any bipartition of $V(F)$ there is a matching with at least three edges between the two sides $L_j$ and $R_{j'}$. Moreover, again using Lemma \ref{lem: ladder induced}, we have that $G[L_j\cup R_{j'}]=\final{H^b_{j+1}}_H=H^b_{j+1}$ and so $F\subset H^b_{j+1}\subseteq G^b$ for a unique $b\in A$  as required. This concludes the proof that $\cH$ is proper and the proof of the theorem. 
\end{proof}

\section{Dilation chains} \label{sec:dilation}

In this section we  introduce collections of \emph{dilation chains}. We construct one chain  $(H^a_i,e^a_i)_{i\in [\tau]}$  for each integer $a$ in a collection $A\subseteq \ZZ_p$ of \emph{dilations}. The chain 
$(H^a_i,e^a_i)_{i\in [\tau]}$  will be obtained by taking a simple chain with vertex set $\ZZ_p$ and dilating each of the vertices by a factor of $a$. We will use that $A$ is free of (non-trivial) solutions to certain equations in order to show that there is no unwanted copy  $F$ of $H-e$ for any  $e\in E(H)$ and therefore verify that the collection of chains is proper. We remark that although we reprove the result of Balogh, Kronenberg, Pokrovskiy and Szab\'o \cite{balogh2019maximum} showing that $M_{K_5}(n)=n^{2-o(1)}$ in this section and both proofs use additive constructions, our chain constructions in this section are new and quite distinct from the one used in \cite{balogh2019maximum}. Indeed, their construction has more in common with the ladder chains of the previous section, splitting the vertex set into three parts instead of two and having slopes defined by a set of integers free of solutions to three-term arithmetic progressions (in a similar way to how we prove Theorem \ref{thm:cube} using Sidon sets). Our construction here using dilation chains seems to be more flexible for proving general lower bounds on $M_H(n)$.

Before formally defining the chain construction and  its properties, we need to revisit additive constructions and obtain sets free of non-trivial solutions to several equations at once. We collect what we need in Section \ref{sec:dil add} below, delaying proofs to the Appendix. Recall the definition of non-trivial solutions from Section \ref{sec:additive}. 

\subsection{Solution-free sets} \label{sec:dil add}
Our main interest will be using constructions that avoid non-trivial solutions to all equations with bounded coefficients. We will use  $\Gamma=\mathbb{Z}_p$ as our host abelian group, which will be more convenient for our proofs. 

\begin{dfn}[$(K,h)$-fold solution-free sets] \label{def:K-fold}
    Let $K, h\in \NN$ and $p\in \NN$ be a prime. We say a set $A\subset \mathbb{Z}_p\setminus \{0\}$  is \emph{$(K,h)$-fold solution-free}, if for any equation \begin{equation}
\label{eq:equations bounded}    \sum_{i=1}^h\alpha_ix_i=0\end{equation} in $h$ variables with $\alpha_i\in \ZZ$ and $|\alpha_i|\leq K$ for $i\in [h]$, there are no non-trivial solutions $(x_1,\ldots,x_h)\in A^h$. We denote by $r_{K,h}(\mathbb{Z}_p)$ the size of a largest $(K,h)$-fold solution-free subset of $\ZZ_p$. 
\end{dfn}

Note that a $(K,h)$-fold solution-free set is also a $(K,h')$-fold solution-free set for $1\leq h'\leq h$ as we allow for coefficients $\alpha_i$ to be equal to $0$. 
We also remark that one can find  dense  subsets $B\subset \mathbb{Z}_p$ with $|B|\geq p/8h^2K^2$ that avoid non-trivial solutions to all  equations as in \eqref{eq:equations bounded} such that $\sum_{i=1}^h\alpha_i\neq 0$, using a simple modular based construction (see Lemma \ref{app lem:dense avoiding sum not zero} in Appendix \ref{appendix:add}).
Moreover, using a well-known random translation trick (Lemma \ref{app lem:dilate to mix sets}),
we can adapt any set  $A\subset \mathbb{Z}_p$ avoiding  some equations with $\sum_{i=1}^h\alpha_i=0$, to get a subset $A'\subseteq \ZZ_p$ with $|A'|\geq |A|/8h^2K^2$ that additionally avoids all equations as in \eqref{eq:equations bounded} with $\sum_{i=1}^h\alpha_i\neq 0$. Therefore, up to constants depending on $h$ and $K$, the problem of lower bounding $r_{K,h}(\ZZ_p)$  reduces to finding large sets avoiding all the  equations $\sum_{i=1}^h\alpha_ix_i=0$ with $\sum_{i=1}^h\alpha_i=0$. 

For $h=1$ this immediately gives dense $(K,1)$-fold solution-free sets
as when $\alpha=0$, any solution to $\alpha x=0$ is trivial. Similarly, when $h=2$, we can get  dense $(K,2)$-fold solution-free sets as  $\alpha x_1-\alpha x_2=0$ with $\alpha\neq 0$ necessarily implies that $x_1=x_2$ and the only solutions are trivial.   When $h=3$, a $(K,3)$-fold solution-free set in particular avoids three-term arithmetic progressions and so can no longer be dense. On the other hand, it is well-known (see for example \cite[Theorem 1]{shapira2006behrend}) that Behrend's construction can be adapted to avoid equations of the form $\alpha_1x_1+\alpha_2x_2-(\alpha_1+\alpha_2)x_3=0$ with $\alpha_1,\alpha_2$ bounded. With the aforementioned random translation trick to avoid equations where the $\alpha_i$ do not sum to zero, this gives the following, which is proven for completeness in Appendix \ref{appendix:add}. 

\begin{thm} \label{thm:behrend extended}For any $K\in \NN$, there exists $C>0$ such that for any prime $p\in \NN$, 
\[r_{K,3}(\ZZ_p)\geq p^{1-C/\sqrt{\log p}}.\]
    \end{thm}

The construction  in Theorem \ref{thm:behrend extended} will be used to give a lower bound on running times of the form $M_H(n)\geq n^{2-o(1)}$  for Behrendian graphs in Theorem \ref{thm:coherent almost quad}. 
Analogously, our construction of similar chains for Sidonian graphs (Theorem \ref{thm:bip coherent run time}) which will be used to prove  lower bounds on running times of the form $M_H(n)\geq n^{3/2-o(1)}$, requires dense $(K,4)$-fold solution-free sets. 

As before, lower bounding $r_{K,4}(\ZZ_p)$ (or the analogous question in the integers $[N]$) reduces to finding sets that avoid non-trivial solutions to all equations of the form $\sum_{i=1}^4\alpha_ix_i=0$ with $\sum_{i=1}^4\alpha_i=0$ and each $|\alpha_i|$ bounded by $K$. Such sets are known as \textit{$K$-fold Sidon sets}. They were introduced formally by Lazebnik and Verstraëte \cite{lazebnik_hypergraphs_2003} and studied by Cilleruelo and Timmons \cite{cilleruelo_k-fold_2014}. Lazebnik and Verstraëte \cite{lazebnik_hypergraphs_2003} conjectured that there should exist $K$-fold Sidon sets in $[N]$ of size $\Omega(N^{1/2})$ (which would lead to a lower bound of $r_{K,4}(\ZZ_p)\geq \Omega(p^{1/2})$) and verified the conjecture for $K=2$. Moreover, in \cite{cilleruelo_k-fold_2014} it is claimed that unpublished work of Axenovich and independently Verstra\"ete, establishes an asymptotic version of the conjecture, giving a lower bound of $N^{1/2-o(1)}$ by adapting the methods of Ruzsa \cite[Theorem 7.3]{ruzsa_solving_1993}. Unfortunately we could not recover this proof and we believe that there is a subtle issue that arises when simultaneously trying to avoid all the equations such that $\prod_{i=1}^4\alpha_i$ is not a square in the integers. In lieu of such a lower bound construction for $K$-fold Sidon sets, it turns out that our methods can work for an adapted notion of $(K,4)$-fold solution-free sets where we additionally have a lower bound on the coefficients that feature and a condition on the parity of the coefficients. We prove the following in Appendix \ref{appendix:add} by adapting the constructions of Ruzsa \cite{ruzsa_solving_1993}.

\begin{thm} \label{thm: sidon extended}
    For all $K\in 2\NN$  and $\eps>0$, there exists $m\in \NN$ such that for any sufficiently large prime $p\in \NN$ the following holds. There is a subset $A\subset \ZZ_p\setminus \{0\}$ such that $|A|\geq p^{1/2-\eps}$, $A$ is $(2(m+2K),3)$-fold solution-free and has no non-trivial solutions to equations of the form 
    \begin{equation} \label{eq:sidon extend}
        \alpha_1x_1+\alpha_2x_2+\alpha_3x_3+\alpha_4x_4=0
    \end{equation}
    with $|\alpha_i|\in \{m+1,m+3,\ldots,m+K-1\}$ for $i\in [4]$.
\end{thm}

\subsection{Constructing dilation chains} \label{sec:Dilation construction}
We now proceed with the definition of dilation chains. There are some technicalities in this definition that arise from requiring the construction to be  applicable to the wheel graph (Theorem \ref{thm:wheel}). We will elaborate on these in due course but  recommend that on first read, the reader ignores  the two distinct sets $V$ and $W$ in Construction \ref{const:dil} and identifies them as one vertex set in the natural way (with $v_i$ and $w_i$ identified for all $i\in [p-1]$). This removes the need for the case distinction in \eqref{eq:dil ends cases} and will illuminate the key idea of the construction. 
 
\vspace{2mm}

\begin{const}[Collection of dilation chains] \label{const:dil}
    Let $4\leq k \in \NN$ and suppose $H$ is a graph with $V(H)=[k]$ such that $e':=\{1,2\}$ and $f':=\{k-1,k\}$ are both edges of $H$.  Let $p\in \NN$ be a prime with $p\geq k^3$ and  $A\subseteq \ZZ_p\setminus {0}$ be a  set of \emph{dilations}. Then fixing $\tau:=\floor{\tfrac{p-3}{k-2}}$,  the collection  \[\cH_{\mathrm{Dil}}(H;A,p)=\{(H^a_i,e^a_i)_{i\in[\tau]}:a\in A\}\] of $H$-chains is defined as follows. 
    Let $V=\{v_1,\ldots,v_{p-1}\}$ and $W=\{w_1,\ldots,w_{p-1}\}$ be two vertex sets indexed by $\ZZ_p\setminus \{0\}$ (with indices always considered modulo $p$). For 
    $a\in A$ and $i\in [\tau]$ the graph $H_i^a$ is the image of the graph isomorphism  $\varphi^a_i : H\rightarrow H^a_i$ defined via the map \begin{equation} \label{eq:dil ends cases}
    \varphi^a_i(j)=\begin{cases}
		v_{a\cdot((i-1)(k-2)+j)}	, & \text{ if } k+2 \leq i    
  \leq  \tau-k-2 \\ &\text{ or } i=k+1, j\notin \{1,2\} \text{ or }i=\tau-k-1, j\notin \{k-1,k\} ; \\
           w_{a\cdot((i-1)(k-2)+j)}, & \text{ if } 1\leq i\leq k \text{ or } \tau-k\leq i\leq \tau  \\ &\text{ or } i=k+1, j\in \{1,2\} \text{ or }i=\tau-k-1, j\in \{k-1,k\}, 
		 \end{cases}
      \end{equation}
      for $j\in [k]=V(H)$. We further define $e^a_i=w_{a\cdot(i(k-2)+1)}w_{a\cdot(i(k-2)+2)}\in E(H^a_i)$ for $i\in\{1,\ldots,k\}\cup \{\tau-k-1,\ldots \tau\}$ and  $e^a_i=v_{a\cdot(i(k-2)+1)}v_{a\cdot(i(k-2)+2)}\in E(H^a_i)$  for $k+1\leq i\leq \tau-k-2$. 
\end{const}

\vspace{2mm}

Note that each dilation chain $(H_i^a,e_i^a)_{i\in[\tau]}$  is a \emph{simple} $H$-chain of length $\tau$  on vertex set 
\[
   V^a:=\{w_a,\ldots, w_{a(k(k-2)+2)}, v_{a(k(k-2)+3)}, \ldots, v_{a(\tau-k-1)(k-2)}, w_{a((\tau-k-1)(k-2)+1)},\ldots,w_{a(\tau(k-2)+2)}\}\] 
   as in Definition \ref{def:simple-chain}, using here that the numbers $1,\ldots, \tau(k-2)+2$ are all distinct modulo $p$ and so their images under multiplication by $a\neq 0$ modulo $p$ are also distinct. 
 For $i\in[\tau-1]$ we have that  $e^a_i=\varphi^a_{i}(f')=\varphi^a_{i+1} (e')=V(H^a_i)\cap V(H^a_{i+1})$ and  $e^a_\tau=\varphi^a_\tau(f)$. As usual with simple chains, we take the canonical choice of origin for the chain  $(H_i^a,e_i^a)_{i\in[\tau]}$ to be $e_0^a:=w_aw_{2a}=\phi_1^a(e')$.  
 
The following observation is at the heart of Construction \ref{const:dil} and we will use it repeatedly. 

\begin{obs} \label{obs:edge equation}
    Suppose that $\cH=\cH_{\mathrm{Dil}}(H;A,p)$ and let $g=u_{\lambda_1}u_{\lambda_2}$ such that $u_{\lambda_{i'}}\in \{v_{\lambda_{i'}},w_{\lambda_{i'}}\}$ and $\lambda_{i'}\in \ZZ_p\setminus \{0\}$ for $i'=1,2$. If $a\in A$ is such that the chain $(H^a_i,e^a_i)_{i\in [\tau]}$ is proper with underlying graph $G^a$ and $g\in E(\final{G^a}_H)$, then there exists $\alpha\in \ZZ\setminus\{0\}$ with $|\alpha|\leq k$ such that \[\lambda_2-\lambda_1=\alpha a \mbox{ mod } p.\] 
  \end{obs}
    \begin{proof}
  Note that $\lambda_1\neq \lambda_2$ as the vertex set $V^a$ of $(H^a_i,e_i^a)_{i\in [\tau]}$ has no pair of vertices $v_\lambda,w_\lambda$ with the same index. As the chain $(H_i^a,e_i^a)_{i\in [\tau]}$ is proper, Lemma \ref{lem:final graph proper} gives that there is some $i_*\in [\tau]$ such that $g\in E(\final{H_{i_*}^a})$. This in turn implies that there are $j_1\neq j_2\in [k]$ such that $\lambda_1=a((i_*-1)(k-2)+j_1) \mbox{ mod } p$ and $\lambda_2=a((i_*-1)(k-2)+j_2) \mbox{ mod } p$ and so $\lambda_2-\lambda_1=a\alpha \mbox{ mod } p$ for some $\alpha =j_2-j_1\neq 0$ with $|\alpha|\leq k$.
     \end{proof}

     In particular, this has the following easy consequence.

      \begin{cor} \label{cor: dilation unique edge}
       Let $\cH=\cH_{\mathrm{Dil}}(H;A,p)=\{(H^a_i, e^a_i)_{i\in [\tau]}:a\in A\}$ be a collection of $H$-chains such that for each $a\in A$, the chain  $(H^a_i, e^a_i)_{i\in [\tau]}$  is proper and has  underlying graph $G^a=\bigcup_{i\in [\tau]}H^a_i$.   If  $A$ is $(k,2)$-fold solution-free then  
the graphs $\{\final{G^a}_H:a\in A\}$ are pairwise edge-disjoint. \end{cor}
\begin{proof}
     Let $g=u_{\lambda_1}u_{\lambda_2}$ such that $u_{\lambda_i}\in \{v_{\lambda_i},w_{\lambda_i}\}$ and $\lambda_i\in \ZZ_p\setminus \{0\}$ for $i=1,2$.
      Then if $g\in E(\final{G^a}_H)\cap E(\final{G^b}_H)$  for some $a\neq b\in A$, Observation \ref{obs:edge equation} gives that there are $\alpha_a, \alpha_b\neq 0$ with $|\alpha_a|,|\alpha_b|\leq k$ such that $\lambda_2-\lambda_1=\alpha_a a=\alpha_b b \mbox{ mod } p$. 
      Then $(a,b)$ form a non-trivial solution to the equation $\alpha_ax_1-\alpha_bx_2=0$, contradicting that $A$ is $(k,2)$-fold solution-free. 
\end{proof}

The collection $\cH_{\mathrm{Dil}}(H;A,p)$ has vertex set $V\cup W$ of size $2p-2$ and each chain has length $\tau= \Omega(p)$. We will use Lemma \ref{lem:linking chains} with appropriately chosen primes $p$ to obtain lower bounds on running times $M_H(n)$ for all $n\in \NN$. In fact, in our applications using Construction \ref{const:dil}, we will not always be working with $(2,1)$-inseparable $H$ and so we have to appeal to Lemma \ref{lem:linking chains plus} which allows us to derive lower bounds from proper collections of chains with the added condition that the start and end of each chain is disjoint from the vertices of all the other chains. This is precisely the reason for the additional vertex set $W$ in Construction \ref{const:dil} which is imposed to separate the start and end of each chain and can be ignored (with all the vertices mapped to $V$ in the natural way) in the case that Construction \ref{const:dil} is applied with $(2,1)$-inseparable $H$. 

\begin{lem} \label{lem:dil start end}
Suppose  that $H$ is a $k$-vertex graph, $p\geq k^3$ is a prime and $A\subset \ZZ_p$ is $(2k^2,2)$-fold solution-free. Then taking  $\cH_{\mathrm{Dil}}(H;A,p)=\{(H_i^a,e_i^a)_{i\in[\tau]}:a \in A\}$ to be the collection of dilation $H$-chains, each with underlying graph $G^a=\bigcup_{i\in [\tau]}H^a_i$, we have that 
	\begin{equation}\label{eq:additional dil}
		\bigcup_{i=1}^{k}V(H^a_i) \cap V(G^b) = \varnothing \text{ and } \bigcup_{i=\tau-k}^\tau V(H^a_i) \cap V(G^b) = \varnothing \text{ for } a\neq b\in A.
		\end{equation}            
\end{lem}
\begin{proof}
    Note that for each $a\in A$, we have that
    \[\bigcup_{i=1}^{k}V(H^a_i)  = \{w_a,\ldots, w_{a(k(k-2)+2)}\}\subseteq \{w_{a\alpha}:\alpha=1,\ldots,k^2 \}.  \] 
    Similarly,  \[\bigcup_{i=\tau-k}^\tau V(H^a_i)=  \{w_{a((\tau-k-1)(k-2)+1)},\ldots,w_{a(\tau(k-2)+2)}\}\subseteq \{w_{a\alpha}:\alpha=-1,\ldots,-2k^2 \} ,\]
as $(\tau-k-1)(k-2)+1\geq p-(k+2)(k-2)\geq p-2k^2$. Moreover,  all other vertices in $V(G^a)$ are vertices in $V$. Therefore if $a\neq b$ and there was some vertex lying in the intersection of $\cup_{i=1}^{k}V(H^a_i)\cup_{i=\tau-k}^{\tau}V(H^a_i)$ and $V(G^b)$, we would necessarily have that the vertex lies in $W$ and is of the form $w_{\lambda}$ such that  $\lambda=\alpha a \mbox{ mod } p$ and $\lambda=\beta b \mbox{ mod } p$ with $|\alpha|, |\beta|\leq 2k^2$. This contradicts that $A$ is $(2k^2,2)$-fold solution-free.
    \end{proof}

\subsection{Dilation chains for Behrendian graphs}
\label{sec: dil coh}
We now show that when $H-e$ is Behrendian  for any $e\in E(H)$ (Definition \ref{def:coherent}) and $A$ is a large $(K,3)$-fold solution-free set (Definition \ref{def:K-fold}) for an appropriate $K\in \NN$, Construction \ref{const:dil} can be used to give lower bounds on running times of the form $M_H(n)\geq n^{2-o(1)}$. 

\begin{thm} \label{thm:coherent almost quad}
    Suppose that $4\leq k\in \NN$ and $H$ is a graph with $V(H)=[k]$ such that for all $\tau\in \NN$, the simple $H$-chain $(H_i,e_i)_{i\in[\tau]}$ is proper. 
    If $H-e$ is Behrendian for any choice of $e\in E(H)$, then we have that $M_H(n)\geq n^{2-O(1/\sqrt{\log n})}=n^{2-o(1)}$.
\end{thm}

Before proving Theorem \ref{thm:coherent almost quad}, we discuss its implications. In Lemma \ref{lem: H minus coherent}, we showed that if $H$ is an odd wheel graph, contains the square of a Hamilton cycle, or $\delta(H)\geq k/2+1$, then $H-e$ is Behrendian for any choice of $e\in E(H)$. In the latter two cases, the graphs are $(2,1)$-inseparable (Observation \ref{obs:mindegree_inseparable and square ham}) and so simple chains of arbitrary length exist for $H$ by  Observation \ref{obs: disjoint edges} and Lemma \ref{lem:inseparableimpliesproper} implies that these chains are proper. In the case that $H$ is an odd wheel with at least $8$ vertices, Lemma \ref{lem:wheel simple} gives the existence of proper simple chains of arbitrary length. Therefore in all these cases of $H$, the conditions of Theorem \ref{thm:coherent almost quad} are satisfied and thus Theorems \ref{thm;dense quad}  (part \eqref{thm: min deg almost quad}), \ref{thm: square ham cycle} and the lower bound of Theorem \ref{thm:wheel} are all corollaries. 

\begin{proof}[Proof of Theorem \ref{thm:coherent almost quad}]
    Let $H$ be a graph with $V(H)=[k]$ as in the statement of the theorem and let $n$ be sufficiently large. Let $p$ be a prime such that $n/8k^2\leq p \leq n/4k^2$ (which exists by Bertrand's postulate/Chebyshev's Theorem). Next we take $A\subset \ZZ_p\setminus \{0\}$ to be a $(2k^2,3)$-fold solution-free subset of size $p^{1-O(1/\sqrt{\log p})}$, as guaranteed by Theorem \ref{thm:behrend extended}, and take 
    \[\cH:=\cH_{\mathrm{Dil}}(H;A,p)=\{(H^a_i,e^a_i)_{i\in[\tau]}:a\in A\},\]
    to be the collection of chains given by Construction \ref{const:dil}, with vertex set $V
    \cup W$  as in the construction.  The chains in $\cH$ are simple and thus proper by assumption. 
    For each $a\in A$, let  $G^a=\bigcup_{i\in [\tau]}H^a_i$ be the underlying graph of the chain. Then we have that 
	\begin{equation}\label{eq:additional 3}
		\bigcup_{i=1}^{k}V(H^a_i) \cap V(G^b) = \varnothing \text{ and } \bigcup_{i=\tau-k}^\tau V(H^a_i) \cap V(G^b) = \varnothing \text{ for } a\neq b\in A,
		\end{equation}
by Lemma \ref{lem:dil start end} and the fact that $A$ is $(2k^2,3)$-fold solution-free (and so $(2k^2,2)$-fold solution-free). Therefore, by Lemma \ref{lem:linking chains plus}, if $\cH$ is a proper collection of simple $H$-chains then there is a proper chain on at most $2k^2(2p-2)\leq n$ vertices, of length at least \[\tau |A|=\left\lfloor\frac{p-3}{k-2}\right\rfloor p^{1-O(1/\sqrt{\log p})}=n^{2-O(1/\sqrt{\log n})},\]
which in turn proves the theorem via an application of Lemma \ref{lem:chain_runningtime}. 

Thus it remains to prove that $\cH$ is a proper collection of simple $H$-chains. We will in fact show that the collection is  strongly proper (Definition \ref{def:stronger-proper}) which suffices by Lemma \ref{lem:strong implies proper}. The collection satisfies property \ref{chaincol:strong proper edge} (by Corollary \ref{cor: dilation unique edge}) of being strongly proper. We thus only need to show property \ref{chaincoll:strong proper H-e} of Definition \ref{def:stronger-proper} and so we fix some $e\in E(H)$ and a copy  $F$ of $H-e$ in $\hat G:=\bigcup_{a\in A}\final{G^a}_H$. We will show that there is a  $b\in A$ such that $F\subseteq \final{G^b}_H$. 

Colour the edges of $F$ with colours from $A$ such that each edge $g\in E(F)$ receives the colour $a\in A$ such that $g\in E(\final{G^{a}}_H)$, noting that there is just one choice of $a$ due to Corollary \ref{cor: dilation unique edge}. Supposing for a contradiction that $F$ is not contained in $\final{G^b}_H$ for some $b\in A$, we have that this colouring is non-monochromatic. Therefore, as $F$ is Behrendian by assumption, there is some non-monochromatic cycle $C$ which is a subgraph of $F$, such that $C$ is the union at most three monochromatic paths. 

Let  $P_1$ be some monochromatic path featuring in $C$ with the edges of $P_1$ belonging to $\final{G^{a_1}}_H$. Suppose that 
the endpoints of $P_1$ are $u_{\rho}\in \{v_{\rho},w_{\rho}\}$ and $u_{\mu}\in \{v_{\mu},w_{\mu}\}$ with $\rho, \mu \in [p-1]$. Note that $\rho \neq \mu$ as $P_1$ is a path in $\final{G^{a_1}}_H$ which has vertex set $V^{a_1}$ which contains no pairs $v_\lambda,w_\lambda$ with the same index. Now we claim that there is some $\beta_1\neq 0$ with $|\beta_1|\leq k^2$ such that $\mu-\rho= \beta_1a_1 \mbox{ mod } p$.
Indeed, we can label the vertices of the path $P_1$ as $u_{\lambda_1},\ldots,u_{\lambda_{\ell}}$ 
with $u_{\lambda_j}\in \{v_{\lambda_j},w_{\lambda_j}\}$ for $j=1,\ldots,\ell$ and $\lambda_1=\rho, \lambda_\ell=\mu$. 
Note that we necessarily have that $2\leq \ell \leq k$ as $P$ is a path contained in $F$. By Observation \ref{obs:edge equation}, we have that for $j=1,\ldots,\ell-1$, there is some $\alpha_j$ with $|\alpha_j|\leq k$ such that $\lambda_{j+1}-\lambda_j=\alpha_ja_1 \mbox{ mod } p$. Hence, again modulo $p$, we have that 
\[\mu-\rho=\lambda_\ell-\lambda_1=(\lambda_\ell-\lambda_{\ell-1})
+\ldots+(\lambda_2-\lambda_1)=(\alpha_{\ell-1}+\ldots+\alpha_1)a_1.\]
Therefore, we can fix $\beta=\sum_{j=1}^{\ell-1}\alpha_j$ and we have that $|\beta_1|\leq (\ell-1)\max_{j}|\alpha_j|\leq (\ell-1)k\leq k^2$ and $\beta_1\neq 0$ as $\rho \neq \mu$. 

If $C$ is the union of two monochromatic paths, then we have some $a_2\neq a_1$ and a path $P_2$ in $\final{G^{a_2}}_H$ with endpoints $u_\rho$ and $u_\mu$ and so the above argument gives some $\beta_2\neq 0$ with $|\beta_2|\leq k^2$ such that $\mu-\rho =\beta_2a_2 \mbox{ mod } p$. However, then we get that $(a_1,a_2)\in A^2$ is a solution to the equation $\beta_1x_1-\beta_2 x_2=0 \mbox{ mod } p$. This solution is non-trivial as $\beta_1,\beta_2\neq 0$ and $a_1\neq a_2$. This contradicts that $A$ is $(k^2,2)$-fold solution-free. Similarly, if $C$ is the union of 3 monochromatic paths $P_1,P_2,P_3$, then we have 3 distinct indices $\rho,\mu,\nu\in [p-1]$ and integer coefficients $\beta_1,\beta_2, \beta_3\neq 0$ with $|\beta_i|\leq k^2$ for $i\in [3]$, such that 
\[\mu-\rho=\beta_1a_1 \mbox{ mod } p, \qquad \rho-\nu=\beta_2a_2 \mbox{ mod } p, \qquad \nu-\mu=\beta_3a_3 \mbox{ mod } p.\]
This contradicts that $A$ is $(k^2,3)$-fold solution-free as we have that $(a_1,a_2,a_3)\in A^3$ forms a solution to the equation $\beta_1x_1+\beta_2x_2+\beta_3x_3=0 \mbox{ mod } p$, which is non-trivial due to the fact that the $a_i$ are distinct and the $\beta_i$ are non-zero. This shows that no such non-monochromatic cycle $C$ can exist and so $F$ must indeed be contained in $\final{G^b}_H$ for some $b\in A$, establishing property \ref{chaincoll:strong proper H-e} of being a strongly proper collection, and completing the proof.
    \end{proof}

\subsection{Constructing bipartite dilation chains} \label{sec:bip dil chains}

In the case that $H$ is bipartite, we utilise the bipartite structure for our collection of chains. We emphasise that the two vertex sets $V_X$ and $V_Y$ appearing in the bipartite Construction \ref{const: bip dil} below are \emph{not} analogous to the two sets $V$ and $W$ appearing in Construction \ref{const:dil} above. Indeed the sets $V_X$ and $V_Y$ correspond to the two parts of the bipartition of the graph $H$ whilst the set $W$ in Construction \ref{const:dil} was introduced to separate the ends of each chain from the other chains. We could add sets $W_X$ and $W_Y$ to Construction \ref{const: bip dil} below to achieve the same purpose. This would lead to a construction that can handle bipartite graphs $H$ that are not $(1,1,1)$-inseparable but nonetheless have proper simple chains (via an application of Lemma \ref{lem:linking chains plus}). However, as we will only apply the  Construction \ref{const: bip dil} to $(1,1,1)$-inseparable graphs, we omit the extra sets $W_X$ and $W_Y$ for simplicity.

\vspace{2mm}

\begin{const}[Collection of dilation chains - Bipartite case] \label{const: bip dil}
 Let $4\leq k \in \NN$ and suppose $H$ is a $k$-vertex bipartite graph with vertex parts $X,Y$ such that   $V(H)=X\cup Y\subseteq[2k]$ with $\{1,2k-1\}\subseteq X\subseteq \{1,3,\ldots,2k-1\}$,  $\{2,2k\}\subseteq Y\subseteq \{2,4,\ldots,2k\}$ and both  $e':=\{1,2\}$ and $f':=\{k-1,k\}$ are  edges of $H$.   Moreover let $m\in \NN \cup \{0\}$, $p\in \NN$  a prime with $p\geq m+2k^3$ and  $A\subseteq \ZZ_p\setminus {0}$ a  set of \emph{dilations}. Then fixing $\tau:=\floor{\tfrac{p-1-2k-m}{2k-2}}$,  the collection  \[\cH_{\mathrm{BipDil}}(H;A,p,m)=\{(H^a_i,e^a_i)_{i\in[\tau]}:a\in A\}\] of $H$-chains is defined as follows. 
    Let $V_X=\{x_1,\ldots,x_{p-1}\}$ and $V_Y=\{y_1,\ldots,y_{p-1}\}$ be two vertex sets indexed by $\ZZ_p\setminus \{0\}$ (with indices always considered modulo $p$). For 
    $a\in A$ and $i\in [\tau]$ the graph $H_i^a$ is the image of the graph isomorphism  $\varphi^a_i : H\rightarrow H^a_i$ defined via the map
		\begin{equation}
    \varphi^a_i(j)=\begin{cases}
		x_{a\cdot((i-1)(2k-2)+j)}	, & \text{ if } j\in X\subseteq [2k]  ; \\
           y_{a\cdot(i(2k-2)+j+m)}, & \text{ if } j\in Y\subseteq [2k]. 
		 \end{cases}
      \end{equation}
      We further define $e^a_i=x_{a\cdot(i(2k-2)+1)}y_{a\cdot((i+1)(2k-2)+2+m)}\in E(H^a_i)$ for $i\in
      [\tau]$. 
\end{const}

\vspace{2mm}

As with the non-bipartite Construction \ref{const:dil}, we have that each dilation chain $(H_i^a,e_i^a)_{i\in[\tau]}$  from Construction \ref{const: bip dil} is a simple $H$-chain of length $\tau$  on vertex set $V_X^a\cup V_Y^a$ where
\[
   V_X^a:=\{x_{a\ell}:\ell\in [\tau(2k-2)+2] \mbox{ and } \ell\mbox{ mod }(2k-2) \in X\} \]
   and \[V_Y^a:=\{y_{a(\ell+2k-2+m)}:\ell\in [\tau(2k-2)+2] \mbox{ and } \ell\mbox{ mod }(2k-2) \in Y\}.\] 
   We again take the canonical choice of origin for the chain  $(H_i^a,e_i^a)_{i\in[\tau]}$ to be $e_0^a:=x_ay_{a(2k+m)}=\phi_1^a(e')$.  We also note that,  taking $G^a=\bigcup_{i\in [\tau]}H^a_i$ to be the underlying graph for each chain $(H_i^a,e_i^a)_{i\in [\tau]}$ with $a\in A$, we have that $\bigcup_{a\in A}G^a$ is bipartite, with each edge between a vertex in $V_X$ and $V_Y$. 

   Finally, we remark about the shift in indices for the vertices in $V_Y$. Indeed, compared to their $V_X$ counterparts, we have that the vertices in $V_Y$ have all their indices translated by a factor of $2k-2+m$ before then being multiplied by some $a\in A$. The reason for this is highlighted by the following observation, which is the bipartite counterpart of Observation \ref{obs:edge equation}. 

   \begin{obs} \label{obs: dil edge equation bip}
       Suppose $H$ is $(1,1,1)$-inseparable, $\cH=\cH_{\mathrm{BipDil}}(H;A,p,m)$ and $g=x_{\lambda}y_{\mu}$ such that $\lambda,\mu \in \ZZ_p\setminus \{0\}$. If $G^a$ is the underlying graph for the chain $(H^a_i,e^a_i)_{i\in [\tau]}\in \cH$   and $g\in E(\final{G^a}_H)$, then there exists  $\alpha\in \{m+1,m+3,\ldots,m+4k-3\}$ such that \[\mu-\lambda=\alpha a \mbox{ mod } p.\] 
  \end{obs}
    \begin{proof}
 As $H$ is $(1,1,1)$-inseparable, the simple chain $(H_i^a,e_i^a)_{i\in [\tau]}$ is proper by Lemma \ref{lem:inseparableimpliesproper}. Therefore Lemma \ref{lem:final graph proper} gives that there is some $i_*\in [\tau]$ such that $g\in E(\final{H_{i_*}^a})$. This in turn implies that there are $j_x\in X\subseteq [2k]$ and $ j_y\in Y\subseteq  [2k]$ such that $\lambda=a((i_*-1)(2k-2)+j_x) \mbox{ mod } p$ and $\mu=a(i_*(2k-2)+j_y+m) \mbox{ mod } p$ and so 
 \[\mu-\lambda=a(i_*(2k-2)+j_y+m) - a((i_*-1)(2k-2)+j_x)=
 a((2k-2)+j_y-j_x+m) \mbox{ mod }p.\]
Now defining $\alpha:=(2k-2)+j_y-j_x+m$, we have that  $\alpha\neq m \mbox{ mod }2$ as $j_y\in Y\subseteq \{2,4,\ldots,2k\}$ and $j_x\in X \subseteq \{1,3,\ldots,2k-1\}$ have different parities.  Moreover $\alpha\in \{m+1,m+3,\ldots,m+4k-3\}$ as required because 
$-2k+3\leq j_y-j_x\leq 2k-1$ with the minimum being achieved when $j_y=2$ and $j_x=2k-1$ and the maximum when $j_y=2k$ and $j_x=1$. 
\end{proof}

   The shift of $2k-2+m$ in the $V_Y$ indices actually occurs for two reasons. Firstly, the shift of $2k-2$ guarantees that edges that appear in $\final{H^a_i}_H$ for some $a\in A$ and $i\in [\tau]$ have \textit{`positive slope'}, that is, that the $\alpha$ in Observation \ref{obs: dil edge equation bip} is positive.
    This turns out to be crucial in our analysis showing that there are no unwanted copies of $F=H-e$ for any $e\in E(H)$, when proving that collections of dilation chains are proper. The reason for the extra shift by $m$ is more technical and stems from our need to consider equations in which there is also a lower bound on the size of the coefficients as in Theorem \ref{thm: sidon extended}. We also remark that we place $X$ on odd integers and $Y$ on even integers, in order to have $\alpha \neq m \mbox{ mod }2$ which is needed to apply Theorem \ref{thm: sidon extended}.  For other applications of Construction \ref{const: bip dil} such as Theorem \ref{thm:complete bip} or indeed if the sets $A$ in Theorem \ref{thm: sidon extended} can be replaced by $(K,4)$-fold solution-free sets, we can simply take $m=0$ when appealing to Construction \ref{const: bip dil} and take $V(H)=X\cup Y=[k]$, adjusting the construction accordingly. 

  As in Corollary \ref{cor: dilation unique edge} in the non-bipartite case, Observation \ref{obs: dil edge equation bip} has the  simple consequence that under mild assumptions on the dilation set $A$, the edges that feature in the process on these chains belong to a unique chain. 

   \begin{cor} \label{cor:bip dilation unique edge}
       Let $H$ be $(1,1,1)$-inseparable and $\cH=\cH_{\mathrm{BipDil}}(H;A,p,m)$  a collection of $H$-chains such that for  $a\in A$, the chain  $(H^a_i, e^a_i)_{i\in [\tau]}$   has  underlying graph $G^a=\bigcup_{i\in [\tau]}H^a_i$.   If $A$ is $(m+4k,2)$-fold solution-free, then the graphs $\{\final{G^a}_H:a\in A\}$ are pairwise edge-disjoint.   
\end{cor}
 \begin{proof}
     If we have an edge $g=x_{\lambda}y_{\mu} \in E(\final{G^a}_H)\cap E(\final{G^b}_H)$ then Observation \ref{obs: dil edge equation bip} implies that $\mu-\lambda=a\alpha=b\beta$ for $\alpha,\beta\in \{m+1,\ldots, m+4k-3\}$. This is prohibited for $a\neq b$ by the fact that $A$ has no non-trivial solutions to $\alpha x-\beta b=0$. 
   \end{proof}

As was the case for the proof of Theorem \ref{thm:coherent almost quad} using the  Construction \ref{const:dil} of dilation chains, proving lower bounds on running times using the bipartite Construction \ref{const: bip dil} will reduce to showing that certain configurations in the final graph $\hat G:=\bigcup_{a\in A}\final{G^a}_H$ lead to non-trivial solutions to equations. Before embarking on the full applications of Construction \ref{const: bip dil}, we show that using a set of dilations that is $(K,3)$-fold solution-free for an appropriate $K$ already greatly restricts the types of $C_4$ that we see in the final graph. 

\begin{lem} \label{lem:C4s in dil bip}
       Let $H$ be $(1,1,1)$-inseparable and $\cH=\cH_{\mathrm{BipDil}}(H;A,p,m)$  such that for  $a\in A$, the chain  $(H^a_i, e^a_i)_{i\in [\tau]}\in \cH$   has  underlying graph $G^a=\bigcup_{i\in [\tau]}H^a_i$.   If $A$ is $(2m+8k,3)$-fold solution-free, and $C$ is a copy of $C_4$ in $\hat G:=\bigcup_{a\in A}\final{G^a}_H$, then either $C$ is entirely contained in $\final{G^b}_H$  for some $b\in A$ or each edge of $C$ belongs to a distinct graph from  $\{\final{G^a}_H:a\in A\}$. 
\end{lem}
\begin{proof}
    Note that $\hat G$ is bipartite with parts $V_X$ and $V_Y$ due to Lemma \ref{lem:staying bipartite} as $\bigcup_{a\in A}G^a$ is bipartite and $H$ is $(1,1,1)$-inseparable. 
    Let $x_{\lambda_1},x_{\lambda_2}\in V_X$ and $y_{\mu_1},y_{\mu_2}\in V_Y$ be the vertices of $C$. 
    Further, for each choice of  $i,j\in \{1,2\}$  let $g_{ij}=x_{\lambda_i}y_{\mu_j}\in E(C)$ and  $a_{ij}\in A$ be such that $g_{ij}\in \final{G^{a_{ij}}}_H$, noting that the choice of $a_{ij}$ is unique due to Corollary \ref{cor:bip dilation unique edge}. By Observation \ref{obs: dil edge equation bip}, we have that for each choice of $i,j\in \{1,2\}$ there is some  $\alpha_{ij}\in [m+4k]$ with  $\mu_j-\lambda_i=\alpha_{ij}a_{ij} \mbox{ mod }p$. In turn, this gives that  
    \begin{equation} \label{eq: C4 dil sum}
       \alpha_{11}a_{11}-\alpha_{12}a_{12}+\alpha_{22}a_{22}-\alpha_{21}a_{21}=(\mu_1-\lambda_1)+(\lambda_1-\mu_2)+(\mu_2-\lambda_2)+(\lambda_2-\mu_1)=0\end{equation}
  with equalities being  modulo $p$. If the values of $a_{ij}$ are all the same or all different we are done, as this corresponds to the two situations in the statement of the lemma. For a contradiction then, suppose that there is some set $B\subseteq A$ with $|B|\in \{2,3\}$ and $a_{ij}\in B$ for all $i,j\in \{1,2\}$. This implies that 
  \begin{equation} \label{eq:beta b}
      \sum_{b\in B}\beta_b b=0 \mbox{ with } \beta_b:=\sum\left\{(-1)^{i+j}\alpha_{ij}:(i,j)\in \{1,2\}^2 \mbox{ and } a_{ij}=b\right\} \mbox{ for each } b\in B.
  \end{equation}
  Now note that $\beta_b$ is the sum of at most three values in $\{\alpha_{11},-\alpha_{12},-\alpha_{21},\alpha_{22}\}$, at most two of which have the same sign and so $|\beta_b|\leq 2\max_{ij}|\alpha_{ij}|\leq 2(2m+4k)=2m+8k$. Thus the collection $B$ forms a solution to a two or three variable equation modulo $p$ with coefficients bounded by $2m+8k$. This is a contradiction to the fact that $A$ is $(2m+8k,3)$-fold solution-free, as long as the solution given by $B$ is non-trivial. As all the values in $B$ are distinct, for non-triviality it suffices to show that we do not have  $\beta_b= 0$ for all $b\in B$. If there is some $b'$ such that $\beta_{b'}\in\{\alpha_{11},-\alpha_{12},-\alpha_{21},\alpha_{22}\}$  then certainly  $\beta_{b'}\neq 0$ as all the  $\alpha_{ij}$ are non-zero. Therefore the remaining case to consider is that $|B|=2$ and each element of $B$ is the sum of two values in $\{\alpha_{11},-\alpha_{12},-\alpha_{21},\alpha_{22}\}$. Letting $b\in B$ be such that $a_{11}=b$, we have that $\beta_b\in \{\alpha_{11}-\alpha_{12},\alpha_{11}-\alpha_{21},\alpha_{11}+\alpha_{22}\}$. Note that $\alpha_{11}-\alpha_{12}=\mu_1-\mu_2\neq 0$ as $\mu_1\neq \mu_2$. Similarly $\alpha_{11}-\alpha_{21}=\lambda_2-\lambda_1\neq 0$. Finally $\alpha_{11}+\alpha_{22}\neq 0$ as both $\alpha_{11},\alpha_{22}\geq 1$. This shows that the solution given by $B$ cannot be trivial and completes the proof.  
\end{proof}

We remark that the final argument of the proof of the previous lemma, where we showed that $\alpha_{11}+\alpha_{22}\neq 0$, is precisely where we need the `positive slope' property, namely that the $\alpha$ values produced by Observation \ref{obs: dil edge equation bip} are always positive. This is why we shift the indices for $V_Y$ by $2k-2$ and without this, we cannot guarantee that there are not copies of $C_4$ whose edges alternate between some pair of graphs $\final{G^a}_H$ and $\final{G^b}_H$.  Such copies of $C_4$ would then be troublesome in our proof showing that there are no  copies $F$ of  $H-e$ such that the edges of  $F$ span multiple graphs $\final{G^a}_H$. 

In general when appealing to Construction \ref{const: bip dil}, we will need to show that there are no such $F$ spanning multiple $\final{G^a}_H$. Whilst Lemma \ref{lem:C4s in dil bip} will be able to rule out many of these, the lemma says nothing about `rainbow' configurations of $F$ or $C_4$, where each edge belongs to a distinct $\final{G^a}_H$. In the next two sections, we use two different strategies to handle these. The first, for Sidonian graphs in Theorem \ref{thm:bip coherent run time}, uses a dilation set $A$ that also avoids certain equations in four variables. The second approach, used for complete bipartite graphs in proving Theorem \ref{thm:complete bip}, uses a random sparsification of the set of dilations in Lemma \ref{lem:C4s in dil bip}.
 
\subsection{Dilation chains for Sidonian graphs} \label{sec: dil bip coh}
We now apply Construction \ref{const: bip dil} to get lower bounds on running times for graphs $H$ such that $H-e$ is Sidonian for any choice of $e\in E(H)$, recalling the definition of Sidonian graphs from Definition \ref{def:bicoh}. 

\begin{thm} \label{thm:bip coherent run time}
    If $H$ is $(1,1,1)$-inseparable and $H-e$ is Sidonian for any choice of $e\in E(H)$, then we have that $M_H(n)\geq n^{3/2-o(1)}$.
\end{thm}

    Any bipartite graph with parts $X,Y$ with $|X|=r\geq 3$, $|Y|=s\geq 3$  such that $d(x)\geq s/2+1$ for all $x\in X$ and $d(y)\geq r/2+1$ for all $y\in Y$, satisfies the conditions of Theorem \ref{thm:bip coherent run time} and so Theorem \ref{thm:bip dense} follows as a corollary. Indeed the fact that such graphs are $(1,1,1)$-inseparable was shown in Lemma \ref{lem:bip dense insep} and the Sidonian condition  was shown in Lemma \ref{lem:dense bip coh}. 

\begin{proof}
    Let $k=v(H)$ and $K=4k$. Moreover, let $\eps>0$ be arbitrary, and let $m\in \NN$ be the integer output by Theorem \ref{thm: sidon extended} with input $\eps$. 
    Next, let $n$ be sufficiently large and let $p$ be a prime such that $n/8k^2\leq p \leq n/4k^2$ (which exists by Bertrand's postulate/Chebyshev's Theorem). We take $A\subset \ZZ_p\setminus \{0\}$ to be a $(2(m+K),3)$-fold solution-free subset of size $p^{1/2-\eps}$, with no 
    non-trivial solutions to equations of the form 
    \begin{equation} \label{eq:c4 in proof bip coh} 
\alpha_1x_1+\alpha_2x_2+\alpha_3x_3+\alpha_4x_4=0  \end{equation}
    with $|\alpha_i|\in \{m+1,m+3,\ldots,m+K-1\}$ for $i\in [4]$.
    The existence of such a set $A$ is guaranteed by Theorem \ref{thm: sidon extended}. Finally, we take 
    \[\cH:=\cH_{\mathrm{BipDil}}(H;A,p,m)=\{(H^a_i,e^a_i)_{i\in[\tau]}:a\in A\},\]
    to be the collection of chains given by Construction \ref{const: bip dil}, with vertex set $V_X
    \cup V_Y$. Each chain in the collection $\cH$ is proper  by  Lemma \ref{lem:inseparableimpliesproper}   as $H$ is $(1,1,1)$-inseparable and  the chains  are simple. For each $a\in A$, let  $G^a=\bigcup_{i\in [\tau]}H^a_i$ be the underlying graph of the chain corresponding to $a\in A$. By Lemma \ref{lem:linking chains}, using that $H$ is $(1,1,1)$-inseparable and $\bigcup_{a\in A}G^a$ is bipartite, if $\cH$ is a proper collection of  $H$-chains then there is a proper chain on at most $2k^2(2p-2)\leq n$ vertices, of length at least \[\tau |A|=\left\lfloor\frac{p-1-2k-m}{2k-2}\right\rfloor p^{1/2-\eps}\geq n^{3/2-2\eps},\]
using in the last inequality that $n$ is sufficiently large in terms of $\eps,m,k$. As $\eps>0$ was arbitrary, the theorem follows via an application of Lemma \ref{lem:chain_runningtime}. 

As in the proof of Theorem \ref{thm:coherent almost quad}, we will show that $\cH$ is in fact strongly proper, which will conclude the proof. For property \ref{chaincol:strong proper edge}, we can appeal to Corollary \ref{cor:bip dilation unique edge} as $A$ is certainly $(m+4k,2)$-fold solution-free. For property \ref{chaincoll:strong proper H-e}, we fix some $e\in E(H)$ and a copy $F$ of $H-e$ in $\hat G:=\bigcup_{a\in A}\final{G^a}_H$. Colour the edges $g\in E(F)$ with elements in $A$ such that $g$ receives colour $a\in A$ if  $g\in E(\final{G^a}_H)$, noting that each edge receives one colour by Corollary \ref{cor:bip dilation unique edge}. If there is some $b\in A$ such that all edges of $F$ are coloured with $b$ then we have that $F\subseteq \final{G^b}_H$ and we are done. Suppose for a contradiction then, that $F$ is coloured in a non-monochromatic fashion.  Therefore, as $F$ is Sidonian by assumption, there is some copy $C$ of $C_4$ in $F$ that is non-monochromatic. By Lemma \ref{lem:C4s in dil bip}, as $A$ is $(2m+8k,3)$-fold solution-free, we must have that $C$ is \emph{rainbow}, that is, each edge of $C$ is coloured by a different $a\in A$. Label the vertices of $C$ as $x_{\lambda_1},x_{\lambda_2}\in V_X$ and $y_{\mu_1},y_{\mu_2}\in V_Y$, and edges of $C$ as $e_{ij}=x_{\lambda_i}y_{\mu_j}\in E(C)$ with $e_{ij}$ coloured by $a_{ij}\in A$ for $i,j\in \{1,2\}$. By Observation \ref{obs: dil edge equation bip}, we have that for each choice of $i,j
\in \{1,2\}$ there is some $\alpha_{ij}\in \{m+1,m+3,\ldots,m+4k-3\}$ such that $\mu_j-\lambda_i=\alpha_{ij}a_{ij}$. But then we have that 
$(a_{11},a_{12},a_{22},a_{21})\in A^4$ form a solution to the equation \[\alpha_{11}x_1-\alpha_{12}x_2+\alpha_{22}x_3-\alpha_{21}x_4=0 \mbox{ mod }p,\]
which is necessarily non-trivial as each of the $a_{ij}$ are distinct and the $\alpha_{ij}$ are non-zero. This contradicts that $A$ is free from non-trivial solutions to equations as in \eqref{eq:c4 in proof bip coh} and completes the proof.
\end{proof}

\subsection{Dilation chains for complete bipartite graphs}
\label{sec: dil comp bip}

For our second application of Construction \ref{const: bip dil}, we prove Theorem \ref{thm:complete bip} giving a lower bound on $M_H(n)$ when $H$ is a complete bipartite graph. 

\begin{proof}[Proof of Theorem \ref{thm:complete bip}]
    Let $3\leq r\leq s$ and let $H'$ be the graph obtained from $H=K_{r,s}$ by removing a single edge.  Fix $m:=0$, $k:=v(H)=r+s$ and
    let $n$ be sufficiently large.  Fix $p$ to be a prime such that $n/8k^2\leq p \leq n/4k^2$ and take $A\subset \ZZ_p\setminus \{0\}$ to be a $(8k,3)$-fold solution-free subset of  size $N:=p^{1-O(1/\sqrt{\log p})}$, as guaranteed by Theorem \ref{thm:behrend extended}. Moreover, let 
    $\cH:=\cH_{\mathrm{BipDil}}(H;A,p,0)=\{(H^a_{i},e^a_i)_{i\in [\tau]}:a\in A\}$ be as in Construction \ref{const: bip dil} and let $G^a:=\bigcup_{i\in [\tau]}H^a_i$ be the corresponding underlying graphs. Letting $\hat G:=\bigcup_{a\in A}\final{G^a}_H$, we claim that there is some subset $A'\subseteq A$ such that 
    \begin{enumerate}
        \item  \label{item: A' size} $|A'|\geq \tfrac{1}{4} N \left(\tfrac{N}{p^k}\right)^{\tfrac{1}{r(s-1)-1}}=p^{1+\tfrac{1-k}{r(s-1)-1}-O\left(\tfrac{1}{\sqrt{\log p}}\right)}$,  and 
        \item \label{item:A' colours} For any copy $F$ of $H'$ in $\hat G$, there is a collection $B\subseteq A'$ with $|B|<r(s-1)$ such that $F\subseteq \bigcup_{b\in B}\final{G^b}_H$. 
    \end{enumerate}

Indeed, fix $q\in (0,1)$ with $q=\tfrac{1}{2}\left(\tfrac{N}{p^k}\right)^{\tfrac{1}{r(s-1)-1}}$ and consider a random set $A_q\subseteq A$ with each $a\in A$ being a member of $A_q$ independently with probability $q$. For each copy $F$ of $H'$ in $\bigcup_{a\in A_q}\final{G^a}_H$, let $B_F$ be the collection of $b\in A_q$ such that $F$ has an edge belonging to $\final{G^b}_H$, noting that each edge belongs to  exactly one $\final{G^b}_H$ by Corollary \ref{cor:bip dilation unique edge}. Now let $A'_q\subseteq A_q$ be a subset of $A_q$ obtained by deleting one $b\in B_F\subseteq A_q$ for every copy $F$ of $H'$ such that $|B_F|\geq r(s-1)$. For each $r(s-1)\leq t\leq rs-1$, letting $\cF^t_q:=\{F\subseteq \bigcup_{a\in A_q}\final{G^a}_H \mbox{ copies of }H': |B_F|=t\},$ we have that
\begin{equation} \label{eq:bad F expec}
\mathbb{E}\left[|\cF^t_q|\right]\leq p^{k}q^t\leq p^kq^{r(s-1)}=
p^k\frac{1}{2^{r(s-1)}}\left(\frac{N}{p^k}\right)^{1+\tfrac{1}{r(s-1)-1}}
\leq \frac{1}{4r} N \left(\frac{N}{p^k}\right)^{\tfrac{1}{r(s-1)-1}},    
\end{equation}
using that the number of copies of $F$ in $\hat G$ is at most $p^k$ and each copy of $F$ which spans exactly $t$ of the graphs $\final{G^a}_H$, lands in $\cF_q^t$ with probability $q^t$. Therefore we have that 
\begin{equation} \label{eq:expectation comp bip}
\mathbb{E}[|A'_q|]\geq \mathbb{E}[|A_q|]-\sum_{t=r(s-1)}^{rs-1}\mathbb{E}[|\cF_t|]\geq q|A|- \frac{1}{4} N \left(\frac{N}{p^k}\right)^{\tfrac{1}{r(s-1)-1}}\geq \frac{1}{4} N \left(\frac{N}{p^k}\right)^{\tfrac{1}{r(s-1)-1}},\end{equation}
and so choosing   an instance of $A_q$   in which the lower bound of \eqref{eq:expectation comp bip} is achieved and fixing $A'=A'_q$, we get that indeed $A'$ satisfies both  properties \eqref{item: A' size} and \eqref{item:A' colours}. 

 Now with such an $A'$ fixed,  we take 
    \[\cH':=\cH_{\mathrm{BipDil}}(H;A',p,0)=\{(H^a_i,e^a_i)_{i\in[\tau]}:a\in A'\}.\]
   The chains in $\cH'$ are simple and thus proper by  Lemma \ref{lem:inseparableimpliesproper}  as $H$ is $(1,1,1)$-inseparable (Lemma \ref{lem:bip dense insep}). By Lemma \ref{lem:linking chains}, using that $H$ is $(1,1,1)$-inseparable and $\bigcup_{a\in A}G^a$ is bipartite, if $\cH'$ is a proper collection of  $H$-chains then there is a proper chain on at most $2k^2(2p-2)\leq n$ vertices, of length at least \[\tau |A|=\left\lfloor\frac{p-1-2k-m}{2k-2}\right\rfloor p^{1+\tfrac{1-k}{r(s-1)-1}-O\left(\tfrac{1}{\sqrt{\log p}}\right)}= n^{2-\tfrac{r+s-1}{r(s-1)-1}-o(1)},\]
and so the theorem follows via an application of Lemma \ref{lem:chain_runningtime}. 

As in the proof of Theorem \ref{thm:bip coherent run time}, we show that $\cH'$ is actually strongly proper and condition  \ref{chaincol:strong proper edge}  follows immediately from Corollary \ref{cor:bip dilation unique edge} upon noticing that $A'$ is $(4k,2)$-fold solution-free as a subset of $A$. To finish, we show Condition \ref{chaincoll:strong proper H-e} of being strongly proper and fix some copy $F$ of $H'$ in $\bigcup_{a\in A'}\final{G^a}_{H}$. Let $F'$ be the graph obtained from $F$ by deleting the vertex of degree $r-1$ from the side of $F$ of size $s$, so that $F'$ is a copy of $K_{r,s-1}$. Note that due to property \eqref{item:A' colours} of the set $A'$, we cannot have that  all $r(s-1)$ edges of $F'$ belong to distinct graphs in the collection $\{\final{G^a}_H:a\in A'\}$ and so there exists some $b\in A'$ and two distinct edges $e_1,e_2\in E(F')$ such that $e_1,e_2\in E(\final{G^b}_H)$. We will show that $F\subseteq \final{G^b}_H$ and so $\cH'$ satisfies property \ref{chaincoll:strong proper H-e} and is indeed strongly proper. To see this, first note that by Lemma \ref{lem:C4s in dil bip} and the fact that $A'\subseteq A$ is $(8k,3)$-fold solution-free, we have the following

\begin{enumerate}[\textbf{(*)}]
    \item \label{(*)} 
 Any copy of $C_4$ in $F$ which contains two edges in $\final{G^b}_H$, must be entirely contained in $\final{G^b}_H$.
\end{enumerate}
  Fixing $C_*\subseteq F'\subseteq F$ to be some copy of $C_4$ that contains $e_1$ and $e_2$, we show that all edges of $F$ are in $\final{G^b}_H$ by considering the following edges sequentially:
\begin{enumerate}[label=(\roman*)]
    \item \textbf{Edges of $C_*$:}  This follows directly from applying \ref{(*)} to $C_*$. 
    \item \textbf{Edges of $F'$ incident to $V(C_*)$:} If $f=uv\in E(F')$ with $v\in V(C_*)$ and $u\notin V(C_*)$, then let $v'\in V(C_*)\setminus\{v\}$ on the same side as $v$ in $F'$ and $u'\in V(C_*)$ on the other side. Then $v,v',u,u'$ span a $C_4$ containing $f$ and  such that both $vu'$ and $v'u'$ belong to $\final{G^b}_H$ as edges of $C_*$,  using the previous step. Therefore by \ref{(*)}, $f$ also belongs to $\final{G^b}_H$.
    \item \textbf{Remaining edges of $F'$:} If $g=xy\in E(F')$ with $x,y \notin V(C_*)$, then let $x'\in V(C_*)$ on the same side as $x$ in $F$ and $y'\in V(F')$ some arbitrary vertex on the same side as $y$. Then $x,x',y,y'$ span a $C_4$ containing $g$ and such that $x'y$ and $x'y'$ are both in $\final{G^b}_H$ by the previous step. An application of \ref{(*)} gives that $g$ also belongs to $\final{G^b}_H$. 
    \item \textbf{Edges of $F\setminus F'$:} These edges are all incident to the vertex $z$ of degree $r-1$ in the part of $F$ of size $s$. Choosing any two of these edges and some arbitrary further vertex in the part of size $s$ gives a $C_4$ in $F$ with two edges in $F'$, which belong to $\final{G^b}_H$ by the previous step, and so all edges incident to $z$ also belong to $\final{G^b}_H$ by \ref{(*)}. 
 \end{enumerate}
This shows that $F$ is entirely contained in $\final{G^b}_H$ and completes the proof that $\cH'$ satisfies property \ref{chaincoll:strong proper H-e} of Definition \ref{def:stronger-proper}, verifying that $\cH'$ is a strongly proper collection and completing the proof of the theorem.
\end{proof}

\begin{rem} \label{rem:comp bip better}
    With a similar proof, one can actually slightly improve the exponent  in Theorem \ref{thm:complete bip}, showing that in fact $M_{K_{r,s}}(n)=n^{2-\tfrac{r+s-1}{r(s-1)}-o(1)}$. Indeed, in the proof of Theorem \ref{thm:complete bip} above, we obtained $A_q'$ by deleting from $A_q$ some $b\in B_F$ for each `bad' copy $F$ of $H'$ with $|B_F|\geq r(s-1)$. This is somewhat wasteful. We could instead just delete a few edges  of  $\bigcup_{i\in [\tau]}\final{H^b_i}_H$, killing  $F$ whilst leaving most of the chain $(H_i^b,e_i^b)_{i \in [\tau]}$ intact. In more detail, we pick an edge $f_b$ of $F$ that lies in $\bigcup_{i\in [\tau]}\final{H^b_i}_H$. Let $i_*\in [\tau]$ be the maximum index $i\in [\tau]$ such that $f_b\in E(\final{H^b_i}_H)$. Then we can `split' the chain 
 $(H_i^b,e_i^b)_{i \in [\tau]}$ into at most two smaller
 chains $(H_i^b,e_i^b)_{i=1}^{i_*-2}$ and $(H_i^b,e_i^b)_{i =i_*+1}^\tau$ which completely avoid the vertices of $f_b$. This splitting procedure maintains all of the relevant properties of the collection being (strongly) proper and we can still apply a variant of Lemma \ref{lem:linking chains} (see Remark \ref{rem: chains diff size}) to link the chains of the new collection to get one long proper chain. As this linking procedure needs a constant number of new vertices for each link between two chains, we need that the number of chains in the collection is at most $O(n)$. Consequently, we need that the number of splittings we do, and so the number of bad $F$ with $|B_F|\geq r(s-1)$, is at most $O(n)$. By considering the expected number of such bad $F$ as in \eqref{eq:bad F expec}, we get that $n^{k-o(1)}q^{r(s-1)}$ should be comparable to $n$, and thus $q=n^{\tfrac{1-k}{r(s-1)}-o(1)}$ and  $A_q$ has expected size $n^{1+o(1)}q=n^{1-\tfrac{r+s-1}{r(s-1)}-o(1)}$. By taking such an $A_q$ that achieves this size (as well as having at most expected upper bound on the number of bad $F$) and splitting chains that contain an edge of each bad copy $F$ of $H'$, we obtain a strongly proper collection of chains (of potentially different sizes) whose combined length is at least $n^{2-\tfrac{r+s-1}{r(s-1)}-o(1)}$, as required. 
\end{rem}

\section{Line chains} \label{sec:line}

Our third and final construction of a chain  collection simply places chains on edges of a hypergraph. We say a hypergraph $\cL$ is \emph{linear} if for any pair of edges $\ell_1,\ell_2\in E(\cL)$, we have that $|\ell_1\cap \ell_2|\leq 1$. The terminology stems from the geometric fact that two lines can intersect in at most one point and indeed many interesting examples of linear hypergraphs are generated by finite geometries.  We proceed with the definition of our line chains. 

\vspace{2mm}

\begin{const}[Collection of line chains] \label{const:lin}
    Suppose that $4\leq k,L\in \NN$, $H$ is a  graph with $V(H)=[k]$ such that $e':=\{1,2\}$ and $f'=\{k-1,k\}$ are edges of $H$ and  $\cL$ is an $L$-uniform hypergraph with vertex set $V$. Fixing $\tau:=\floor{\tfrac{L}{4(k-2)}}$, a collection 
    \[\cH_{\mathrm{Lin}}(H;\cL)=\{(H^\ell_i,e_i^\ell)_{i\in [\tau]}:\ell\in E(\cL)\}\]
   of $H$-chains indexed by edges $\ell$ of $\cL$ is a collection of \emph{line chains} if for each  $\ell\in E(\cL)$ the chain $(H^\ell_i,e_i^\ell)_{i\in [\tau]}$ is  a simple $H$-chain on the vertices of $\ell$. 
\end{const}

\vspace{2mm}

We do not  specify exactly how to place the simple $H$-chains on each line of $\cL$ and in most cases this can be done arbitrarily, noting that there is plenty of space as a simple $H$-chain of length $\tau$ has $\tau(k-2)+2\leq L/4+2$ vertices. When $H$ is bipartite, we show that this placement can be done in such a way that the union of the line chains is also bipartite.

\begin{lem} \label{lem:bip line chains}
    Let $H$ be a $k$-vertex bipartite graph and suppose $\cL$ is an $N$-vertex linear hypergraph with edges of size $L\geq 200\log N$. Then there is a collection $\cH=\cH_{\mathrm{Lin}}(H;\cL)=\{(H^\ell_i,e_i^\ell)_{i\in [\tau]}:\ell\in E(\cL)\}$ of line chains such that $G:=\bigcup_{\ell\in E(\cL),i\in [\tau]}H_i^\ell$ is bipartite. 
\end{lem}

\begin{proof}
    Consider a bipartition of $V=V(\cL)$ as $V=V_1\cup V_2$ where each vertex $v\in V$ is added to $V_j$ with probability $1/2$ for $j\in[2]$, independently from the other vertices. For each $\ell\in E(\cL)$ and $j=1,2$, we have that $\EE[|\ell\cap V_j|]=L/2$. Moreover, by Chernoff's theorem (Theorem \ref{thm:chernoff}), we have that $\PP[|\ell \cap V_j|\leq 3L/8]\leq \exp(-L/64)\leq N^{-3}$ using our lower bound on $L$. Note also that $|E(\cL)|\leq \binom{N}{2}$ as each pair of vertices in $V$ lies in at most one edge of $\cL$. Therefore, taking a union bound over $\ell\in E(\cL)$ and $j=1,2$, we get that with positive probability there is some partition of $V$ as $V_1\cup V_2$, such that $|V_j\cap \ell|\geq 3L/8$  for all $j=1,2$ and $\ell\in E(\cL)$. Fixing such a partition and using that a simple $H$-chain of length $\tau$ has at most $L/4+2$ vertices, for each $\ell\in E(\cL)$ we can place a simple $H$-chain $(H^\ell_i,e_i^\ell)_{i\in [\tau]}$ on the vertices of $\ell$ such that $G^{\ell}:=\bigcup_{i\in [\tau]}H_i^\ell$ is bipartite with one part in $\ell\cap V_1$ and the other part in $\ell\cap V_2$. Consequently, $G=\bigcup_{\ell\in E(\cL)}G^\ell$ is also bipartite. 
\end{proof}

Our construction \ref{const:lin} will be applied with $\cL$ being a linear hypergraph with high \textit{girth}. For our purposes we define a hypergraph $\cC$ to be  a \textit{cycle} if it has at least two edges and  there are distinct vertices $v_1,\ldots,v_r\in V(\cC)$ and  an ordering of $E(\cC)$ as $\ell_1,\ldots,\ell_r$ such that $v_i\in \ell_i\cap \ell_{i+1}$ for $i=1,\ldots,r$ (with $\ell_{r+1}=\ell_1$ here). This is sometimes referred to as a \emph{Berge} cycle in the literature. The \textit{length} of the cycle is the number  of edges $r$ in $E(\cC)$ and 
the \emph{girth} of a hypergraph $\cL$ is the length of the smallest cycle that is a subgraph of  $\cL$. Note that two edges of a hypergraph which intersect in at least two vertices form a cycle of length 2, so any hypergraph with girth at least 3 is necessarily linear. 
We use the following construction of Lazebnik, Ustimenko and Woldar \cite{lazebnik1995new}.  

\begin{thm} \label{thm:high girth hyp}
  For $3\leq g\in \NN$ and $N$ sufficiently large, there is an $L$-uniform hypergraph with $N/2^{2g}\leq |V(\cL)|,|E(\cL)|\leq N$,  girth at least $g$ and $L\geq \tfrac{1}{2}N^{\tfrac{2}{3g-5}}$. 
\end{thm}

The actual construction of Lazebnik, Ustimenko and Woldar \cite[Theorem 3.2]{lazebnik_connectivity_2004} is in the setting of bipartite graphs. They show that if $q$ is a prime power, $s\in \NN$ is odd and $t:=\floor{\tfrac{s+2}{4}}$, then there is a $q$-regular bipartite graph $G(q,s)$ with no cycle of length at most $s+4$, and at most $2q^{s-t+1}$ vertices. In following work, it was proven that $G(q,s)$ has exactly $2q^{s-t+1}$ vertices, first when $q$ is odd \cite{lazebnik1996characterization} and then when $q$ is even \cite{lazebnik_connectivity_2004}.  To derive Theorem \ref{thm:high girth hyp} from this construction, given $g,N\in \NN$, we fix $s=2g-5$, $t=\floor{\tfrac{2g-3}{4}}$ and let $q$ be some prime power such that $\tfrac{1}{2}N^{\tfrac{1}{s-t+1}}\leq q\leq N^{\tfrac{1}{s-t+1}}$. Taking the bipartite graph $G(q,s)$ from \cite[Theorem 3.2]{lazebnik_connectivity_2004} with parts $X,Y$, we define a hypergraph $\cL$ with vertex set $Y$ and edges $\ell_x:=N_{G(q,s)}(x)$ for each $x\in X$. We therefore have that $\cL$ has $q^{s-t+1}$ vertices and edges and so the bounds in Theorem \ref{thm:high girth hyp} for $|V(\cL)|$ and $|E(\cL)|$ are indeed satisfied.   
 Moreover, the edges of $\cL$ have size $q\geq \tfrac{1}{2}N^{\tfrac{2}{3g-5}}$, using that $s-t+1\leq (3g-5)/2$.  Finally, note that a cycle of length $r< g$ in $\cL$ gives rise to a cycle of length $2r< 2g=s+5$ in $G(q,s)$, and so we indeed have that the girth of $\cL$ is at least $g$.

With the hypergraph from Theorem \ref{thm:high girth hyp}, we are ready to prove the main theorem of this section.

\begin{thm} \label{thm:insep sup lin}
   Any $(2,1)$-inseparable or $(1,1,1)$-inseparable graph $H$ with $v(H)=k$ satisfies
   \[M_H(n)\geq \Omega\left(n^{1+\tfrac{2}{3k-2}}\right).\]
\end{thm}

In particular, note that Theorem \ref{thm:insep sup lin} encompasses Theorem \ref{thm:connectivity} which was stated just for $(2,1)$-inseparable graphs. 

\begin{proof}[Proof of Theorem \ref{thm:insep sup lin}. ]
 For $n$ sufficiently large, fix $N:=\floor{n/2k^2}$, $g:=k+1$ and let $\cL$ be the hypergraph generated by Theorem \ref{thm:high girth hyp}. We then take 
 \[\cH=\cH_{\mathrm{Lin}}(H;\cL)=\{(H^\ell_i,e_i^\ell)_{i\in [\tau]}:\ell\in E(\cL)\}\]
 to be a collection of line chains as in Construction \ref{const:lin}, with any ordering of $V(H)=[k]$ such that both $e'=\{1,2\}$ and $f'=\{k-1,k\}$ are edges noting that such an ordering exists because $H$ is inseparable (Observation \ref{obs: disjoint edges}).  In the case that $H$ is bipartite, we place the simple $H$-chains in $\cH$ in such a way that $G:=\bigcup_{\ell\in E(\cL),i\in [\tau]}H_i^\ell$ is bipartite, using Lemma \ref{lem:bip line chains}.  Each chain is proper by Lemma \ref{lem:inseparableimpliesproper}.
 
 We appeal to Lemma \ref{lem:linking chains}, which states that if $\cH$ is a proper collection of chains, then there is a proper chain of length at least \[\tau |E(\cL)|\geq \frac{LN}{4k}\geq \Omega\left(n^{1+\tfrac{2}{3g-5}}\right)=\Omega\left(n^{1+\tfrac{2}{3k-2}}\right),\]
on at most $n$ vertices which in turn gives the desired lower bound on $M_H(n)$ through an application of Lemma \ref{lem:chain_runningtime}. It remains then to prove that the collection $\cH$ is proper. For this, we again use Lemma \ref{lem:strong implies proper} and in fact show that $\cH$ is strongly proper. 
Letting $G^\ell:=\bigcup_{i\in [\tau]}H^\ell_i$ for $\ell\in E(\cL)$, we have that each edge of $\final{G^\ell}_H$ is contained in $\ell$. As $\cL$ is linear, this implies that the graphs $\{\final{G^\ell}_H:\ell\in E(\cL)\}$ are pairwise edge-disjoint and condition \ref{chaincol:strong proper edge} of being strongly proper is satisfied. 

    It remains to verify condition \ref{chaincoll:strong proper H-e}, so we fix some $e\in E(H)$, a copy $F$ of $H-e$ in $\bigcup_{\ell\in E(\cL)}\final{G^\ell}_H$ and suppose for a contradiction that there are $\ell'\neq \ell''\in E(\cL) $ and edges $f',f''\in E(F)$ with $f'\in E(\final{G^{\ell'}}_H)$ and $f''\in E(\final{G^{\ell''}}_H)$. By Lemma \ref{lem:insep edges in cycle}, there is some cycle $C\subseteq F$ such that $f',f''\in E(C)$. Let $x_1,\ldots, x_{k'}$ be the vertices of $C$, noting that $k'\leq k$ as $C\subseteq F$ and $F$ is a copy of $H-e$.  For each $i\in [k']$, we have that $x_ix_{i+1}$ is an edge in $\bigcup_{\ell\in E(\cL)}\final{G^\ell}_H$ (taking $x_{k'+1}=x_1$ when $i=k'$) and so there is some $\ell_i\in E(\cL)$ such that $x_ix_{i+1}\in\final{G^{\ell_i}}_H$ and so $\{x_i,x_{i+1}\}\subseteq \ell_i$. 
    The edges $\ell_1,\ldots,\ell_{k'}$, which contain both $\ell'$ and $\ell''$, do not necessarily form a cycle, as there may be repeats among the edges. However it is true that there is a cycle of length at most $k'$ as a subgraph of the hypergraph formed by the edges $\{\ell_1,\ldots,\ell_{k'}\}$. Indeed, consider the closed walk $x_1\ell_1x_2\cdots x_{k'}\ell_{k'}$ in the bipartite incidence graph between $V(\cL)$ and $E(\cL)$. If there is some $x_i$ such that $\ell_{i-1}=\ell_i$ (taking $\ell_{i-1}=\ell_{k'}$ if $i=1$) then remove $x_i\ell_i$ from the walk. This results in a circuit as all of the $x_i$ are distinct and so the only possible repeated edges in the walk  $x_1\ell_1x_2\cdots x_{k'}\ell_{k'}$ are of the form  $\ell_{i-1}x_i$ and $x_i\ell_i$. The resulting circuit then contains a cycle in the bipartite incidence graph, corresponding to a cycle in the hypergraph $\cL$ of length at most $k'\leq k$. This contradicts that $\cL$ has girth at least $k+1$ and verifies condition \ref{chaincoll:strong proper H-e} of being strongly proper, completing the proof. 
\end{proof} 

\section{Upper bounds on running times} \label{sec:upper}
Here we prove the upper bounds on running times, where one has to show that the process terminates in a given number of steps, no matter what the $n$-vertex starting graph $G_0$ is.  The main content of this section is the proof of Theorem \ref{thm:wheel} giving an upper bound for the running time $M_H(n)$ when $H$ is an odd wheel graph. 

\subsection{Sparse random graphs} \label{sec:upper sparse random} We start  with a quick and simple proof bounding the running time when $H$ is a sparse random graph, establishing Theorem \ref{thm:random} \eqref{random 0}. In fact, we show something slightly stronger, giving an explicit bound of just 3 for large enough $n$. 

\begin{prop} \label{prop:random}
    The following holds a.a.s.\ for $H=G(k,p)$ when $p=o(\log k/k)$. If $n\geq 2k$, then $M_H(n)\leq 3$. 
\end{prop}
\begin{proof}
    By Fact \ref{fact:G(k,p) isolated edge}, we have that a.a.s.\ $G(k,p)$ is either empty or has an isolated edge. If the former holds, then  we certainly have $M_H(n)=0$. If the latter holds then fixing $H$ to be such an instance of $G(k,p)$ and $G_0$ to be an arbitrary $n$-vertex graph, we have that if $G_1\neq G_0$, then there is some copy $H_1$ of $H$ in $G_1$. Letting $e_1$ be an isolated edge in $H_1$ and $W=(V(G_0)\setminus V(H))\cup e_1$, we have that $G_2[W]$ is complete as any edge in $\binom{W}{2}$ can replace $e_1$ in a copy of $H$. Finally $G_3$ is then complete as any remaining edge can complete a copy of $H$ with some selection of $k-2$ vertices of $W$. 
\end{proof}

\subsection{The wheel graph} \label{sec: upper wheel}

Before proving the upper bound of Theorem \ref{thm:wheel}, we need some preliminary lemmas. Firstly, we need  the famous  Ruzsa-Szemer\'edi $(6,3)$-problem/theorem. Answering a question of Brown, Erd\H{o}s and S\'os \cite{brown1973some}, Ruzsa and Szemer\'edi \cite{ruzsa1978triple} showed that any $n$-vertex $3$-uniform hypergraph with no $6$ vertices inducing 3 or more edges, must have $o(n^2)$ edges.   The solution of Ruzsa and Szemer\'edi \cite{ruzsa1978triple} was one of the earliest applications of Szemer\'edi's  regularity lemma. The theorem has several equivalent formulations and applications. In particular,  it can be used to give a proof of Roth's theorem showing that sets in $[n]$ that are free of three-term arithmetic progressions, have size $o(n)$. We will use the following equivalent formulation (see for example \cite[Theorem 3.2]{komlos1995szemeredi}), where we say a matching $M$ in a graph $G$ is an \emph{induced matching} if $E(G[V(M)])=E(M)$.  

\begin{thm}[Ruzsa-Szemer\'erdi $(6,3)$-theorem] \label{thm (6,3)}
    For all $\eps'>0$, there exists $n_0\in \NN$ such that if $n\geq n_0$ and $G$ is an $n$-vertex graph which is the union of $n$ induced matchings, then $e(G)\leq \eps' n^2$. 
\end{thm}

Secondly, we need the following simple result on $C_k$-processes which we proved in a previous paper  \cite[Lemma 3.5]{FMSz1}.

 \begin{lem}\label{lem:cycles_smalldistance}
    Let $k\geq 5$ be odd and let  $G$ be  
     a connected graph of order at least $k+1$ which contains a copy of $C_k$. 
  Then $\final{ G}_{C_k}$ is a clique.
    \end{lem}

We remark that Lemma \ref{lem:cycles_smalldistance} does not hold for even cycles as the final graph may be complete bipartite. It is for this reason that we focus here on wheel graphs $W_k$ with $k$ odd. 
Finally we give a simple lemma showing how cliques grow during the $W_k$-process. 

\begin{lem} \label{lem:wheel cliques grow}
    Let $k\geq 7$ be an odd integer and $H=W_k$ the wheel graph  with $k+1$ vertices. Suppose that $(G_i)_{i\geq 0}$ is some $H$-process and  $Y\subseteq V(G_0)$ with $|Y|\geq k+1$ and $G_0[Y]$ being a clique. If $F$ is some copy of $H$ in $G_0$ with $2\leq |V(F)\cap Y|< k+1$, then there is a vertex $z\in V(F)\setminus (Y\cap V(F))$ such that 
    $G_{1}[Y\cup \{z\}]$ is  a clique. Moreover, $G_{k}[Y\cup V(F)]$ is also a clique.
\end{lem}
\begin{proof} 
 We claim that there is some vertex $z\in V(F)\setminus Y$ such that $|N_{G_{0}}(z, Y)|\geq 2$. Indeed, if the universal vertex $u$ of $F$ lies outside of $Y$, then we can take $z=u$ and use that $|V(F)\cap Y|\geq 2$. If $u$  lies in $Y$, then the cycle $C=F-u$ has vertices both inside and outside of $Y$. Therefore, there is some vertex $z\in V(C)\setminus(V(C)\cap Y)$ such that $|N_C(z,Y)|\geq 1$ and as $z$ also has $u\in Y$ as a neighbour, we indeed have that $|N_{G_{0}}(z,Y)|\geq 2$. Now as the minimum degree of $H$ is 3, for any $y\in Y$, such that $zy\notin E(G_{0})$, we have that $zy$ completes some copy of $H$ with $z$ and vertices of $Y$. Hence $Y\cup \{z\}$ is complete at time $1$, completing the proof.

 The `moreover' statement comes from repeated applications of the main statement of the lemma,  growing the size of $Y$ at each time step to incorporate another vertex in $V(F)$. This then gives that all of $V(F)$ is in a clique with $Y$ after at most $k+1-|V(F)\cap Y|\leq k$ steps.
\end{proof}

We are now in a position to prove the upper bound of Theorem \ref{thm:wheel}. 

\begin{proof}[Proof of upper bound in Theorem \ref{thm:wheel}]
    Let $k\geq 7$ be an odd integer, $H=W_k$ the wheel graph with $k+1$ vertices and $\eps>0$ be arbitrary. We will show that for $n$ sufficiently large, any $n$-vertex graph $G_0$ has $\tau_{H}(G_0)\leq \eps n^2$.  We fix $K=\binom{k}{2}$, $0<\eps':=1/8k!$ and suppose that $n$ is large enough  such that any $n$-vertex graph $G$ which is the union of $Kn$ induced matchings  has $e(G)\leq \eps'n^2$, using Theorem \ref{thm (6,3)}. We fix some $n$-vertex graph $G_0$, suppose for a contradiction that $\tau:=\tau_H(G_0)\geq \eps n^2$ and let $(G_i)_{i=0}^\tau$ be the $H$-process on $G_0$.

    Let $(H_i,e_i)_{i\in [\tau]}$ be some $H$-chain such that $e_i\in E(G_i)\setminus E(G_{i-1})$ for $i\in [\tau]$. As discussed at the beginning of Section \ref{sec:chains}, such a chain exists and can be constructed in reverse order, starting with some edge $e_\tau$ added at time $\tau$ and a copy $H_\tau$ of $H$ completed by $e_\tau$. The edge $e_{\tau-1}$ is then chosen as an edge of $H_\tau$ that is added at time $\tau-1$, which must exist as otherwise $e_\tau$ would be added before time $\tau$. Defining $H_{\tau-1}$ to be a copy of $H$ completed by $e_{\tau-1}$ and continuing this process, one can define the full chain $(H_i,e_i)_{i\in [\tau]}$. 

We further fix $\tau_1:=\ceil{\eps n^2/2}$ and give the following claim. 

\begin{clm} \label{clm:growing cliques}
  There is some $1\leq t_0 \leq \tau_1$ such that for  $s=k+1,\ldots n$,  there is a clique in $G_{\tau_1+2s}$ of size $s$ containing $V(H_{t_0})$. 
\end{clm}
    Note that it suffices to prove this claim, as then $G_{\tau_2}$ is a clique with $\tau_2:=\tau_1+2n<\tau$, contradicting that the process stabilises only at time $\tau$.  We prove the claim by induction on $s$ and so assume that for some $r\in \{k+2,\ldots,n\}$, the claim holds for $k+1\leq s<r$. Therefore there is some $1\leq t_0\leq \tau_1$ and a set of vertices $W\supseteq V(H_{t_0})$ of size $r-1$ such that $G_{\tau_1+2(r-1)}[W]$ is a clique. Let \[t':=\min\left\{t:t_0+1\leq t \leq \tau, V(H_t)\nsubseteq W\right\}.\]
    Note that $t'$ is well-defined and in fact $t'\leq \tau_1+2r-1$  as $W$ is already a clique at time $\tau_1+2r-2$ and so $V(H_{\tau_1+2r-1})$ cannot be contained in $W$. We  have that $V(H_{t'-1})\subseteq W$ by the definition of $t'$, using here that $V(H_{t_0})\subseteq W$ if $t'=t_0+1$. Therefore as $e_{t'-1}\in V(H_{t'-1})\cap V(H_{t'})$, we have that $|V(H_{t'})\cap W|\geq 2$, and Lemma \ref{lem:wheel cliques grow} gives that there is some vertex $z\notin W$ such that $W\cup\{z\}$ is a clique at time
     $t'+1\leq \tau_1+2r$, proving the induction step. 

    It remains to prove the base case $s=k+1$. Let $\tau_0=\tau_1-\binom{k+1}{2}$ and for $t\in[\tau_0]$, let $u_t\in V(H_t)$ be the universal vertex of $H_t$ and $Q_t:=H_t-u_t$ be the cycle obtained by removing the universal vertex. Furthermore,  we define $\tilde{G}:=\cup_{t=1}^{\tau_0}Q_t$, $G':=\cup_{t=1}^{\tau_0} H_t$ and for $v\in V(G_0)$, we let 
    \[
    \tilde{G}_v:=\bigcup\{Q_t:1\leq t\leq \tau_0, u_t=v\} \qquad \mbox{ and } \qquad N_v:=N_{G'}(v).
    \]
    Now we claim that we are done if for some $v\in V(G_0)$ we have that $G'[N_v]$ contains a connected component of size at least $k+1$ that contains a copy  of the cycle $C_k$. Indeed, suppose that $X\subseteq N_v$ is a vertex set with $|X|=k+1$ such that $G'[X]$ is connected and contains a copy of $C_k$. Then as all vertices in $N_v$ are adjacent to $v$ in $G_{\tau_0}$, any edge added in the $C_k$-process on $G'[X]$ is added in the $H$-process also. Hence by Lemma \ref{lem:cycles_smalldistance}, $G_{\tau'}[X]$ is  a clique of size $k+1$ for some $\tau'\leq \tau_0+\binom{k}{2}$. We are not quite done as there is no guarantee that $X$ is the vertex set of some $H_{t_0}$ with $1\leq t_0\leq \tau_1$. However there is some  $t_0\in [\tau_0]$ such that $G'[X]$ contains an edge of $V(H_{t_0})$ (as $G'$ is simply the union of the $H_t$ with $1\leq t\leq \tau_0$). Hence  $|V(H_{t_0})\cap X|\geq 2$ and by Lemma \ref{lem:wheel cliques grow}, we have that $V(H_{t_0})$ is a clique of size $k+1$ at time $\tau'+k\leq \tau_1$.

We can therefore assume that for all $v\in V(G_0)$,  
\begin{itemize}[($\dagger$)]
    \item  any connected component in $G'[N_v]$ that contains a copy of $C_k$ has exactly $k$ vertices. 
\end{itemize} 
Moreover, as $\tilde{G}\subseteq G'$, we also have that 
\begin{itemize}[($\tilde \dagger$)]
    \item  any connected component in $\tilde G[N_v]$ that contains a copy of $C_k$ has exactly $k$ vertices. 
\end{itemize} 
This implies that for each $v\in V(G_0)$, the graph $\tilde{G}_v$ is a union of at most $K=\binom{k}{2}$ matchings that are induced in $\tilde{G}$. Indeed, as $\tilde{G}_v$ is simply a union of copies of $C_k$, we have that each non-empty connected component  of $\tilde{G}_v$ contains a copy of $C_k$ and is contained in $\tilde{G}[N_v]$. Hence all non-empty connected components of $\tilde{G}_v$ have exactly $k$ vertices and we can use $K$ colours to colour the edges of $\tilde{G}_v$ such that no two edges in the same connected component receive the same colour. Each colour class is then a matching and it is induced in $\tilde{G}$ as if it was not, then there would be a component of $\tilde{G}[N_v]$ that violates ($
\tilde \dagger$). As $\tilde{G}=\bigcup_{v\in V(G_0)}\tilde{G}_v$, we have that $\tilde{G}$ is the union of at most $Kn$ induced matchings and thus, by Theorem \ref{thm (6,3)}, we have that $\tilde{G}$ has at most $\eps' n^2$ edges. 

Finally then, we will lower bound the number of edges of $\tilde{G}$ to obtain a contradiction. Considering the edges $e_t\in E(H_t)\cap E(H_{t+1})$ for $1\leq t\leq \tau_0-1$, which are all distinct, if $(\tau_0-1)/2>\eps' n^2$ of these lie in $\tilde{G}$, then we are done. Therefore, we can assume that that there are at least $(\tau_0-1)/2$ steps $T\subseteq [\tau_0-1]$ such that for each $t\in T$, we have that $e_t\notin E(Q_t)\cup E(Q_{t+1})$. We claim that for all $t\in T$, we have that $e_t=u_tu_{t+1}$. Indeed, we must have that $e_t$ is incident to both $u_t$ and $u_{t+1}$ as it is not in $Q_t$ nor $Q_{t+1}$. Moreover, if $u_t=u_{t+1}$ then $Q_t$ and $Q_{t+1}$ must intersect (as $e_t\setminus \{u_t\}\subseteq  V(Q_t)\cap V(Q_{t+1})$) and lie on distinct vertex sets (as $e_{t+1}\setminus \{u_t\}\nsubseteq V(Q_t)$) and thus the component of $\tilde G$ in $\tilde{G}[N_{u_t}]$ containing $Q_t\cup Q_{t+1}$ violates condition $(\tilde\dagger)$. 

So indeed, for all $t\in T$ we have that $e_t=u_tu_{t+1}$ and we fix $f_t=x_tu_{t+1}\in E(Q_t)$, choosing $x_t$ to be one of the two neighbours of $u_{t+1}$ on $Q_t$. Note that we must have that  $x_t\in V(Q_{t+1})$ as otherwise $V(Q_{t+1})\cup \{x_t\}$ is a connected graph in $G'[N_{u_{t+1}}]$ violating ($\dagger$), using here that $x_tu_t\in E(H_t)\subseteq E(G')$ and $u_t\in V(Q_{t+1})$ as $e_t\in E(H_{t+1})$. 

Now we claim that the set $F:=\{f_t:t\in T\}\subseteq \tilde G$ has size $|F|\geq (\tau_0-1)/2k!>\eps'n^2$ which gives our desired contradiction. To see this,  we show that each edge $g\in E(\tilde{G})$ appears at most $k!$ times as $f_t$ for some $t\in T$. 
 Fixing some vertex $v\in g$, let $t_0\in T$ 
 be the first time that $e_t=g$ and  $u_{t+1}=v$ (if such a time exists). Then every other time $t_0<t\in T$ such that $g=f_t$ with $u_{t+1}=v$, we must have that $V(Q_{t+1})=V(Q_{t_0+1})$  
as otherwise $Q_{t+1}\cup Q_{t_0+1}\subseteq \tilde{G}[N_{v}]$ would violate ($\tilde \dagger$). 
As there are at most $k!/2$ ways to map a cycle to $k$ vertices, we indeed get that $g$ can appear as $f_t$ at most $k!$ times where the additional factor of two  comes from considering the other vertex of $g$ playing the role of $u_{t+1}$.
\end{proof}

\section{Concluding remarks} \label{sec:conclude}

We close the paper with some remarks about future directions. 

\subsection{Random graphs around the threshold}
In this paper we proved that the binomial random graph $H=G(k,p)$ has a phase transition around $\log k/k$  with $M_H(n)=O(1)$ a.a.s.\ when $p=o(\log k/k)$ and $M_H(n)=\Omega(n^2)$ a.a.s.\ when $p=\omega(\log k/k)$, see Theorem \ref{thm:random}. It is natural to ask whether this threshold is sharper than presented here. The $0$-statement, asserting that $M_H(n)$ is constant when $p$ is small, can in fact be tightened easily. Indeed, the only property of $H=G(k,p)$ that we used for the proof was the existence of an isolated edge, which a.a.s.\ occurs when $p\leq c_0\log k/k$, for any $c_0<1/2$, see for example \cite[Theorem 5.4]{bollobas_random_2001}. The $1$-statement, giving quadratic running time for $p$ much larger than $\log k/k$, is more delicate and in particular, we use that $H=G(k,p)$ is self-stable, applying a result of Kim, Sudakov and Vu \cite{kim2002asymmetry}. Whilst it is believable that results on the asymmetry of $G(k,p)$ can be strengthened to give that $H=G(k,p)$ is self-stable as soon as $p> \log k/k$, we did not find such a result in the literature and it is not immediate how one could prove such a statement. Nonetheless, we believe it is likely that our proof can be pushed to give  $M_H(n)=\Omega(n^2)$ a.a.s.\ already when $p>c_1\log k/k$ for any $c_1>1$. The range $\log k/2k<p<\log k/k$ seems to be the most interesting and it is unclear what to expect for the behaviour of $M_H(n)$ with $H=G(k,p)$. In this range, $H$ will a.a.s.\  have one giant component and some isolated vertices, which can be ignored from the perspective of running times. The giant component however will a.a.s.\ be \textit{separable}  in this range, even having vertices of degree one. This suggests that chain constructions will be useless for proving lower bounds on $M_H(n)$     but it does not preclude the possibility of having a slow maximum running time. Indeed, in \cite{FMSz2} we showed that there are graphs $H$ with $\delta(H)=1$ and $M_H(n)=\Omega(n^2)$. It would also be interesting to determine $M_H(n)$ for $H$ being a random 3-regular graph, which we expect to be (almost) quadratic.


\subsection{The minimum degree threshold for quadratic running time}
Another question that remains open from our work is to fully understand the behaviour of $M_H(n)$ under minimum degree conditions on $H$. We proved in Theorem \ref{thm;dense quad} that $M_H(n)$ is of the form $n^{2-o(1)}$ when $H$ is a $k$-vertex graph with $\delta(H)\geq k/2+1$. This is tight as there are examples ($K_{k/2,k/2}$ and the $H'_{k/2+1}$ in Example \ref{ex:H'k}) which have $\delta(H)=k/2$ and running time $M_H(n)$ polynomially separated from $n^2$ (or even linear in the case of $H'_{k/2+1}$). We also showed that in fact $M_H(n)=\Omega(n^2)$ when $\delta(H)>3k/4$. It would be interesting to determine the exact minimum degree threshold for quadratic running time and in particular, whether it is possible to have dense graphs $H$ with $\delta(H)\geq k/2+1$ that exhibit a running time like the wheel graph $W_k$, having  $M_H(n)=o(n^{2})$.

\subsection{Bipartite graphs}
The results of this paper show that bipartite graphs $H$ have running time $M_H(n)$ polynomially separated from $n^2$. Indeed, in Proposition \ref{prop:bipartite_extremalnumber}, we showed that $M_H(n)=O(\ex(n,H))$ where $\ex(n,H)$ is the extremal number of $H$. We expect that this is never tight. 

\begin{conj} \label{conj:bip upper}
   For all bipartite graphs $H$, we have that 
   \[M_H(n)=o(\ex(n,H)).\]
\end{conj}
It could well be that a stronger conjecture is true and there is always polynomial separation between $M_H(n)$ and $\ex(n,H)$ for all bipartite graphs $H$.

We also gave lower bounds for many different bipartite graphs $H$, in particular showing in Corollary \ref{cor:compl bip lower} that for $r$ fixed and $s$ large, the maximum running time $M_{K_{r,s}}(n)$ approaches the upper bound given by the extremal number (Corollary \ref{cor:extremal_bipartite}). For smaller bipartite graphs, we expect that the bounds given by Proposition \ref{prop:bipartite_extremalnumber} are far from  tight. In particular, for $K_{3,3}$ and the cube graph $Q_3$, it would be very interesting to improve on any of  the bounds 
\[n^{3/2-o(1)}\leq M_{K_{3,3}}(n)\leq O(n^{5/3}) \quad \mbox{ and } \quad \Omega(n^{3/2})\leq M_{Q_3}(n)\leq O(n^{8/5}), \]
and we tend to believe that the lower bounds are closer to the truth. 





\subsection{Further applications of chain constructions}
We expect that our framework for general chain constructions will have many further applications giving lower bounds for $M_H(n)$. We have already demonstrated the flexibility of the \textit{ladder chains} (Construction \ref{const:ladder}) and \textit{dilation chains} (Constructions \ref{const:dil}
 and \ref{const: bip dil}) in giving general results for different classes of graphs $H$. For \textit{line chains}, we explored applications using different linear hypergraphs $\cL$ but did not include these in the paper. In fact, one can achieve general results using these constructions but they are weaker than the constructions using additive constructions and dilation chains. Indeed, for example, taking a collection of line chains defined by a generalized quadrangle can give a lower bound of $M_H(n)=\Omega(n^{5/3})$ for  graphs $H$ such that $H-e$ is Behrendian for any $e\in E(H)$. In general it seems that the constructions coming from line chains are far from tight but allow to give very general results, for example giving super-linear bounds for all $(2,1)$-inseparable graphs $H$ as we did here in Theorem \ref{thm:insep sup lin}. We expect that they will also be useful for giving non-trivial general lower bounds on $M_H(n)$ for higher uniformity hypergraphs $H$, a direction which we did not explore at all here but has already been studied for cliques \cite{espuny_diaz_long_2022,hartarsky_maximal_2022,noel_running_2022}.

 There are also other previous chain constructions that fit into our general framework and may find further applications in the future. Indeed, Bollob\'as, Przykucki, Riordan and Sahasrabudhe \cite{bollobas2017maximum} used a randomised construction to give lower bounds on $M_{K_k}(n)$ approaching quadratic as $k$ grows. We expect that similar constructions could also be used to achieve the bipartite analogue (Corollary \ref{cor:compl bip lower}) showing that $M_{K_{r,s}}(n)$ approaches $n^{1-1/r}$ as $s$ grows but our approach of using dilation chains and alterations gives a less technical route. Finally, the proof of Balogh, Kronenberg, Pokrovskiy and Szab\'o \cite{balogh2019maximum} showing that $M_{K_5}(n)\geq n^{2-o(1)}$ uses a chain construction that is similar to the ladder constructions given here, except that the vertex sets of the copies of $H=K_5$ (as well as the vertices of the construction $G_0$) are split into three parts rather than two. Exploring this approach further could lead to interesting new classes of chain construction. 

\bibliography{Biblio}

\appendix

\section{Additive constructions} \label{appendix:add}

In this appendix, we prove Theorems \ref{thm:behrend extended} and \ref{thm: sidon extended} giving large sets avoiding non-trivial solutions to several equations. In order to do so, we  adapt classical constructions of Behrend \cite{behrend1946sets} and Ruzsa \cite{ruzsa_solving_1993}.  We start with some notation. For $h,k\in \NN$, let $\cE^{h}$ denote  the set of equations $E=E(x_1,\ldots,x_h)=\sum_{i=1}^h\alpha_ix_i$ in at most $h$ variables (we allow some of the $\alpha_i$ to be zero). 
The set $\cE^h_k\subseteq \cE^h$ is then the equations $E=\sum_{i=1}^h\alpha_ix_i$ with $|\alpha_i|\leq k$ for all $i\in [h]$. 
Finally, $\cE_k^h(0)\subseteq \cE_k^h$ are the equations  $E=\sum_{i=1}^h\alpha_ix_i\in \cE_k^h$ such that $\sum_{i=1}^h\alpha_i=0$. For some collection $\cE$ of equations, $p$ a prime and $n\in \NN$,  $r_{\cE}(\ZZ_p)$ denotes the size of the largest subset of $\ZZ_p$ that has no non-trivial solutions modulo $p$ to equations in $\cE$, whilst  $r_\cE(n)$ is the size of the largest subset of $[n]$ which is free from non-trivial solutions (in the integers) to all equations in $\cE$. When $\cE$ is a single equation, that is, $\cE=\{E\}$, we drop the set brackets and simply write $r_E(\ZZ_p)$ and $r_E(n)$.

\subsection{Preparatory lemmas}
We now proceed with some simple lemmas.

\begin{lem} \label{app lem:mod or not}
Suppose that $h,k\in \NN$ and $\cE\subseteq \cE_k^h$ is some collection of equations. Then for any  prime  $p\in \NN$, fixing  $n:=\floor{(p-1)/hk}$, we have that \[r_{\cE}(\ZZ_p)\geq r_{\cE}(n).\]
\end{lem} 
\begin{proof}
    Taking $A\subseteq [n]$ to be a subset free of non-trivial solutions to equations in $\cE$, that achieves the size $r_\cE(n)$, we have that 
    $|E(a_1,\ldots,a_h)|\in \{1,\ldots,p-1\}$ for any $E(x_1,\ldots,x_h)\in \cE$ and any $(a_1,\ldots,a_{h})\in A^h$. Therefore $E(a_1,\ldots,a_h)\neq 0 \mbox{ mod }p$ and considering $A$ as a subset of $\ZZ_p$ in the natural way, we have that $A$ is free from non-trivial solutions modulo $p$ to equations in $\cE$. 
\end{proof}

\begin{lem} \label{app lem:dense avoiding sum not zero}
    For any $k,h\in \NN$ and sufficiently large prime $p\in \NN$,  taking $\cE=\cE^h_k\setminus \cE^h_k(0)$, we have that $r_{\cE}(\ZZ_p) \geq p/8k^2h^2$. 
 \end{lem}

 \begin{proof}
   By Lemma \ref{app lem:mod or not}, it suffices to show that there is some set $B\subseteq [n]$ with $n=\floor{(p-1)/hk}\geq p/2hk$ such that  $|B|\geq n/4hk$ and $B$ has no non-trivial solutions (in the integers) to equations in $\cE^h_k\setminus \cE^h_k(0)$. We take $B'\subseteq [n]$ to be all integers $b\in [n]$ such that $b=1 \mbox{ mod } (hk+1)$. Then certainly $|B|\geq n/4hk$ and for any $(b_1,\ldots,b_h)\in B^h$ and any $E(x_1,\ldots,x_h)=\sum_{i=1}^h\alpha_ix_i\in \cE^h_k\setminus \cE^h_k(0)$, we have that modulo $(hk+1)$,
   \[E(b_1,\ldots,b_h)= \sum_{i=1}^h\alpha_i\neq 0,\]
   using that $E\notin \cE^h_k(0)$. Therefore also in the integers $E(b_1,\ldots,b_h)\neq 0$ and we are done. 
 \end{proof}

 \begin{lem} \label{app lem:dilate to mix sets}
     Suppose that $\ell,h,k\in \NN$ and $\cE_1,\ldots,\cE_\ell\subseteq \cE_{k}^h(0)$ are collections of equations. Suppose  further that $p\in \NN$ is a sufficiently large prime and that there are $\delta_i\in (0,1)$ for $i\in [\ell]$ such that $r_{\cE_i}(\ZZ_p)\geq \delta_ip$ for all $i\in [\ell]$. Then for any $B\subseteq \ZZ_p$, there is some $B'\subseteq B$ such that $|B'|\geq \prod_{i=1}^\ell \delta_i |B|$ and $B'$ has no non-trivial solutions to any equation in $\bigcup_{i=1}^\ell\cE_i.$
 \end{lem}
 \begin{proof}
     For $i\in [\ell]$, let $A_i\subseteq \ZZ_p$ be a set free of non-trivial solutions to equations in $\cE_i$ and achieving the maximum size $r_{\cE_i}(\ZZ_p)$. Note that for any $t\in \ZZ_p$, the subset $A_i+t:=\{a+t:a\in A_i\}$ is also free of non-trivial solutions to equations in $\cE_i$. Indeed, if $(a_1,\ldots,a_h)\in (A_i+t)^h$ is a non-trivial solution to some equation $E(x_1,\ldots,x_h)=\sum_{i=1}^h\alpha_i x_i\in \cE_i$, then \[E(a_1-t,\ldots,a_h-t)=E(a_1,\ldots,a_h)-\left(\sum_{i=1}^h\alpha_i\right)t=0,\]
     using that $E\in \cE_i\subseteq \cE^h_k(0)$ and so $\sum_{i=1}^h \alpha_i=0$. 
     Therefore $(a_1-t,\ldots,a_h-t)\in A_i^h$ and is a non-trivial solution to $E$, which is a contradiction. 

We now consider $t_1,\ldots,t_\ell\in \ZZ_p$ to be independent uniform samples from $\ZZ_p$ and take the resulting random set
     \[A_*= \left(\bigcap_{i=1}^\ell (A_i+t_i)\right) \cap B.\]
     As $A_*$ is a subset of $B$ and $A_i+t_i$ for each $i\in [\ell]$, we have that $A_*$ has no non-trivial solutions to equations in $\cE_*$. Moreover, for each $b\in B$, we have that 
     \[\PP[b\in A_*]=\prod_{i=1}^\ell \PP[b\in (A_i+t_i)]=\prod_{i=1}^\ell \frac{|A_i|}{p}\geq \prod_{i=1}^\ell \delta_i.\]
     Therefore, by linearity of expectation, $\EE[|A_*|]\geq \left(\prod_{i=1}^\ell \delta_i\right)|B|$. Taking $B'$ to be some outcome of $A_*$ that achieves (at least) this expected size establishes the lower bound of the lemma. 
 \end{proof}

\subsection{Avoiding solutions to 3-variable equations} 
 The next lemma is  essentially due to Behrend. 

\begin{lem}[Behrend's construction \cite{behrend1946sets}] \label{app lem: behrend}
    Let $k\geq 2$ and let $E=E(x_1,\ldots,x_{k}, y)=x_1+\ldots+x_k-ky\in \cE^{k+1}_{k}(0)$. Then there exists some $c>0$ such that for all sufficiently large primes $p\in \NN$, we have that
    $r_E(\ZZ_p)\geq p^{1-c/{\sqrt{\log p}}}$.
\end{lem}
	\begin{proof}
Fixing $n:=\floor{(p-1)/k(k+1)}$, by Lemma \ref{app lem:mod or not}, it suffices to finds some subset $A\subseteq [n]$ such that $A$ has no solutions to $E$ in the integers. Now for any selection of integers $q,d \geq 2$ and $r$, we define the set
		\[
		S_r(d,q) := \left\lbrace a_1+a_2(k q-1)+\ldots+a_d(k q-1)^{d-1} : a_1,\ldots,a_d \in [0,q-1], a_1^2+\ldots+a_d^2 = r \right\rbrace.
		\]
	Note that for  $r>d(q-1)^2$, we have that $S_r(d,q)=\emptyset$. We claim that for any choice of parameters $q,d$ and $r$, the set $S_r(d,q)$ has no non-trivial solutions to $E$.
	Indeed, suppose there exist $a^{(1)},\ldots,a^{(k)},b\in S_r(d,q)$ such that $a^{(1)} + \ldots + a^{(k)} = k b$. 
	For any $a\in\NN$ with digits $a_1,\ldots,a_d$ in base $k q-1$, let $\|a\| := \sqrt{a_1^2+\ldots+a_d^2}$.
	We have that
		\begin{equation} \label{eq:sphere}
		\| a^{(1)} + \ldots + a^{(k)}\| = \| k b\| = k \sqrt{r} = \|a^{(1)}\| +\ldots+ \|a^{(k)}\|,
		\end{equation}
	where we used here that $b\in S_r(d,q)$ and so if $b_1,\ldots, b_d$ are the digits of $b$ in base $kq-1$, then $kb_1,\ldots, kb_d$ are the digits of the expansion of $kb$ in base $kq-1$. Now by \eqref{eq:sphere}, the $k$ vectors $(a^{(j)}_1,\ldots,a^{(j)}_d)\in \RR^d$, $j\in[k]$, must be pairwise linearly dependent as this is the only case in which the triangle inequality is  an equality. As all the vectors have the same modulus (by definition of $S_r(d,q)$), they are all in fact equal and hence $a^{(1)}=\ldots=a^{(k)}=b$. This shows that the solution $(x_1,\ldots,x_k,y)=(a^{(1)},\ldots, a^{(k)},b)$ is trivial.

	It remains to choose parameters $r$, $d$, $q$ such that $|S_r(d,q)|$ is as large as possible with $S_r(d,q) \subset [n]$. Note that for a fixed choice of $d$ and $q$,
		\[
		\sum_{r=0}^{d (q-1)^2} |S_r(d,q)| = q^d.
		\]
	By averaging we can pick $r = r(d,q)$ such that
		\[
		|S_r(d,q)| \geq \frac{q^d}{d (q-1)^2 + 1} \geq \frac{q^{d-2}}{d}
		\]
	With the choices $d := \floor{\sqrt{\log n}}$ and $q := \floor{\frac{n^{1/d}}{k}}$ we have both $S_r(d,q) \subseteq [n]$ and $|S_r(d,q)| \geq n^{1-c'/\sqrt{\log n}} \geq p^{1-c/\sqrt{\log p}}$ for  suitable constants $c',c>0$, as required. 
	\end{proof}

This has the following corollary. 

\begin{cor} \label{app cor: behrend}
    For any $k_{0}\in \NN$, there exists some $c_0>0$ such that for any sufficiently large prime $p\in \NN$ the following holds. Let   $\cE_{+}\subseteq \cE^4_{3k_0}(0)$ be some set of equations $E$ of the form \begin{equation} \label{eq: E + type}
    E(x_1,\ldots,x_4)=\alpha_1 x_1+\alpha_2 x_2+\alpha_3 x_3-(\alpha_1+\alpha_2+\alpha_3) x_4,\end{equation}
   with $\alpha_1,\alpha_2,\alpha_3\in \{0,1,\ldots, k_0\}$. Then we have that
   $r_{\cE_{+}}(\ZZ_p)\geq p^{1-c_0/{\sqrt{\log p}}}$. In particular, we have that $r_{\cE}(\ZZ_p)\geq p^{1-c_0/{\sqrt{\log p}}}$ with $\cE=\cE^3_{k_0}(0)$.  
\end{cor}
\begin{proof}
    For each $j\in \{1,\ldots, 3k_0\}$, let $\cF_{j}\subseteq \cE_{+}$ be the collection of equations \begin{equation} \label{eq:beh fixed}E(x_1,\ldots,x_4)=\alpha_1 x_1+\alpha_2 x_2+\alpha_3x_3-(\alpha_1+\alpha_2+\alpha_3)x_4\end{equation} such that $0\leq \alpha_1,\alpha_2,\alpha_3\leq k_0$ and $\alpha_1+\alpha_2+\alpha_3=j$.  Note that $\cE_{+}=\bigcup_{j=1}^{3k_0}\cF_{j}$. 
   By Lemma \ref{app lem: behrend}, there is some $c_j>0$ such that $r_{\cF_j}(\ZZ_p)\geq p^{1-c_j/\sqrt{\log p}}$. Indeed, any non-trivial solution $(a_1,\ldots,a_4)$ to an equation as in \eqref{eq:beh fixed}, gives a non-trivial solution to the equation $E(x'_1,\ldots,x'_{j},y)=x_1'+\ldots+x'_j-jy$ by fixing $x'_i=a_1$ for $i=1,\ldots,\alpha_1$, $x'_i=a_2$ for $i=\alpha_1+1,\ldots,\alpha_1+\alpha_2$, $x'_i=a_3$ for $i=\alpha_1+\alpha_2+1,\ldots,j$ and $y=a_4$. The corollary then follows by fixing $c_0=\sum_{j=1}^{3k_0}c_j$ and applying Lemma \ref{app lem:dilate to mix sets} with $B=\ZZ_p$ to avoid all non-trivial solutions to equations in $\cE_{+}=\bigcup_{j=0}^k\cF_{j}$. The `in particular' statement then follows because every equation $E(x_1,x_2,x_3)\in \cE^3_{k_0}(0)$ has the form \eqref{eq: E + type} with $\alpha_3=0$. 
\end{proof}

The proof of Theorem \ref{thm:behrend extended} follows quickly, recalling the definition of $r_{K,3}(\ZZ_p)$ from Definition \ref{def:K-fold}.

\begin{proof}[Proof of Theorem \ref{thm:behrend extended}]
  Let $c_0>0$ be the constant given by Corollary \ref{app cor: behrend} with input $k_0=K$.   Therefore $r_{\cE^3_K(0)}(\ZZ_p)\geq p^{1-c_0/{\sqrt{\log p}}}$. Let $B\subseteq \ZZ_p$ be a subset free of non-trivial solutions to equations in $\cE^3_K\setminus \cE^3_K(0)$ such that $|B|\geq p/72K^2$ as guaranteed by Lemma \ref{app lem:dense avoiding sum not zero}. Then applying Lemma \ref{app lem:dilate to mix sets}, we get a set $B'\subseteq B$ which is $(K,3)$-fold solution-free and of size 
  \[|B'|\geq \frac{p^{1-c_0/{\sqrt{\log p}}}}{72K^2}\geq p^{1-C/{\sqrt{\log p}}}, \] for an appropriate $C>0$, as required.
\end{proof}

\subsection{Avoiding solutions to four-variable equations}

In order to prove Theorem \ref{thm: sidon extended}, we need to find large sets avoiding non-trivial solutions to certain four-variable equations. We start with some lemmas looking at special cases of equations. Our first lemma builds on an approach of Ruzsa \cite[Theorem 7.3]{ruzsa_solving_1993}.

\begin{lem} \label{app lem: 4 var type 1}
    Let $k_1\in \NN$ and $\eps_1>0$. Further let $\cE_=\subseteq \cE^4_{2k_1}(0)$ be the collection of equations of the form 
    \[E(x_1,x_2,x_3,x_4)=\alpha_1x_1+\alpha_2 x_2-\alpha_1 x_3-\alpha_2 x_4,\]
    with $\alpha_1,\alpha_2\in \{1,\ldots, k_1\}$. Then for sufficiently large primes $p\in \NN$, we have $r_{\cE_=}(\ZZ_p)\geq p^{1/2-\eps_1}$. 
\end{lem}
\begin{proof}
    Fix $n=\floor{(p-1)/8k_1}$ and let $q\in \NN$ be some prime between $\sqrt{n}/8k_1$ and $\sqrt{n}/4k_1$. Further, fix some maximal set $B\subset \ZZ_q\setminus\{0\}$ such that $B$ has no non-trivial solutions to equations of the form 
    \begin{equation} \label{eq betas app}
   \beta_1x_1+\beta_2 x_2-(\beta_1+\beta_2)x_3,     
    \end{equation}
    with $\beta_1,\beta_2\in \{0,\ldots,2k_1\}$. By Corollary \ref{app cor: behrend}, there is some $c_0>0$ such that  $|B|\geq q^{1-c_0/\sqrt{\log q}}$.

    Next we fix \[A:=\{b+(2k_1+1)qb': b\in B, b'\in [q], b'=b^2 \mbox{ mod }q\}\subseteq [n].\]
    We will show that $A$ has no non-trivial solutions (in the integers) to equations in $\cE_=$ and so by Lemma \ref{app lem:mod or not}, we get that 
    \[r_{\cE_=}(\ZZ_p)\geq |A|=|B|\geq q^{1-c_0/\sqrt{\log q}}\geq p^{1/2-\eps_1},\]
    for $p$ sufficiently large, as required. So suppose there is some equation $E(x_1,x_2,x_3,x_4)=\alpha_1x+\alpha_2x_2-\alpha_1 x_3-\alpha_2x_4\in \cE_=$ with a non-trivial solution $(a_1,a_2,a_3,a_4)\in A^4$. Without loss of generality, we can assume that $\alpha_1\geq \alpha_2$. 
    Letting $a_i=b_i+(2k_1+1)qb_i'$ with $b_i\in B$ for $i\in [4]$, we get that 
    \begin{equation} \label{eq:++--=1}
        \alpha_1(b_1-b_3) +\alpha_2(b_2-b_4)+(2k_1+1)q\left(\alpha_1(b'_1-b'_3) +\alpha_2(b'_2-b'_4)\right)=0.
    \end{equation}
    Taking \eqref{eq:++--=1} modulo $(2k_1+1)q$, we get that $D=\alpha_1(b_1-b_3) +\alpha_2(b_2-b_4)$ is congruent to $0$ modulo $(2k_1+1)q$. However we also have that $|D|\leq 2k_1q$ and so we must in fact have that $D=0$ in the non-modular setting, that is \begin{equation} \label{eq:simple ++--}
    \alpha_1(b_1-b_3)=\alpha_2(b_4-b_2). \end{equation}  

    We can also now reduce \eqref{eq:++--=1} to 
    \begin{equation} \label{eq:++--=2}
        \alpha_1(b_1^2-b_3^2)=\alpha_2(b_4^2-b_2^2) \mbox{ mod } q. 
    \end{equation}
    If $b_1=b_3$ or $b_2=b_4$ then both equalities must hold and $(a_1,a_2,a_3,a_4)$ is a trivial solution. Therefore  we can divide the left and right hand side of \eqref{eq:++--=2} by $\alpha_1(b_1-b_3)=\alpha_2(b_4-b_2)\neq 0$ to obtain that  $b_1+b_3=b_4+b_2 \mbox{ mod }q$. Using this and adding $\alpha_1(b_1+b_3)$ to \eqref{eq:simple ++--}, we get that 
    \[2\alpha_1b_1=(\alpha_1+\alpha_2) b_4+(\alpha_1-\alpha_2)b_2 \mbox{ mod }q.\]
    As $b_2\neq b_4$, this gives a non-trivial solution to an equation as in \eqref{eq betas app} with $\beta_1=\alpha_1-\alpha_2\geq 0$, 
 and $\beta_2=\alpha_1+\alpha_2$, contradicting the definition of $B$. 
\end{proof}

Our next lemma tackles other 4 variable equations, building on ideas from \cite[Theorem 7.5]{ruzsa_solving_1993}. 

\begin{lem} \label{app lem: 4 var type 2}
     Let $k_2\in 2\NN$ be fixed. Then for all $\eps_2>0$, there exists $m\in \NN$ such that the following holds. Let $\cE_{\neq}\subseteq \cE^4_{m+k_2}(0)$ be the collection of equations of the form 
    \[E(x_1,x_2,x_3,x_4)=\alpha_1x_1+\alpha_2 x_2-\alpha_3 x_3-\alpha_4 x_4,\]
    with $\alpha_i\in \{m+1,m+3,\ldots,m+k_2-1\}$  for $i\in [4]$  such that $\{\alpha_1,\alpha_2\}\neq \{\alpha_3,\alpha_4\}$ and $\alpha_1+\alpha_2=\alpha_3+\alpha_4$.  Then for sufficiently large primes $p\in \NN$, we have $r_{\cE_{\neq}}(\ZZ_p)\geq p^{1-\eps_2}$. 
\end{lem}
\begin{proof}
Fix $\eps'=\eps_2/k_2$. For $d\in \{m+2,\ldots, m+k_2\}$, let $\cF_d\subseteq \cE_{\neq}$ be all equations as in the statement of the lemma such that $\min\{\alpha_1,\alpha_2,\alpha_3,\alpha_4\}=d-1$. Note that $\cE_{\neq}=\bigcup_{d=m+2}^{m+k_2}\cF_d$ (with the union here going over $d$ such that $d=m\mbox{ mod }2$) and so by Lemma \ref{app lem:dilate to mix sets} (applied with $B=\ZZ_p$), it suffices to show that for each $d\in \{m+2,\ldots, m+k_2\}$, we have that $r_{\cF_d}(\ZZ_p)\geq p^{1-\eps'}$. So let us fix some $d\in \{m+2,m+4,\ldots, m+k_2\}$ for the remainder of the proof.

    Fix $n=\floor{(p-1)/4(m+k_2)}$ and let $q\in \NN$ be some prime between $d/8k_2$ and $d/4k_2$. We assume that $m$ (and hence $d$ and  $q$) is sufficiently large so that there is a set $B\subseteq \ZZ_q\setminus \{0\}$ of size at least $q^{1-\eps'/4}$ with $B$ avoiding non-trivial solutions to equations in \begin{equation} \label{eq:E+ for 4 var}
\cE_+:=\{\beta_1x_1+\beta_2x_2+x_3-(\beta_1+\beta_2+1)x_4: \beta_1,\beta_2\in \{1,\ldots,k_2\}\},        
    \end{equation}
     using Corollary \ref{app cor: behrend}. Now we define $J:=\floor{\log_d(n)}-2$ and the set 
    \[A:=\left\{\sum_{j=0}^{J}b_j d^j:b_j\in B \mbox{ for }0\leq j \leq J \right\}\subseteq [n].\]
    We will show that $A\subseteq [n]$ has no non-trivial solutions (in the integers) to equations in $\cF_d$ and hence, by Lemma \ref{app lem:mod or not}, we get that 
   \[ r_{\cF_d}(\ZZ_p)\geq |B|^{J+1}\geq \left(q^{1-\tfrac{\eps'}{4}}\right)^{\log_d(n)-2}\geq \left(\frac{d}{8k_2}\right)^{\log_d(n)\left(1-\tfrac{2}{\log_d(n)}\right)\left(1-\tfrac{\eps'}{4}\right)}.\]
   Now $(d/8k_2)^{\log_d(n)}=n^{1-\tfrac{1}{\log_{8k_2}(d)}}\geq n^{1-\eps'/4}$  for $m$ (and hence $d$)  sufficiently large in terms of $k_2$. Therefore, we have that 
   \[r_{\cF_d}(\ZZ_p)\geq n^{\left(1-\tfrac{\eps'}{4}\right)^2\left(1-\tfrac{3}{\log_d(n)}\right)}\geq \left(\frac{p}{4(m+k_2)}\right)^{1-\tfrac{3\eps'}{4}}\geq p^{1-\eps'},\]
   for $p$ (and hence $n$) sufficiently large. 

   It remains to show that $A$ has no non-trivial solutions to equations in $\cF_d$. Suppose to the contrary that there is some $E(x_1,x_2,x_3,x_4)=\alpha_1x_1+\alpha_2x_2-\alpha_3x_3-\alpha_4x_4\in \cF_d$ with a non-trivial solution $(a_1,a_2,a_3,a_4)\in A^4$. Without loss of generality, we can assume that 
   \begin{equation} \label{eq: alpha train}
       d-1= \alpha_3 \mbox{ and } d+1 \leq\alpha_1\leq \alpha_2<\alpha_4\leq m+k_2,
   \end{equation}
   due to the fact that $\{\alpha_1,\alpha_2\}\neq \{\alpha_3,\alpha_4\}$, $\alpha_1+\alpha_2=\alpha_3+\alpha_4$ and $d-1=\min\{\alpha_1,\alpha_2,\alpha_3,\alpha_4\}$. For $i\in [4]$, let 
   $a_i=\sum_{j=0}^{J}b^{(i)}_{j}d^j$ with $b^{(i)}_j\in B$ for $0\leq j \leq J$. Let $j_*$ be the minimum $j$ with $0\leq j\leq J$ such that not all the $b^{(i)}_j$, $i\in [4]$ are equal. Note that such a $j_*$ must exist, otherwise all the $a_i$ are equal and the solution is trivial. Taking $E(a_1,a_2,a_3,a_4)$ modulo $d^{j_*+1}$, we get that \[\alpha_1 b^{(1)}_{j_*} d^{j_*} + \alpha_2 b^{(2)}_{j_*} d^{j_*} -\alpha_3 b^{(3)}_{j_*} d^{j_*}- \alpha_4 b^{(4)}_{j_*} d^{j_*} =0 \mbox{ mod }d^{j_*+1}.\]
Therefore considering $E(a_1,a_2,a_3,a_4)/d^{j_*}$ modulo $d$, we get that 
\[\gamma_1 b^{(1)}_{j_*}  + \gamma_2  d^{j_*} -\gamma_3 b^{(3)}_{j_*} -\gamma_4 b^{(4)}_{j_*}  =0 \mbox{ mod }d,\]
where $\gamma_i=\alpha_i \mbox{ mod }d$. Now note that $\gamma_3=-1$, $\gamma_i\geq 1$ for $i=1,2,4$ and $\gamma_1+\gamma_2-\gamma_3-\gamma_4=\alpha_1+\alpha_2-\alpha_3-\alpha_4=0$. Therefore modulo $d$ we have that  $(b^{(1)}_{j_*},b^{(2)}_{j_*}, b^{(3)}_{j_*},b^{(4)}_{j_*})\in B^4$ is a solution to the equation $E_*(x_1,x_2,x_3,x_4)=\gamma_1x_1+\gamma_2x_2+x_3-(\gamma_1+\gamma_2+1)x_4$.   Moreover, as $\gamma_1,\gamma_2\in \{1,\ldots,k_2\}$ we have that $|E_*(b^{(1)}_{j_*},b^{(2)}_{j_*}, b^{(3)}_{j_*},b^{(4)}_{j_*})|\leq (2k_2+1)q<d$ and so $E_*(b^{(1)}_{j_*},b^{(2)}_{j_*}, b^{(3)}_{j_*},b^{(4)}_{j_*})=0$ in the non-modular setting (and so also modulo $q$). Therefore, as $E_*\in \cE_+$ (see \eqref{eq:E+ for 4 var}), we have that the solution is trivial and all of the $b^{(1)}_{j_*}$ coincide. This contradicts the definition of $j_*$ and concludes the proof. 
   \end{proof} 

We are now in a position to prove Theorem \ref{thm: sidon extended}. 

\begin{proof}[Proof of Theorem \ref{thm: sidon extended}]
Fix $\eps_1=\eps_2=\eps/4$ and let $m$ be as output by Lemma \ref{app lem: 4 var type 2} with input $k_2=K$ and $\eps_2$. Let  $\cE':=\cE^4_{2(m+2K)}\setminus \cE^4_{2(m+2K)}(0)$ and  fix $\cE_+=\cE_{3k_0}^4(0)$ as in Corollary \ref{app cor: behrend} with $k_0=2(m+2K)$. Further, fix  $\cE_=\subseteq \cE^4_{2k_1}(0)$ as in Lemma \ref{app lem: 4 var type 1} with $k_1=2(m+K)$, and $\cE_{\neq}\subseteq \cE^4_{m+k_2}(0)$ as in Lemma \ref{app lem: 4 var type 2} with $m$ as defined and $k_2=K$. Finally, take 
$\cE:=\cE'\cup \cE_+\cup \cE_= \cup \cE_{\neq}$. Due to Lemmas \ref{app lem:dense avoiding sum not zero}, \ref{app lem:dilate to mix sets}, \ref{app lem: 4 var type 1} and \ref{app lem: 4 var type 2} as well as Corollary \ref{app cor: behrend}, we have that for $p$ sufficiently large, $r_\cE(\ZZ_p)\geq p^{1/2-\eps}$. 

It remains to verify that all of the equations for which we want to avoid non-trivial solutions in Theorem \ref{thm: sidon extended}, are contained in $\cE$. Firstly note that any set avoiding solutions to $\cE'$ and $\cE_+$ is certainly $(2(m+2K),3)$-fold solution-free. Moreover for any equation 
\[E(x_1,x_2,x_3,x_4)=\alpha_1x_1+\alpha_2x_2+\alpha_3x_3+\alpha_4x_4,\]
with $|\alpha_i|\in \{m+1,m+3,\ldots,m+K-1\}$ for $i\in [4]$ as in \eqref{eq:sidon extend}, we have that either $E\in \cE'$ or $\sum_{i\in [4]}\alpha_i=0$. In the latter case, let $t$ be the number of positive $\alpha_i$ with $i\in [4]$ and note that we can assume that $t\in \{2,3\}$ as any solution to the equation $-E$ also gives a solution to $E$. 
If $t=3$, then we have that $E\in \cE_+$. If $t=2$, then without loss of generality $\alpha_1,\alpha_2$ are positive and $\alpha_1+\alpha_2=-\alpha_3-\alpha_4$. 
If $\{\alpha_1,\alpha_2\}=\{-\alpha_3,-\alpha_4\}$, 
then $E\in \cE_{=}$, whilst if $\{\alpha_1,\alpha_2\}\neq\{-\alpha_3,-\alpha_4\}$ 
then  $E\in\cE_{\neq}$. This completes the proof. 
\end{proof}

\end{document}